\documentclass{article}
\usepackage[utf8]{inputenc} 
\usepackage{amsmath}
\usepackage{amsfonts}
\usepackage{mathrsfs}
\usepackage{amssymb}
\usepackage{amsthm}
\usepackage{bbm}
\usepackage{bm}
\usepackage{dsfont}
\usepackage{appendix}
\usepackage[normalem]{ulem}
\usepackage{enumerate}
\usepackage{xmpmulti} 
\usepackage[english]{babel} 
\usepackage[T1]{fontenc} 
\usepackage{overpic} 
\usepackage{verbatim} 
\usepackage{tikz} 
\numberwithin{equation}{section}
\usepackage{enumitem}
\usepackage{hyperref} 
\usepackage{mathtools}

\newcommand{\vertiii}[1]{{\left\vert\kern-0.25ex\left\vert\kern-0.25ex\left\vert #1 
    \right\vert\kern-0.25ex\right\vert\kern-0.25ex\right\vert}}

\newcommand*{\qtext}[1]{\quad\text{#1}\quad}

\newcommand*{\coleq}{\mathrel{\mathop:}=}

\newcommand{\epf}{{{\hfill $\Box$ \smallskip}}}

\newcommand{\R}{\mathbb{R}}
\newcommand{\N}{\mathbb{N}}

\newcommand{\Z}{\mathbb{Z}}
\newcommand{\Q}{\mathbb{Q}}

\newcommand{\tbf}{\mathbf{t}}

\newcommand{\sbf}{\mathbf{s}}

\newcommand{\ibf}{\mathbf{i}}

\newcommand{\jbf}{\mathbf{j}}

\newcommand{\Pp}{\mathbf{P}}

\newcommand{\E}{\mathbf{E}}
\newcommand{\Ic}{\mathcal{I}}

\newcommand{\AC}{\mathcal{A}}
\newcommand{\Asc}{\mathscr{A}}

\newcommand{\F}{\mathcal{F}}

\newcommand{\Sc}{\mathcal{S}}

\newcommand{\Tc}{\mathcal{T}}

\newcommand{\Dc}{\mathcal{D}}
\newcommand{\eps}{\epsilon}

\newcommand{\id}{\mathbbm{1}}

\newcommand{\zbf}{\mathbf{z}}
\newcommand{\nbf}{\mathbf{n}}

\newcommand{\Psf}{\mathbf{P}}

\newcommand\Item[1][]{%
  \ifx\relax#1\relax  \item \else \item[#1] \fi
  \abovedisplayskip=0pt\abovedisplayshortskip=0pt~\vspace*{-\baselineskip}}
  
\renewcommand{\emptyset}{\varnothing}
\newcommand{\Jtconst}{\nu}
\newcommand{\sublin}{\sigma}

\newcommand{\nuzero}{\nu_0}
\newcommand{\nuone}{\nu_1}

\newtheorem{theorem}{Theorem}[section]
\newtheorem{lemma}{Lemma}[section]
\newtheorem{claim}{Claim}[section]
\newtheorem{remark}{Remark}[section]

\newtheorem{proposition}{Proposition}[section]

\newtheorem{assumption}{Assumption}[section]

\DeclareFontFamily{U}{mathc}{}
\DeclareFontShape{U}{mathc}{m}{it}%
{<->s*[1.03] mathc10}{}

\DeclareMathAlphabet{\mathscr}{U}{mathc}{m}{it}

\title{A factorization formula for the partition function in the semi-discrete parabolic Anderson model}

\author{
Tobias Hurth \footnote{Department of Mathematics, University of Toronto, Toronto, Canada} 
\and Konstantin Khanin \footnote{Beijing Institute of Mathematical Sciences and Applications (BIMSA), Beijing, China} 
\footnote{Department of Mathematics, University of Toronto, Toronto, Canada} 
\and Beatriz Navarro Lameda \footnote{Department of Mathematics, University College London, London, United Kingdom}
}

\makeatletter
\@namedef{r@tocindent4}{0pt}
\@namedef{r@tocindent5}{0pt}
\makeatother

\makeatletter

\renewcommand\subsection{\@startsection{subsection}{2}{\z@}%
                                     {-3.25ex\@plus -1ex \@minus -.2ex}%
                                     {1.5ex \@plus .2ex}%
                                     {\normalfont\bfseries}}
\renewcommand\subsubsection{\@startsection{subsubsection}{3}{\z@}%
                                     {-3.25ex\@plus -1ex \@minus -.2ex}%
                                     {-1ex \@plus -3ex}%
                                     {\normalfont\bfseries}}

\makeatother

\makeatletter
\renewenvironment{proof}[1][\proofname]{\par
  \pushQED{\qed}%
  \normalfont \topsep6\p@\@plus6\p@\relax
  \trivlist
  \item[\hskip\labelsep
        \bfseries
    #1\@addpunct{.}]\ignorespaces
}{%
  \popQED\endtrivlist\@endpefalse
}
\makeatother

\makeatletter
\def\@tocline#1#2#3#4#5#6#7{\relax
  \ifnum #1>\c@tocdepth 
  \else
    \par \addpenalty\@secpenalty\addvspace{#2}%
    \begingroup \hyphenpenalty\@M
    \@ifempty{#4}{%
      \@tempdima\csname r@tocindent\number#1\endcsname\relax
    }{%
      \@tempdima#4\relax
    }%
    \parindent\z@ \leftskip#3\relax \advance\leftskip\@tempdima\relax
    \rightskip\@pnumwidth plus4em \parfillskip-\@pnumwidth
    #5\leavevmode\hskip-\@tempdima
      \ifcase #1
       \or\or \hskip 1.8em \or \hskip 4.5em \else \hskip 3em \fi%
      #6\nobreak\relax
    \hfill\hbox to\@pnumwidth{\@tocpagenum{#7}}\par
    \nobreak
    \endgroup
  \fi}
\makeatother

\date{} 

\begin{document}

\maketitle 

\begin{abstract}
We consider a continuous-time simple symmetric random walk on the integer lattice $\Z^d$ in dimension $d \geq 3$, subject to a random potential given by a field of two-sided Wiener processes. In the high-temperature regime, we prove the existence of the $L^2$- and almost sure limit of the partition function as time $t \to \pm \infty$, and show that these limiting partition functions are positive almost surely. Our main result is a factorization formula for the point-to-point partition function, which is shown to be valid up to any sub-ballistic scale. 

\end{abstract}




\section{{Introduction}}      \label{sec:introduction} 

Directed polymers in a random potential are a class of models in statistical mechanics where a stochastic process interacts with its space-time environment through a random potential~\cite{Vargas}. They are also known under the name \emph{Anderson polymer model}, in honor of P. W. Anderson's work on entrapment of electrons in crystals with impurities~\cite{anderson1958absence}. Statistical properties of the paths of the stochastic process are examined under a random Gibbs measure that depends on a parameter representing the temperature in the system. Intuitively, as the temperature increases, the influence of the random potential decreases, and typical paths under the Gibbs measure start to resemble typical paths under the random-walk measure. 

Directed polymers were first introduced in the 1980s in the context of ferromagnetism by Huse and Henley~\cite{PhysRevLett.54.2708}, who studied the phase boundary of the Ising model with random impurities. They then found additional applications in physics to a variety of phenomena from tearing sheets of paper~\cite{Kertez:1993aa}, to turbulence in electrically driven liquid crystals~\cite{takeuchi2010universal, zbMATH06062539}, to kinetic roughening of growing surfaces~\cite{Krug:1991aa}. The first mathematical treatments were undertaken in the late 1980s by Imbrie and Spencer~\cite{Imbrie1988}, and by Bolthausen~\cite{Bolthausen}. A modern treatment of directed polymers in a random potential can be found in the monograph~\cite{zbMATH06684616} by Comets. 

In polymer models, the spacetime is parameterized by $(x,t)$, where space and time are typically either continuous (i.e., $x \in \R^d$ for some dimension $d \geq 1$ and $t \in \R$) or discrete (i.e., $x \in \Z^d$ and $t \in \Z$). In this paper, however, we consider a semidiscrete spacetime $\Z^d \times \R$. By means of the Feynman--Kac formula, continuity in the time variable lets us apply our results on the polymer model to the stochastic heat equation on the integer lattice $\Z^d$, see~\cite{HKN_SHE}. And while continuity in the space variable is desirable from a physical point of view, it comes with technical challenges that we presently avoid by working on a discrete space. 

The two most typical choices for the underlying stochastic process are a Brownian motion in continuous space and a random walk in discrete space. In this paper, we consider a continuous-time simple symmetric random walk $\eta$ on $\Z^d$, i.e., a random walk where jumps from a site $x \in \Z^d$ may occur only to one of its $2d$ neighboring sites, with equal probabilities, and where the times between consecutive jumps are independent exponentially distributed random variables of rate $1$. In this setting, a directed polymer is the graph $(t, \eta(t))_t$ of the random walk $\eta$ parameterized by time $t$.  

In discrete space, the random potential is given by a family $\{F^x(t): \ x \in \Z^d\}$ of independent identically distributed stochastic processes. Several choices for the law of the random potential have been considered in the literature. The case of a Bernoulli potential ($F^x(t) = \pm 1$, each with probability $1/2$) was studied in discrete spacetime by Imbrie and Spencer~\cite{Imbrie1988}, Bolthausen~\cite{Bolthausen}, and Song and Zhou~\cite{zbMATH01279712}; and in a semidiscrete setting by Coyle~\cite{MR2693058, Coyle96}. In~\cite{CarmonaHu}, P. Carmona and Hu investigated a Gaussian potential in discrete spacetime, where each $F^x(t)$ is a standard normal random variable. As a generalization of a Bernoulli potential, one may consider distributions of bounded support~\cite{zbMATH01279712}. A Gaussian potential generalizes naturally to a spacetime field of random variables with finite exponential moments, as studied by Sinai~\cite{Sinai_95} and Vargas~\cite{Vargas}. In the present paper, for every $x \in \Z^d$, $(F^x(t))_{t \in \R} = (W^x_t)_{t \in \R}$ is a two-sided Brownian motion, and the Brownian motions $(W^x)_{x \in \Z^d}$ are assumed independent.  

The energy a path $\eta$ acquires in the random potential $\{F^x(t): \ x \in \Z^d\}$ over a time interval $[s, t]$ is given as the integral 
$$
\AC_s^t(\eta) \coleq \int_s^t F^{\eta(\tau)}(\tau) \ d \tau. 
$$
The statistical properties of the polymer model are then encoded in the point-to-point partition function $Z_{x,s}^{y,t}$ corresponding to the random-walk bridge between $(x,s)$ and $(y,t)$ for $t > s$. In our semidiscrete setting, the point-to-point partition function is given by  
$$
Z_{x,s}^{y,t} \coleq e^{-\frac{\beta^2}{2} (t-s)} \ p_{t-s}^{y-x} \  \E_{x,s}^{y,t} e^{\beta \AC_s^t}.  
$$
Here $\E_{x,s}^{y,t}$ is the expectation obtained by conditioning on random walks starting at $x$ at time $s$ and ending at $y$ at time $t$, and $p_{t-s}^{y-x}$ is the transition probability of the continuous-time simple symmetric random walk going from $x$ to $y$ in time $t-s$. The parameter $\beta > 0$ is called the inverse temperature. It describes the strength of the interaction between the directed polymer and the random potential. In the high-temperature regime, when $\beta$ is small, the interaction is weak. In the low-temperature regime, when $\beta$ is large, directed polymers are heavily influenced by the potential.  

Of particular interest is the longterm behavior of the polymer system. A directed polymer can be either \emph{localized} (the endpoint distribution has bounded variance with high probability) or \emph{diffusive} (the variance of the endpoint distribution grows linearly in time). It was shown in\cite{CarmonaHu, Comets} that localized behavior always takes place in spatial dimension $d=1$ and $d=2$, while in higher dimension $d \geq 3$ there is a transition from diffusive to localized behavior as the inverse temperature increases. The weak-disorder regime, where polymers are diffusive, has been studied since the late 1980s. For $d \geq 3$ and small $\beta$, diffusivity was established for polymer models in discrete spacetime by Imbrie and Spencer~\cite{Imbrie1988}, Bolthausen~\cite{Bolthausen}, and Sinai~\cite{Sinai_95}; in continuous spacetime by Kifer~\cite{Kifer} and Comets, Shiga, and Yoshida~\cite{zbMATH02072459}; and in a semidiscrete setting by Coyle~\cite{MR2693058, Coyle96}. 

In Sinai's approach to proving the diffusive behavior of polymers in the weak-disorder regime, the main tool is a factorization formula describing the asymptotic behavior of $Z_{x,s}^{y,t}$ as $t-s$ tends to $\infty$. In the semidiscrete setting, this formula reads 
\begin{equation}    \label{eq:factorization_form} 
Z_{x,s}^{y,t} = p_{t-s}^{y-x} (Z_{x,s}^{\infty} Z_{-\infty}^{y,t} + \delta_{x,s}^{y,t}). 
\end{equation} 
Here, $Z_{x,s}^{\infty}$ and $Z_{-\infty}^{y,t}$ are the limiting partition functions corresponding to the random walk starting from $(x,s)$ and evolving into a distant future, and to the random walk coming from a distant past and ending at $(y,t)$, respectively. A rigorous description is given in Section~\ref{sec:setting}. The error term $\delta_{x,s}^{y,t}$ tends to zero as $t-s \to \infty$, uniformly in $x$ and $y$, provided $\|y-x\|$ is not too large relative to $t-s$. The convergence of $\delta_{x,s}^{y,t}$ to zero for $y-x$ in the diffusive region (i.e., $\|y-x\| = O(\sqrt{t-s})$) was proved by Sinai in a discrete setting~\cite{Sinai_95} and later by Kifer in a continuous setting~\cite{Kifer}. In~\cite{HKNN}, in the discrete setting considered by Sinai, convergence to zero is shown to hold even in the sub-ballistic region, i.e., for $\|y-x\| = O((t-s)^{1-\eps})$, where $\eps > 0$ can be taken arbitrarily small. 

The main purpose of this paper is to prove that the convergence of $\delta_{x,s}^{y,t}$ to zero in the sub-ballistic region still holds if time is continuous instead of discrete (Theorem~\ref{thm:factorization}). In this sense, our paper can be seen as a companion paper to~\cite{HKNN}. While this result is not surprising, its proof is technically challenging, mainly because the continuous-time simple symmetric random walk may reach any point in $\Z^d$ on time intervals of finite length and thus interacts with much larger portions of the potential than a random walk in discrete time.  

Inspired by~\cite{Kifer}, we use the factorization formula in~\cite{HKN_SHE} to prove the following result for the semidiscrete stochastic heat equation with continuous time variable $t \in \R$ and discrete space variable $x \in \Z^d$: In spatial dimension $d \geq 3$ and for small $\beta > 0$, two solutions $u_1$ and $u_2$ to 
\begin{equation*}
\partial_t u(y,t)  = \Delta_y u(y,t) + \beta u(y,t) \dot{W}^y_t, \qquad y \in \Z^d, \ t > 0
\end{equation*}
with respective initial conditions $u_1(y,0) = f_1(y)$ and $u_2(y,0) = f_2(y)$ of subexponential growth satisfy 
$$ 
\biggl \lvert \frac{u_1(y,t)}{u_1(0,t)} - \frac{u_2(y,t)}{u_2(0,t)}  \biggr \rvert \xrightarrow[t \to \infty] {} 0  \text{ in probability}
$$ 
for every $y \in \Z^d$. The link between partition functions for the polymer model and the semidiscrete stochastic heat equation is provided by a Feynman--Kac formula established in~\cite{Carmona_Molchanov}, which implies in particular that the limiting partition function $u(y,t) \equiv Z_{-\infty}^{y,t}$ is a global solution.   

The rest of this paper is organized as follows:
Section~\ref{sec:setting} contains the setting and our main results. We show the existence of the limiting partition functions $Z_{x,s}^{\infty} = \lim_{t \to \infty} Z_{x,s}^t$ and $Z_{-\infty}^{y,t} = \lim_{s \to -\infty} Z_s^{y,t}$ from the factorization formula~\eqref{eq:factorization_form}, both in the sense of $L^2$-convergence and as almost sure limits (Theorem~\ref{thm:limiting_part_fun_exists}). 
In Section~\ref{sec:transition_prob}, we collect several estimates on transition probabilities for the simple symmetric random walk on $\Z^d$, both in discrete and in continuous time.
In Section~\ref{sec:D_sequences}, we formulate a key lemma (Lemma~\ref{lm:standard_arg}) that will allow us in many cases to deduce convergence statements in the continuous-time setting from convergence results in discrete time. Most of these discrete-time convergence results were already established in~\cite{HKNN} and are similar to results obtained in~\cite{Sinai_95}. 
Sections~\ref{sec:proof_lm_factorization} and~\ref{ssec:main_contribution_lemmas} are devoted to the proof of the factorization formula. 
In Section~\ref{sec:proof_thm_limiting_part_fun}  we obtain a rate of convergence to the limiting partition function. 
Finally, in Section~\ref{sec:proof_uniformity} we prove that the limiting partition function is positive almost surely.

\textbf{Notation:} Throughout this article the Euclidean norm and inner product in $\R^d$ are denoted by $\| \cdot \|$ and $\langle \cdot, \cdot \rangle$, respectively. The $1$-norm in $\R^d$ is denoted by $\| \cdot \|_1$. We denote the Euclidean ball of radius $r$ centred at $x$ by $B_r(x)$. For functions $A$ and $B$ of a continuous time variable $t > 0$ and potentially of several other variables, we write $A \lesssim B$ (or $A$ is \emph{dominated} by $B$) if there is a constant $c > 0$ such that $A(t, \cdot) \leq c B(t, \cdot)$ for all $t \geq c$. The constant $c$ may depend on the dimension $d$, the inverse temperature $\beta$, or on scaling parameters. However, it may not depend on any time or space variables such as $t, n, r,$ or $z$. We will simply write $a \equiv b$ to indicate that $a \equiv b \bmod 2$. Finally, in order to simplify notation, we will write $\sum_{\zbf}$ to indicate that we are summing over all $\zbf \in (\Z^d)^r$, where the value of $r$ is clear from the context.

\section*{Acknowledgments}

Part of this paper was written during two-week stays at Mathematisches 
For\-schungs\-zentrum Oberwolfach, in 2018, and Centre International de Rencontres Math\'ematiques (CIRM), in 2019, as part of their respective research in pairs programs. We thank both institutions for their kind hospitality. TH gratefully acknowledges support from the Einstein Foundation Berlin through grant IPF-2021-651 (2022--2024) and, until 2020, from the Swiss National Science Foundation through grant $200021-175728/1$. KK is grateful to acknowledge support from the Natural Sciences and Engineering Research Council of Canada through NSERC Discovery Grant RGPIN-2024-05804.

\section{Setting and Main Results}
\label{sec:setting}

The setting we consider is a continuous-time version of the one in~\cite{HKNN} and~\cite{Sinai_95}. Let $\eta = (\eta_t)_{t \in \R}$ be a continuous-time simple symmetric random walk on $\Z^d$, $d \geq 3$, starting at point $x \in \Z^d$ at time $s \in \R$, with corresponding probability measure $\Pp_{x,s}$ and corresponding expectation $\E_{x,s}$.  
We assume that the jump times of $\eta$ constitute a Poisson point process on the real line with Lebesgue intensity measure. Recall that the gaps between adjacent points of such a process form an i.i.d. sequence of exponential random variables with rate $1$. 
If observed over a time interval $[s,t)$, a sample path of $\eta$ is characterized by 
(i) the number $n_{s,t}$ of jumps that occur within $(s,t)$, 
(ii) a discrete-time path $\gamma = (\gamma_0, \gamma_1, \ldots, \gamma_{n_{s,t}})$ on $\Z^d$ such that $\|\gamma_j - \gamma_{j-1}\|_1= 1$ for $1 \leq j \leq n_{s,t}$, and 
(iii) the jump times $s < s_1 < \ldots < s_{n_{s,t}} <  t$. It is convenient to denote $s_0 := s$ and $s_{n_{s,t}+1} := t$, although we do not assume that $s$ and $t$ are jump times.
Instead of $n_{0,t}$, we will typically write $n_t$. The letter $\eta$ is also used to denote sample paths of the continuous-time simple symmetric random walk. 

For real numbers $t > s$ and $y \in \Z^d$, we denote the probability measure obtained from $\Psf_{x,s}$ by conditioning on $\{\eta_t = y\}$ by $\Psf_{x,s}^{y,t}$, and the corresponding expectation by $\E_{x,s}^{y,t}$.
We also define  
$$ 
	p_t^z \coleq \Pp_{0,0}(\eta_t = z)
\qtext{and}
	q_j^z \coleq \Pp (\gamma_j = z \mid \gamma_0 = 0)
$$
to be respectively the continuous-time and discrete-time transition probabilities for simple symmetric random walk on $\Z^d$. 

In this continuous-time setting, the random potential is given by a collection of independent two-sided Wiener processes $(W^x)_{x \in \Z^d}$. Let $\Omega$ denote the set of functions $\omega : \Z^d \times \R \to \R$ such that for every $x \in \Z^d$, $t \mapsto \omega(x,t)$ is continuous and satisfies $\omega(x,0) = 0$. Let $\F$ be the canonical $\sigma$-field on $\Omega$, and let $Q$ be the probability measure on $(\Omega, \F)$ under which $W^x_t(\omega) \coleq \omega(x,t)$, $x \in \Z^d$, are independent two-sided Wiener processes. To ensure measurability for sets contained in sets of measure zero, we work with the completion of $(\Omega, \F, Q)$, which we still denote by $(\Omega, \F, Q)$. The expectation corresponding to $Q$ is denoted by $\langle \:\cdot\: \rangle$. For every $s \in \R$, $Q$ is invariant under the Wiener shift 
$$
\theta_s \omega(x,t) := \omega(x,t+s) - \omega(x,s), \quad x \in \Z^d, \ t \in \R,  
$$
i.e., $Q(\theta_s A) = Q(A)$ for every $A \in \F$. 

To a random walk sample path $\eta$ over a time interval $[s,t)$, we assign the random action 
\begin{equation}    \label{eq:def_action} 
\AC_s^t (\eta, \omega) \coleq \sum_{j=0}^{n_{s,t}} \left(\omega(\gamma_j, s_{j+1}) - \omega(\gamma_j, s_j) \right).  
\end{equation}  
For real numbers $s < t$, $x, y \in \Z^d$, and inverse temperature $\beta > 0$, we define the random normalized point-to-point partition function 
$$
Z_{x,s}^{y,t}(\omega) \coleq e^{-\frac{\beta^2}{2} (t-s)} p_{t-s}^{y-x} \E_{x,s}^{y,t} e^{\beta \AC_s^t(\cdot, \omega)}.  
$$
We also define the random normalized point-to-line partition functions 
$$ 
Z_{x,s}^t(\omega) \coleq \sum_{y \in \Z^d} Z_{x,s}^{y,t}(\omega), \quad \text{and} \quad Z_s^{y,t}(\omega) \coleq \sum_{x \in \Z^d} Z_{x,s}^{y,t}(\omega). 
$$
Since $e^{-\frac{\beta^2}{2} (t-s)} \langle e^{\beta \AC_s^t(\eta, \cdot)} \rangle = 1$ for every $\eta$, we have $\langle Z_{x,s}^t \rangle = \langle Z_s^{y,t} \rangle = 1$. 
Since the law for the increments of the Wiener processes $(W^x)_{x \in \Z^d}$ is stationary in space and time, the law of $(Z_{x,s}^{s+t})_{t \geq 0}$ does not depend on $x$ or $s$. 
Besides, $(Z_{x,s}^{s+\tau})_{\tau \geq 0}$ and $(Z_{t-\tau}^{y,t})_{\tau \geq 0}$ have the same law because of time-reversibility of $\eta$. \\ 

Given $z \in \Z^d$ and $s, t \in \R$ such that $s < t$, define
\begin{equation}  \label{eq:h_definition}
	h(z; s,t) 
		\coleq 
			e^{-\frac{\beta^2}{2} (t-s)} e^{\beta (W^z_t - W^z_s)} - 1
		=
\dfrac{e^{\beta (W^z_t - W^z_s)} - e^{\frac{\beta^2}{2} (t-s)}}{e^{\frac{\beta^2}{2} (t-s)}}.
\end{equation}
Notice that $\langle h(z;s,t) \rangle = 0$ and $\langle h(z;s,t)^2 \rangle = e^{\beta^2 (t-s)}-1$.
Moreover, $h(z;s,t)$ and $h(z';s',t')$ are independent if $z \neq z'$ or if $(s,t) \cap (s',t') = \emptyset$.

For $n \in \N_0$ and $1 \leq r \leq n+1$, let 
$$
I(n,r) \coleq \{\ibf = (i_1, \ldots, i_r) \in \N_0^r : \ 0 \leq i_1 < \ldots < i_r \leq n\}. 
$$
The point-to-point partition function $Z_{x,s}^{y,t}$ admits the expansion 
$$
	Z_{x,s}^{y,t}
		= p_{t-s}^{y-x} + \sum_{n=0}^{\infty} \sum_{r=1}^{n+1} \sum_{\ibf \in I(n,r), \zbf} q_{i_1}^{z_1 - x} \ldots q_{n - i_r}^{y - z_r} \E_{x,s} \biggl[ \prod_{j=1}^r h(z_j; s_{i_j}, s_{i_j +1}) \id_{n_{s,t} = n} \biggr], 
$$
where one should recall from Section~\ref{sec:introduction} that $\sum_{\zbf}$ means summation over all $\zbf \in (\Z^d)^r$ and where 
$s_{n+1} = t$ if $n = n_{s,t}$. Similarly, 
$$
Z_s^{y,t} = 1 + \sum_{n=0}^{\infty} \sum_{r=1}^{n+1} \sum_{\ibf \in I(n,r), \zbf}  q_{i_2 -i_1}^{z_2 - z_1} \ldots q_{n - i_r}^{y - z_r} \E_{x,s} \biggl[ \prod_{j=1}^r h(z_j; s_{i_j}, s_{i_j +1}) \id_{n_{s,t} = n} \biggr]
$$
and 
$$
Z_{x,s}^{t} =  1 + \sum_{n=0}^{\infty} \sum_{r=1}^{n+1} \sum_{\ibf \in I(n,r),  \zbf} q_{i_1}^{z_1 - x} \ldots q_{i_r - i_{r-1}}^{z_r - z_{r-1}} \E_{x,s} \biggl[ \prod_{j=1}^r h(z_j; s_{i_j}, s_{i_j +1}) \id_{n_{s,t} = n} \biggr].
$$
Since $\E_{x,s}[\prod_{j=1}^r h(z_j; s_{i_j}, s_{i_j + 1}) \id_{n_{s,t} = n}]$ does not depend on $x$, the choice of $x$ in the expansion for $Z_s^{y,t}$ does not matter. 

\subsection{$L^2$ convergence to limiting partition functions.}
\label{ssec:L2_convergence}

\bigskip


\noindent Exactly as in~\cite{HKNN} and going back to~\cite{Sinai_95}, define 
\begin{equation}    \label{eq:def_alpha} 
\alpha_d \coleq \sum_{n=1}^{\infty} \sum_{z \in \Z^d} \left(q_n^z \right)^2
\end{equation} 
which is finite because $d \geq 3$.
We also define 
\begin{equation}\label{eq:def_lambda} 
	\lambda \coleq \E \left[e^{\beta^2 \tau} - 1 \right] = \frac{\beta^2}{1-\beta^2}
\end{equation} 
for $\tau$ an exponential random variable with rate $1$. 
Observe that 
$$
\alpha_d \lambda < 1 \quad \Longleftrightarrow \quad \beta < 1/\sqrt{1+\alpha_d}.
$$ 
The following theorem is the continuous-time version of Theorem 2.1 in~\cite{HKNN}. It corresponds to Theorem 1 in~\cite{Sinai_95}.

\begin{theorem}     \label{thm:limiting_part_fun_exists}
Let $\beta$ be so small that $\alpha_d \lambda  < 1$ and let $(x,s) \in \Z^d \times \R$. As $t \to \infty$, $Z_{x,s}^t$ converges in $L^2(Q)$ to a limiting partition function $Z_{x,s}^{\infty}$. 
\end{theorem} 


Due to symmetry, 
$$
Z_{-\infty}^{y,t}  := \lim_{s \to -\infty}  Z_s^{y,t}  
$$
also exists in the sense of $L^2$ convergence.  
In addition, as $(Z_{x,s}^t)_{t \geq s}$ is a martingale~\cite{Bolthausen}, convergence to the limiting partition functions holds $Q$-almost surely by the martingale convergence theorem. 

\begin{proof}[Proof of Theorem~\ref{thm:limiting_part_fun_exists}] We follow the approach in~\cite{Sinai_95}, i.e., we rely on the expansion of $Z_{x,s}^t$ and its orthogonality structure. 
Since $h(z;s,t)$ and $h(z';s',t')$ are independent if $z \neq z'$ or if $(s,t) \cap (s',t') = \emptyset$, and since $\langle h(z;s,t) \rangle = 0$, we have with Jensen's inequality and Fubini's theorem that $\langle (Z^t_{x,s})^2 \rangle$ is bounded from above by 
\begin{align}      \label{eq:Sinai_3}
2 + 2 \sum_{r=1}^{\infty} & \sum_{\substack{0 \leq i_1 < \ldots < i_r, \\ z_1, \ldots, z_r \in \Z^d}} \notag \\
&\left(q_{i_1}^{z_1-x} \right)^2 \ldots \left(q_{i_r - i_{r-1}}^{z_r - z_{r-1}} \right)^2 \E_{x,s} \biggl[ \biggl \langle \prod_{j=1}^r h(z_j; s_{i_j}, s_{i_j +1})^2 \biggr \rangle \biggr]. 
\end{align}
Since $\langle h(z;s,t)^2 \rangle = e^{\beta^2 (t-s)} - 1$, and since $s_{i_j +1} - s_{i_j}$ is an exponentially distributed random variable with rate $1$, we find
\begin{equation}\label{eq:eq_lambda}
\E_{x,s} \biggl[ \biggl \langle \prod_{j=1}^r h(z_j; s_{i_j}, s_{i_j +1})^2 \biggr \rangle \biggr] 
	= \E_{x,s} \biggl[ \prod_{j=1}^r \left(e^{\beta^2 (s_{i_j +1} - s_{i_j})} - 1 \right) \biggr]
	= \lambda^r. 
\end{equation}
Hence, the expression in~\eqref{eq:Sinai_3} is finite provided that 
$$ 
\sum_{r=1}^{\infty} (\alpha_d \lambda)^r < \infty,
$$ 
which holds for $\alpha_d \lambda < 1$. The asserted $L^2$ convergence then follows from the martingale convergence theorem. 
\end{proof}

The following theorem is the continuous-time version of Theorem 2.2 in \\~\cite{HKNN}. It provides a rate of convergence to the limiting partition function, which is needed in order to prove the factorization formula (Theorem~\ref{thm:factorization}).

\begin{theorem}     \label{thm:limiting_part_fun}
For $\beta$ sufficiently small, there is $\theta \in (0, \min\{\tfrac{d}{2}-1,-\ln(\alpha_d \lambda)\})$ such that 
$$
\lim_{t \to \infty} (t-s)^{\theta} \left \langle \left(Z_{x,s}^t - Z_{x,s}^{\infty} \right)^2 \right \rangle = 0.
$$
\end{theorem} 

The proof of Theorem~\ref{thm:limiting_part_fun} is given in Section~\ref{sec:proof_thm_limiting_part_fun}. \\

\subsection{Almost sure convergence and positivity} 

As pointed out in the preceding subsection, for any $(x,s) \in \Z^d \times \R$, $Z_{x,s}^t$ converges almost surely to $Z_{x,s}^{\infty}$ as $t \to \infty$. A stronger statement holds: Almost surely, $Z_{x,s}^t$ converges to $Z_{x,s}^{\infty}$ for every $(x,s) \in \Z^d \times \R$, i.e., there is a set of measure $1$ that works for all $(x,s)$ at once. Moreover, the limiting partition functions $(Z_{x,s}^{\infty})_{x \in \Z^d, s \in \R}$ are strictly positive with probability $1$. 

\begin{theorem}     \label{thm:positivity}
Let $\beta$ be so small that $\alpha_d \lambda < 1$. Then there is a set $\Omega^{\lim}_+ \in \F$ with $Q(\Omega^{\lim}_+) = 1$ such that for all $(x,s), (y,t) \in \Z^d \times \R$ and all $\omega \in \Omega^{\lim}_+$, the limiting partition functions $Z_{x,s}^{\infty}(\omega)$ and $Z_{-\infty}^{y,t}(\omega)$ exist as pointwise limits $\lim_{T \to \infty} Z_{x,s}^T(\omega)$ and $\lim_{S \to -\infty} Z_S^{y,t}(\omega)$, respectively, and are positive. 
\end{theorem} 

The proof of Theorem~\ref{thm:positivity} can be found in Section~\ref{sec:proof_uniformity}. 
In discrete time, positivity of the limiting partition functions in the regime of weak disorder was shown in~\cite[Theorem 3]{Sinai_95}. 
While positivity of the limiting partition functions is not used in the proof of the factorization formula, it does play an important role in showing uniqueness of global solutions to the semi-discrete stochastic heat equation~\cite{HKN_SHE}.


\subsection{Factorization formula} 

Below we state the main result of this article, a factorization formula for the partition function $Z_{x,s}^{y,t}$. 

\begin{theorem}     \label{thm:factorization}
For $\beta$ sufficiently small, the following holds: For any $\sublin \in (0,1)$ there exists $\theta = \theta(\sublin) > 0$, independent of $\beta$, such that for all $x,y\in\Z^d$ and $s < t$ with $\| x - y \| < (t - s)^\sublin$, the partition function $Z_{x,s}^{y,t}$ has the representation 
\begin{equation}     \label{eq:factor_formula}
Z_{x,s}^{y,t} = p_{t-s}^{y-x} \left(Z_{x,s}^{\infty} Z_{-\infty}^{y,t} + \delta_{x,s}^{y,t} \right),
\end{equation} 
where the error term $\delta_{x,s}^{y,t}$ defined by the formula above satisfies 
\begin{equation}   \label{eq:convergence_error}
\lim_{(t-s) \to \infty} (t-s)^{\theta} \sup_{x, y \in \Z^d: \|x-y\| < (t-s)^{\sublin}} \langle \lvert \delta_{x,s}^{y,t} \rvert \rangle = 0.
\end{equation}
\end{theorem}

Theorem~\ref{thm:factorization} is proved in Section~\ref{sec:proof_lm_factorization}. 
We conjecture the conclusion of Theorem~\ref{thm:factorization} to hold for all $\beta$ so small that $\alpha_d \lambda < 1$, i.e., for all $\beta < 1/(\sqrt{1+\alpha_d})$. This region of $\beta$ corresponds to the so-called $L^2$-regime, where the $L^2$-series is convergent (see Theorem~\ref{thm:limiting_part_fun_exists}). 
We should mention that recently St. Junk extended the factorization formula in the discrete setting to the whole of the weak disorder regime (\cite{St_Junk}). This result seems to suggest that in the semidiscrete setting the factorization formula should also hold for all $\beta$ up to the point of transition from weak disorder to strong disorder. 
In~\cite{HKN_SHE}, Theorem~\ref{thm:factorization} is a key component in the proof of uniqueness for global solutions to the semidiscrete stochastic heat equation.

\subsection{Lower tail estimate} 

The following statement is a lower tail estimate for the law of the partition function $Z_0^{y,t}$. In the case of discrete space-time, such estimates have been obtained by P. Carmona and Hu in~\cite[Theorem~1.5]{CarmonaHu} using concentration of measure arguments for discrete directed polymers in Gaussian environments that originated in Talagrand's work on spin glasses (see~\cite{zbMATH01129463,Talagrand_11}). See also~\cite[Theorem 1(a)]{zbMATH06333766}. 

\begin{theorem}       \label{thm:continuous_Talagrand}
For $\beta$ sufficiently small, there exists a constant $c > 0$ such that 
\begin{equation*}
Q \left(Z_0^{y,t} < e^{-u} \right) < c e^{-u^2 / c}, \qquad \forall t, u > 0. 
\end{equation*}
\end{theorem}

Theorem~\ref{thm:continuous_Talagrand} is proved in Section~\ref{sec:proof_continuous_Talagrand}. The Gaussianity of the noise is important here. While its main purpose is to let us prove an attraction result for the global stationary solution to the stochastic heat equation in~\cite{HKN_SHE}, Theorem~\ref{thm:continuous_Talagrand} is interesting in its own right because it implies that the limiting partition function $Z_{0,0}^\infty$ admits negative moments of any order.

\section{Transition Probabilities for the Simple Symmetric Random Walk}
\label{sec:transition_prob}
In this section we collect several estimates on transition probabilities for the simple symmetric random walk on $\Z^d$, both in discrete and in continuous time. Section~\ref{ssec:transition_prob_proofs} of the appendix is devoted to the proofs of these estimates.

\begin{lemma}    \label{lm:lclt_con_time} 
For any $y \in \Z^d$ such that $\| y \| \leq \frac{t}{2 \sqrt{d}}$, we have for the transition probability of the continuous-time simple symmetric random walk on $\Z^d$ 
\begin{equation}    \label{eq:lclt_con_time} 
p_t^y = \left(\tfrac{d}{2 \pi t} \right)^{\frac{d}{2}} 
		\exp \big( -\tfrac{d}{2t} \| y \|^2 \big) 
		\exp \left( O \left(  \frac{1}{\sqrt{t}} + \frac{\| y \|^3}{t^2} \right) \right).
\end{equation}
Therefore, for any $\sublin \in (0,1)$ and $t$ sufficiently large,
\begin{equation}\label{ineq:lclt_con_time}
	\frac{1}{p^y_t} \leq e^{t^{\sublin}}, \qquad \forall y \in \Z^d: \|y\| \leq t^{\sublin}. 
\end{equation}
\end{lemma}






The following lemma is Lemma~3.4 in~\cite{HKNN}. We include it for convenience. 

\begin{lemma}    \label{lm:ratio_estimate}
There are constants $\rho, c > 0$ such that for any $n, n' \in \N$ and for any $z, z' \in \Z^d$ with $\|z\| \leq \rho n$ and $\|z\|_1 \equiv n$, we have 
$$ 
\frac{q_{n'}^{z'}}{q_n^z} \leq  \left(1+O(n^{-\frac{2}{5}}) \right) \exp \left(c \left(\frac{\|z\|}{n} \left(\|z-z'\| + \lvert n' - n \rvert \right) + \ln(n) \frac{\lvert n-n' \rvert}{n} \right) \right). 
$$
\end{lemma}

For $y \in \Z^d$ and $n \in \N_0$, set 
\begin{equation}    \label{eq:def_iota} 
\iota(y,n) = \begin{cases} 
                       n, & \quad \|y\|_1 \equiv n, \\ 
                       n+1, & \quad \|y\|_1 \not\equiv n.  
                       \end{cases}
\end{equation}         
Fix $\Jtconst \in (\tfrac{1}{2},1)$ and set  
\begin{equation} \label{eq:Jt_def}
	J(t) \coleq \left\{n \in \N: \left \lvert \frac{n}{t} - 1 \right \rvert < 1 - \Jtconst \right\} = \{n \in \N : \Jtconst t < n < (2 - \Jtconst)t\}. 
\end{equation}
%

\begin{lemma}     \label{lm:q_iota} 
Let $\sublin \in (\tfrac{3}{4},1)$ and $\xi_1 \in (0, 1 - \sublin)$. There are constants $T, c > 0$ such that the following holds: For any $t \geq T$, $y \in \Z^d$ such that $\|y\| \leq t^{\sublin}$, $m \in J(2 t^{\xi_1})$, and $l \in J(t-2t^{\xi_1})$, we have 
$$ 
q^y_{m+l} \leq c q^y_{\iota(y,l)}. 
$$
\end{lemma}

For $t > 0$, $y \in \Z^d$, and $x \in \Z^d$ such that $\|x\| \leq t^{\sublin}$, let 
$$ 
\Dc_t(y,x) \coleq \frac{\sum_{n \notin J(t)} e^{-t} \frac{t^n}{n!} q_n^{y-x}}{\sum_{n \in J(t)} e^{-t} \frac{t^n}{n!} q_n^{y-x}}. 
$$ 
%
\begin{lemma}     \label{lm:uniform_D_conv}
Let $y \in \Z^d$. For any $\sublin \in (\tfrac{3}{4}, 1)$, we have 
$$ 
\lim_{t \to \infty} \sup_{\|x\| \leq t^{\sublin}} \Dc_t(y,x) = 0. 
$$ 
\end{lemma}

\begin{lemma}     \label{lm:p_ratio}
Let $\xi \in (0,1)$ and $\sublin \in (\tfrac{3}{4},1)$. There are constants $T, c > 0$ such that for any $t \geq T$ and $y \in \Z^d$ with $\|y\| \leq t^{\sublin}$, we have
$$ 
\frac{p^y_{t-2t^{\xi}}}{p^y_t} \leq c e^{\beta^2 t^{\xi}}.  
$$ 
\end{lemma}

\section{Reduction to Discrete Time} 
\label{sec:D_sequences}
In this section, we formulate a key lemma (Lemma~\ref{lm:standard_arg}) that will allow us in many cases to deduce convergence statements in the continuous-time setting from convergence results in discrete time. The latter have typically been established in~\cite{HKNN}. \\

We begin by formulating a framework that is general enough to treat a variety of expansions arising in calculations with partition functions. Let $(R_n)_{n \in \N_0}$ be a family of sets such that $R_n \subset \{1, \ldots, n+1\}$ for every $n \in \N_0$.
Let $(I_{n,r})_{n \in \N_0, r \in R_n}$ be a collection of index sets such that
\begin{equation*}
	I_{n,r}
		\subset \{\ibf = (i_1, \ldots, i_r) \in \N_0^r ~:~  0 \leq i_1 < \ldots < i_r \leq n\}
\end{equation*}
for every $n \in \N_0$ and $r \in R_n$.
In addition, let $\mathscr{q}_{n,r}^y (\ibf, \zbf)$ be a collection of nonnegative real numbers indexed by $n \in \N_0, r \in R_n, y \in \Z^d, \ibf \in I_{n,r}, \zbf = (z_1, \ldots, z_r) \in (\Z^d)^r$, which satisfy the following finiteness condition: for every $n \in \N_0, r \in R_n, y \in \Z^d, \ibf \in I_{n,r}$,
\begin{equation}                   \label{eq:finite_cond_q}
\sum_{\zbf} \mathscr{q}_{n,r}^y(\ibf, \zbf) \leq 2. 
\end{equation}
Notice that~\eqref{eq:finite_cond_q} implies 
\begin{equation}
\label{190921123823}
	\sum_{\zbf} 
		\mathscr{q}_{n,r}^y(\ibf, \zbf)^2 
		\leq 4, \qquad \forall  n \in \N_0, \ r \in R_n, \ y \in \Z^d,  \ \ibf \in I_{n,r}. 
\end{equation}

\noindent For a fixed realization of $\eta$, for $y \in \Z^d$ and $s, t \in \R$ such that $s < t$, define
$$
	\Tc(y;s,t) \coleq
		\sum_{r \in R_{n_{s,t}}} 
		\sum_{\ibf \in I_{n_{s,t},r}, \zbf} 
			\mathscr{q}_{n_{s,t},r}^y(\ibf, \zbf) 
			\prod_{j=1}^r h(z_j; s_{i_j}, s_{i_j +1}), 
$$
where one should recall that $s_0 \coleq s$ and $s_{n_{s,t}+1} \coleq t$. Note that 
\begin{equation}     \label{eq:T_square_bracket}  
\langle \Tc(y;s,t)^2 \rangle = \sum_{r \in R_{n_{s,t}}} \sum_{\ibf \in I_{n_{s,t}, r}, \zbf} \mathscr{q}_{n_{s,t}, r}(\ibf, \zbf)^2 \prod_{j=1}^r \left(e^{\beta^2 (s_{i_j + 1} - s_{i_j})} - 1 \right) 
\end{equation}  
by the properties of $h$. For any $\vartheta > 0$, we define a sequence of functions $(D^{\vartheta}_n)_{n \in \N_0}$ on the lattice $\Z^d$ by  
\begin{equation*}
	D^{\vartheta}_n(y)
		\coleq 	\sum_{r \in R_n} \vartheta^r
			\sum_{\ibf \in I_{n,r}, \zbf} 
				\mathscr{q}_{n,r}^y(\ibf, \zbf)^2, \qquad y \in \Z^d.   
\end{equation*}
We call $(D^{\vartheta}_n)_{n \in \N_0}$ the \emph{$D$-sequence} associated with $\Tc$.
As a consequence of Condition~\eqref{eq:finite_cond_q}, every $D_n^{\vartheta}(y)$ is finite. 
Recall that $n_t$ is a shorthand for $n_{0,t}$. In addition, we use the shorthand $\E \coleq \E_{0,0}$. 

\begin{lemma}[Key Lemma]
\label{lm:standard_arg} 
For every $\vartheta > 0$, there exists $\beta_* > 0$ such that the following holds: Assume that there are $\tilde \sublin \in (\tfrac{3}{4}, 1)$ and $\theta > 0$ such that 
\begin{equation*}
\lim_{n \to \infty} n^{\theta} \sup_{\|y\| \leq n^{\tilde \sublin}} D^{\vartheta}_n(y) = 0.  
\end{equation*}
Then, if $\beta \leq \beta_*$, one has 
\begin{equation*}
\lim_{t \to \infty} t^{\theta /2} \sup_{\|y\| \leq t^{\sublin}} \frac{1}{p_t^y} \E \left[q_{n_{t}}^y \left \langle \left \lvert \Tc(y; 0,t) \right \rvert \right \rangle \right] = 0
\end{equation*}
for every $\sigma \in (0, \tilde \sigma)$. 
\end{lemma}

\noindent We prove Lemma~\ref{lm:standard_arg} in Subsection~\ref{ssec:key_lemma_proof}.

\begin{remark}    \rm 
We do not attempt to find the supremum over those $\beta_* > 0$ for which the conclusion of Lemma~\ref{lm:standard_arg} holds for a given $\vartheta$. 
\end{remark}

\begin{remark}  \rm     \label{rm:standard_arg}
The proof of Lemma~\ref{lm:standard_arg} implies the following: If $\Tc$, $D_n^{\vartheta}$ do not depend on $y$ and if $\lim_{n \to \infty} n^{\theta} D_n^{\vartheta} = 0$ for some $\theta, \vartheta > 0$, then, for $\beta$ sufficiently small (depending on $\vartheta$),  
\begin{equation*}
\lim_{t \to \infty} t^{\theta} \E \left \langle \left( \Tc(0,t) \right)^2 \right \rangle = 0.
\end{equation*}
\end{remark}

\subsection{Some preliminary estimates}
\label{ssec:key_lemma_estimates}

Before proving Lemma~\ref{lm:standard_arg}, we collect several auxiliary statements whose proofs can be found in Section~\ref{ssec:key_lemma_estimates_proofs} of the appendix. For any $t > 0$, $l \in \N_0$, and $1 \leq r \leq l+1$, we introduce the following notation:
\begin{equation}   \label{eq:def_hat_A}
A(t,l,r) \coleq \E \biggl[ \prod_{j=1}^{r-2} \left(e^{\beta^2 t_j} - 1 \right) e^{\beta^2 (t_{r-1} + t_r)} \bigg\vert n_t = l \biggr], 
\end{equation} 
where $t_0 \coleq 0$, $t_j \coleq s_j - s_{j-1}$ for $1 \leq j \leq l$, and $t_{l+1} \coleq t-s_l$. On account of~\eqref{eq:T_square_bracket}, for any $l \in \N_0$, 
\begin{equation}    \label{eq:T_square_bracket_A_est}
\E \left[ \left \langle \Tc(y;0,t)^2 \right \rangle \big\vert n_t = l \right] \leq \sum_{r \in R_l} \sum_{\ibf \in I_{l,r}, \zbf} \mathscr{q}_{l,r}^y(\ibf,\zbf)^2 A(t,l,r). 
\end{equation} 




\begin{lemma}\label{lm:expected_value}
If $\beta < 1$, then   
\begin{equation} \label{eq:A_bound}
	A(t,l,r) 
		\lesssim
		(l+1)^2 e^{\beta^2 t} \beta^{2r} \frac{t^r}{(l+1) \ldots (l+r)}. 
\end{equation} 
\end{lemma}

\begin{lemma}     \label{lm:precise_I_bounds} 
Let $\nu \in (\tfrac{1}{2}, 1)$, $\delta \in (0,1)$, $\nuzero \in (0, \tfrac{\delta}{2})$, $\kappa \in (0, \delta)$, and $\hat \kappa \in (0, \nu)$. Then  
\begin{flalign}
\label{190921111055}
	\frac{t^{l+r}}{(l+r)!} \lesssim& \ t^{-1/2} \left(\frac{e}{\kappa} \right)^{\delta t} \qquad \text{for} \ 0 \leq l \leq \nuzero t, \ 1 \leq r \leq l+1;
\\
\label{190921111101}
	\frac{t^l}{l!} \lesssim& \ t^{-1/2} \left(\frac{e}{\hat \kappa} \right)^{\nu t} \qquad \text{for} \ \nuzero t < l \leq \nu t. 
\end{flalign}
\end{lemma}


\begin{lemma}\label{lm:A_estimate}
Let $\nu \in (\tfrac{1}{2}, 1)$, $\nuone \in (\Jtconst^{-1} - 1,1)$, and let 
\begin{equation}    \label{eq:def_psi} 
\psi(\beta) \coleq \frac{\beta^2}{(1-\nuone) ((2-\Jtconst) (1-\nuone) - \beta^2)}.
\end{equation} 
Then, if $\beta$ is so small that $\psi(\beta) > 0$, one has  
$$
A(t,l,r) \lesssim \psi(\beta)^r
$$
for $\Jtconst t < l < (2-\Jtconst) t$ and $1 \leq r < \nuone l$.   
\end{lemma}

\begin{lemma}     \label{lm:tail_exponential}
Let $f(t)$ be an integer-valued function such that there are $\rho_2 > \rho_1 > 1$ for which $e^{\rho_2 - 1} < \rho_1^{\rho_2}$ and $\rho_1 t < f(t) < \rho_2 t$ for all $t$ sufficiently large. Then, we have for $\lambda \in (0, 1 - \rho_2 (1 - \ln(\rho_1)))$    
\begin{equation*}   
\lim_{t \to \infty} e^{(\lambda -1)t} \sum_{n = f(t)}^{\infty} \frac{t^n}{n!} = 0.
\end{equation*}
\end{lemma}




\subsection{Proof of the Key Lemma (Lemma~\ref{lm:standard_arg})}
\label{ssec:key_lemma_proof}

Fix constants $\delta \in (0,1)$, $\kappa \in \left(e^{1-\frac{1}{\delta}}, \delta \right)$, $\nu \in (\tfrac{1}{2}, 1)$, and $\hat \kappa \in \left(e^{1-\frac{1}{\nu}}, \nu\right)$. Let $\nuzero \in (0, \tfrac{\delta}{2})$ be so small that 
$$ 
e^{1-\nuzero} > \left(\frac{e}{\kappa} \right)^{\delta}
$$ 
and let $\nuone \in (\nu^{-1} - 1,1)$. Let $\rho_1, \rho_2, \tilde \rho_1, \tilde \rho_2$ satisfy the estimates
\begin{align*} 
1 <& \rho_1 < (2-\nu) < \rho_2, & e^{\rho_2 - 1} <& \rho_1^{\rho_2}, \\
1 <& \tilde \rho_1 < \nu (1 + \nuone) < \tilde \rho_2, & e^{\tilde \rho_2 -1} <& \tilde \rho_1^{\tilde \rho_2}. 
\end{align*} 
Let $\vartheta > 0$. We choose $\beta_* > 0$ so small that the following inequalities hold for every $\beta \in (0, \beta^*]$: 
\begin{align*}
\beta <& 1, & e^{1 - \nuzero - \beta^2} >& \left(\frac{e}{\kappa} \right)^{\delta}, & \beta^2/\nuzero <& \vartheta, \\
e^{1 - \beta^2} >& \left(\frac{e}{\hat \kappa} \right)^{\nu}, & \beta^2 <& 1 - \rho_2 (1 - \ln(\rho_1)), & \beta^2 <& 1 - \tilde \rho_2 (1 - \ln(\tilde \rho_1)).     
\end{align*} 
Assume in addition that $0 < \psi(\beta) < \vartheta$ for every $\beta \in (0, \beta_*]$ (recall that $\psi(\beta)$ was defined in~\eqref{eq:def_psi}). 

Let $\tilde \sigma \in (\tfrac{3}{4}, 1)$ and $\theta > 0$ such that 
$$
\lim_{n \to \infty} n^{\theta} \sup_{\|y\| \leq n^{\tilde \sigma}} D^{\vartheta}_n(y) = 0. 
$$
Fix $\sigma \in (0, \tilde \sigma)$ and assume that the inverse temperature $\beta$ is less than or equal to $\beta_*$. By Jensen's inequality, in order to prove Lemma~\ref{lm:standard_arg}, it is enough to show that 
$$
\lim_{t \to \infty} t^{\theta}  \sup_{\|y\| \leq t^{\sublin}} \frac{1}{p_t^y} \E \left[ q^y_{n_t} \left \langle \Tc(y;0,t)^2 \right \rangle \right] = 0. 
$$
Using~\eqref{eq:T_square_bracket_A_est}, one has for $t > 0$ and $y \in \Z^d$ such that $\|y\| \leq t^{\sublin}$ the estimate   
$$
t^{\theta} (p_t^y)^{-1} \E \left[q^y_{n_t} \left \langle \Tc(y;0,t)^2 \right \rangle \right] \leq \sum_{l=0}^{\infty} t^{\theta} \frac{q^y_l}{p^y_t} e^{-t} \frac{t^l}{l!} \ \sum_{r \in R_l} \sum_{\ibf \in I_{l,r}, \zbf} \mathscr{q}^y_{l,r}(\ibf, \zbf)^2 \ A(t,l,r).  
$$
Let us write the right-hand side as  
\begin{equation}     \label{eq:three_Y_sums}
\sum_{0 \leq l \leq \nuzero t} t^{\theta} \frac{q^y_l}{p^y_t} Y_l(y,t) + \sum_{l > \nuzero t, l \notin J(t)} t^{\theta} \frac{q^y_l}{p^y_t} Y_l(y,t) + \sum_{l \in J(t)} t^{\theta} \frac{q^y_l}{p^y_t} Y_l(y,t),  
\end{equation} 
where 
$$ 
J(t) := \{l \in \N: \ \nu t < l < (2-\nu) t\} 
$$ 
was introduced in~\eqref{eq:Jt_def} and where 
\begin{equation}    \label{eq:Y_def} 
Y_l(y,t) \coleq e^{-t} \frac{t^l}{l!} \ \sum_{r \in R_l} \sum_{\ibf \in I_{l,r}, \zbf} \mathscr{q}^y_{l,r}(\ibf, \zbf)^2 \ A(t,l,r). 
\end{equation}  
Using Lemma~\ref{lm:lclt_con_time} and the fact that $q^y_l \leq 1$, one has the estimate  
\begin{equation}    \label{eq:q_p_ratio} 
\frac{q^y_l}{p^y_t} \lesssim e^{t^{\sigma}} \quad \text{for} \ \|y\| \leq t^{\sigma}, \ l \in \N_0. 
\end{equation} 
In addition,~\eqref{190921123823} yields %
\begin{equation}   \label{eq:q_square_sum}
\sum_{r \in R_l} \sum_{\ibf \in I_{l,r}, \zbf} \mathscr{q}^y_{l,r}(\ibf, \zbf)^2 \leq 4 \sum_{r \in R_l} \binom{l+1}{r} \leq 2^{l+3} \leq 2^3 e^{\nuzero t}.  
\end{equation}
The estimate in~\eqref{eq:q_p_ratio}, Lemma~\ref{lm:expected_value}, Lemma~\ref{lm:precise_I_bounds}, and~\eqref{eq:q_square_sum} then imply  
\begin{flalign*}
&\sum_{0 \leq l \leq \nuzero t} t^{\theta} \ \frac{q^y_l}{p^y_t} \ Y_l(y,t) \\
	\lesssim& e^{t^{\sublin}} t^{\theta} \sum_{0 \leq l \leq \nuzero t} Y_l(y,t) \\
    \lesssim& e^{t^{\sublin}} e^{(\beta^2 - 1)t} \ t^{\theta} \sum_{0 \leq l \leq \nuzero t} (l+1)^2 
    \sum_{r \in R_l} \beta^{2r} \frac{t^{l+r}}{(l+r)!} \sum_{\ibf \in I_{l,r}, \zbf} \mathscr{q}_{l,r}^y(\ibf, \zbf)^2 \\ 
	\lesssim& e^{t^{\sublin}} e^{(\beta^2 - 1)t} \ t^{\theta-\frac{1}{2}} \left(\frac{e}{\kappa} \right)^{\delta t} \sum_{0 \leq l \leq \nuzero t} (l+1)^2  
	\sum_{r \in R_l} \beta^{2r} \sum_{\ibf \in I_{l,r}, \zbf} \mathscr{q}^y_{l,r}(\ibf,\zbf)^2 \\
	\lesssim& e^{t^{\sublin}} e^{(\beta^2 + \nuzero -1)t} \ t^{\theta-\frac{1}{2}} \left( \frac{e}{\kappa} \right)^{\delta t} (\nuzero t + 1)^3. 
\end{flalign*}
The right-hand side does not depend on $y$ and tends to $0$ as $t \to \infty$ because $e^{1- \beta^2-\nuzero} > (e/\kappa)^{\delta}$.

With regard to the second term in~\eqref{eq:three_Y_sums}, we use~\eqref{eq:q_p_ratio}, Lemma~\ref{lm:expected_value} and the fact that $\frac{t^r}{(l+1)\cdots(l+r)} \leq \nuzero^{-r}$ for $l > \nuzero t$ to obtain  
\begin{align*} 
&\sum_{l > \nuzero t, l \notin J(t)} t^{\theta} \frac{q^y_l}{p^y_t} Y_l(y,t) \\
\lesssim& e^{t^{\sublin}} t^{\theta} \sum_{l > \nuzero t, l \notin J(t)} Y_l(y,t) \\
\lesssim& e^{t^{\sublin}} e^{(\beta^2-1)t} \ t^{\theta} \sum_{l > \nuzero t, l \notin J(t)} (l+1)^2 \frac{t^l}{l!} \sum_{r \in R_l} \beta^{2r} \frac{t^r}{(l+1) \ldots (l+r)} \sum_{\ibf \in I_{l,r}, \zbf} \mathscr{q}_{l,r}^y(\ibf,\zbf)^2  \\
\leq& e^{t^{\sigma}} e^{(\beta^2 - 1)t} \ t^{\theta} \sum_{l > \nuzero t, l \notin J(t)} (l+1)^2 \frac{t^l}{l!} \sum_{r \in R_l} (\beta^2/\nuzero)^r \sum_{\ibf \in I_{l,r}, \zbf} \mathscr{q}_{l,r}^y(\ibf, \zbf)^2 \\
=&  e^{t^{\sublin}} e^{(\beta^2 -1) t} \ t^{\theta} \sum_{l > \nuzero t, l \notin J(t)} (l+1)^2 \frac{t^l}{l!} D_l^{\beta^2 \nuzero^{-1}}(y).   
\end{align*}
For $\|y\| \leq t^{\sublin}$ and $l > \nuzero t$, we have 
$
\|y\| \leq \nuzero^{-\sublin} l^{\sublin} < l^{\tilde \sublin}
$
for $t$ large enough.
Since $\beta^2 \nuzero^{-1} < \vartheta$ and $\lim_{n \to \infty} \sup_{\|y\| \leq n^{\tilde \sublin}} D_n^{\vartheta}(y) = 0$, one has
$$
D_l^{\beta^2 \nuzero^{-1}}(y) \leq \sup_{\|z\| \leq l^{\tilde \sublin}} D_l^{\vartheta}(z) \leq \sup_{n \in \N} \sup_{\|z\| \leq n^{\tilde \sublin}} D_n^{\vartheta}(z) < \infty
$$
for such $l$ and $y$, so 
$$
\lim_{t \to \infty} \sup_{\|y\| \leq t^{\sublin}} \sum_{l > \nuzero t, l \notin J(t)} t^{\theta} \ \frac{q^y_l}{p^y_t} \ Y_l(y,t) = 0 
$$
will follow once we show that  
\begin{equation}    \label{eq:conv_tail_sum} 
\lim_{t \to \infty} e^{t^{\sublin}} e^{(\beta^2 - 1) t} \ t^{\theta} \sum_{l > \nuzero t, l \notin J(t)} (l+1)^2 \frac{t^l}{l!} = 0. 
\end{equation} 
Let us write 
$$
\sum_{l > \nuzero t, l \notin J(t)} (l+1)^2 \frac{t^l}{l!} = \sum_{\nuzero t < l \leq \nu t} (l+1)^2 \frac{t^l}{l!} + \sum_{l \geq (2-\nu) t} (l+1)^2 \frac{t^l}{l!}. 
$$
By Lemma~\ref{lm:precise_I_bounds}, the first sum on the right is dominated by 
$$
t^{-1/2} \left(\frac{e}{\hat \kappa} \right)^{\nu t} \sum_{\nuzero t < l \leq \nu t} (l+1)^2 \leq t^{-1/2} \left(\frac{e}{\hat \kappa} \right)^{\nu t} (\nu t + 1)^3. 
$$
Since $e^{1 - \beta^2} > (e/\hat \kappa)^{\nu}$, this gives 
\begin{equation}       \label{eq:conv_tail_sum_1} 
\lim_{t \to \infty} e^{t^{\sublin}} e^{(\beta^2 - 1)t} \ t^{\theta} \sum_{\nuzero t < l \leq \nu t} (l+1)^2 \frac{t^l}{l!} = 0. 
\end{equation}
And 
\begin{equation}    \label{eq:conv_tail_sum_2} 
\lim_{t \to \infty} e^{t^{\sublin}} e^{(\beta^2 - 1)t} \ t^{\theta} \sum_{l \geq (2 - \nu) t} (l+1)^2 \frac{t^l}{l!} = 0 
\end{equation} 
by virtue of Lemma~\ref{lm:tail_exponential}, with $f(t) = \lfloor (2-\nu) t \rfloor$ and with $\rho_1, \rho_2$ introduced at the beginning of the proof. The convergence statements in~\eqref{eq:conv_tail_sum_1} and~\eqref{eq:conv_tail_sum_2} imply~\eqref{eq:conv_tail_sum}.  

It remains to show that 
$$
\lim_{t \to \infty} t^{\theta} \sup_{\|y\| \leq t^{\sublin}} \sum_{l \in J(t)}  \frac{q^y_l}{p^y_t} Y_l(y,t) = 0. 
$$ 
For $l \in J(t)$, set 
\begin{align*}
Y^1_l(y,t) \coleq& e^{-t} \frac{t^l}{l!} \sum_{r \in R_l, r < \nuone l} \sum_{\ibf \in I_{l,r}, \zbf} \mathscr{q}_{l,r}^y(\ibf, \zbf)^2 A(t,l,r), \\ 
Y^2_l(y,t) \coleq& e^{-t} \frac{t^l}{l!} \sum_{r \in R_l, r \geq \nuone l} \sum_{\ibf \in I_{l,r}, \zbf} \mathscr{q}_{l,r}^y(\ibf, \zbf)^2 A(t,l,r), 
\end{align*}
and note that $Y_l(y,t) = Y^1_l(y,t) + Y^2_l(y,t)$. Let us first show that 
$$
\lim_{t \to \infty} t^{\theta} \sup_{\|y\| \leq t^{\sublin}} \sum_{l \in J(t)} \frac{q^y_l}{p^y_t} Y^1_l(y,t) = 0. 
$$
Fix $\eps > 0$ and choose $L \in \N$ so large that 
$$
l^{\theta} \sup_{\|z\| \leq l^{\tilde \sublin}} D^{\vartheta}_l(z) < \eps \Jtconst^{\theta}, \quad \forall l \geq L. 
$$
Let $t$ be so large that $\Jtconst t > L$ and $\Jtconst^{-\sublin} l^{\sublin} < l^{\tilde \sublin}$ for all $l > \Jtconst t$. Since $\psi(\beta) < \vartheta$, we have for $l > \Jtconst t$ and $y \in \Z^d$ such that $\|y\| \leq t^{\sublin}$ the estimate 
$$
t^{\theta} D_l^{\psi(\beta)}(y) \leq \Jtconst^{-\theta} l^{\theta} \sup_{\|z\| \leq l^{\tilde \sublin}} D_l^{\vartheta}(z) \leq \eps. 
$$
By Lemma~\ref{lm:A_estimate}, 
\begin{align*}
t^{\theta} \sup_{\|y\| \leq t^{\sublin}} \sum_{l \in J(t)} \frac{q^y_l}{p^y_t} Y^1_l(y,t) \lesssim& \ t^{\theta} \sup_{\|y\| \leq t^{\sublin}} (p_t^y)^{-1} \sum_{l \in J(t)} e^{-t} \frac{t^l}{l!} D^{\psi(\beta)}_l(y) \\
\leq& \ \eps \sup_{\|y\| \leq t^{\sublin}} (p_t^y)^{-1} \sum_{l \in J(t)} e^{-t} \frac{t^l}{l!} q^y_l \leq \eps. 
\end{align*}
To complete the proof of the Key Lemma, we show 
$$
\lim_{t \to \infty} t^{\theta} \sup_{\|y\| \leq t^{\sublin}} \sum_{l \in J(t)} \frac{q^y_l}{p^y_t} Y^2_l(y,t) = 0. 
$$
Let $t$ be so large that $\nu^{-\sigma} l^{\sigma} < l^{\tilde \sigma}$ for all $l > \nu t$. For $\| y \| \leq t^{\sublin}$,~\eqref{eq:q_p_ratio} and Lemma~\ref{lm:expected_value} imply
\begin{align}    
&t^{\theta} \sum_{l \in J(t)} \frac{q^y_l}{p^y_t} Y^2_l(y,t) \notag \\
	\lesssim& 
		t^{\theta} e^{t^{\sublin}} e^{(\beta^2 -1) t} \sum_{l \in J(t)} (l+1)^2 \sum_{r \in R_l, r \geq \nuone l} \beta^{2r} \frac{t^{l+r}}{(l+r)!} \sum_{\ibf \in I_{l,r}, \zbf} \mathscr{q}_{l,r}^y(\ibf,\zbf)^2.  \label{eq:Y_2_est} 
\end{align}  
For $l \in J(t)$ and $r \geq \nuone l$, 
$$
l+r \geq \lceil l (1 + \nuone) \rceil > t \Jtconst (1 + \nuone) > t, 
$$
so the expression on the right-hand side of~\eqref{eq:Y_2_est} is dominated by 
$$
t^{\theta} e^{t^{\sublin}} e^{(\beta^2 - 1) t} \sum_{l \in J(t)} (l+1)^2 \frac{t^{\lceil (1 + \nuone) l \rceil}}{\lceil (1+\nuone) l \rceil !} D^{\beta^2}_l(y)  
			  \lesssim t^{\theta} e^{t^{\sublin}} e^{(\beta^2 - 1)t} \sum_{l \geq \Jtconst (1+\nuone) t} (l+1)^2 \frac{t^l}{l!}, 
$$
where we used that 
$$
D^{\beta^2}_l(y) \leq \sup_{\|z\| \leq l^{\tilde \sublin}} D^{\vartheta}_l(z) \leq \sup_{n \in \N} \sup_{\|z\| \leq n^{\tilde \sublin}} D^{\vartheta}_n(z) < \infty. 
$$
It remains to show 
$$
\lim_{t \to \infty} t^{\theta} e^{t^{\sublin}} e^{(\beta^2 - 1)t} \sum_{l \geq \Jtconst (1+\nuone) t} (l+1)^2 \frac{t^l}{l!} = 0, 
$$
which follows from Lemma~\ref{lm:tail_exponential} with $f(t) = \lfloor \Jtconst (1+\nuone) t \rfloor$ and $\tilde \rho_1, \tilde \rho_2$ introduced at the beginning of the proof. 
 	




\section{Proof of Theorem \protect\ref{thm:factorization}}
\label{sec:proof_lm_factorization}

The proof of the factorization formula for the continuous-time model is structured in almost exactly the same way as the proof in discrete time from \\~\cite{HKNN}. There are, however, a number of technical obstacles in the continuous setting, mainly due to the fact that the number of jumps in any given time interval is random and unbounded. 

Since the law of the field of two-sided Wiener processes is invariant under regular shifts in space and under Wiener shifts in time, it suffices to show the following statement for $\beta$ sufficiently small: For every $\sublin \in (0,1)$, there is $\theta > 0$, depending on $\sublin$ but independent of $\beta$, such that 
$$ 
\lim_{t \to \infty} t^{\theta} \sup_{y \in \Z^d: \|y\| < t^{\sublin}} \langle \lvert \delta_{0,0}^{y,t} \rvert \rangle = 0. 
$$ 
For $y \in \Z^d$, $n \in \N_0$, $1 \leq r \leq n+1$, $\ibf = (i_1, \ldots, i_r) \in I(n,r)$, and $\zbf = (z_1, \ldots, z_r) \in (\Z^d)^r$, define
\begin{equation*}
q_{n}^y(\ibf,\zbf) \coleq q_{i_1}^{z_1} q_{i_2 - i_1}^{z_2 - z_1} \ldots q_{n - i_r}^{y-z_r}.
\end{equation*}
Then, from the expansion for the partition function $Z^{y,t}_{0,0}$ from Section~\ref{sec:setting}, one has
\begin{equation}    \label{eq:compact_expansion} 
	Z_{0,0}^{y,t} 
		=  p_{t}^{y} + \sum_{n=0}^{\infty} \sum_{r=1}^{n+1} \sum_{\ibf \in I(n,r), \zbf} q_{n}^y(\ibf,\zbf) \E \biggl[\prod_{j=1}^r h(z_j; s_{i_j}, s_{i_j +1}) \id_{n_t = n} \biggr], 
\end{equation} 
where we wrote $\E$ instead of $\E_{0,0}$ to simplify notation. The first step is to split the triple sum on the right into terms according to the size of the largest gap between indices, as discussed in the next subsection.

\subsection{Large and huge gaps}
\label{ssec:large_huge_gaps}

Fix $\sublin \in (0,1)$ and assume without loss of generality that $\sublin > 3/4$. We define the notions of \emph{large} and \emph{huge gaps} as in~\cite{HKNN}: Fix positive constants $\kappa_1, \kappa_2 \in (\tfrac{1}{2} (3 \sublin -1), 1)$ such that $\kappa_1 < \kappa_2$.
Let $N_{\kappa_2} \in \N$ be so large that 
$
2 (n - n^{\kappa_2}) > n 
$
for all $n \geq N_{\kappa_2}$.
Then, define
$$
k(n) \coleq \begin{cases}
             ( n-N_{\kappa_2} )^{\kappa_1} - 1, & \quad (n-N_{\kappa_2} )^{\kappa_1} - 1 \geq 1, \\
                                       0, & \quad (n-N_{\kappa_2})^{\kappa_1} - 1 < 1. 
                                       \end{cases}
$$
A collection of indices $0 \leq i_1 < \ldots < i_r \leq n$ is said to have \emph{many gaps} if $r > k (n)$.

Fix $\xi \in (0, \min \big\{ 1-\sublin, \kappa_2 - \kappa_1 \})$.
Let $n \in \N$ such that $k(n) \geq 1$, let $r \in \{1, \ldots, k(n)\}$, and consider a sequence of indices $0 = i_0 \leq i_1 < \ldots < i_r \le i_{r+1} = n$.
We call the gap between two consecutive indices $i_{j-1}$ and $i_j$ 
\begin{itemize}
\item \emph{large} if $i_j - i_{j-1} \geq n^{\xi}$;
\item \emph{huge} if $i_j - i_{j-1} \geq n - r n^{\xi}$.
\end{itemize}
Below we state some elementary facts about large and huge gaps. See~\cite{HKNN} for more details. 
\begin{enumerate}
\item There is always at least one large gap.
\item A huge gap is also large.
\item If there is only one large gap, then this large gap is also huge.
\item There is at most one huge gap.
\item Even if there is a huge gap, there may be several other large gaps. 
\end{enumerate} 
Let $r \in \N$ and $n \in \N_0$. As in~\cite{HKNN}, define for $1 \leq m \leq r+1$
$$
	I_1 (n, r, m)
		\coleq
		\left\{ (i_1, \ldots, i_r) \in I(n,r) ~:~
			\text{the gap between $i_{m-1}$ and $i_m$ is huge}
			\right\}. 
$$
Also define  
$$
	I_2 (n, r)
		\coleq
		\left\{ (i_1, \ldots, i_r) \in I(n,r)~:~
			\text{there is no huge gap}
			\right\}.
$$
Then we decompose the expansion of $Z^{y,t}_{0,0}$ from~\eqref{eq:compact_expansion} as follows:
$$ 
Z_{0,0}^{y,t} = p_t^y + \sum_{j=1}^3 \E B_j^{y,t}, 
$$
where
\begin{align*}
B_1^{y,t} \coleq &  \sum_{k(n_t) < r \leq n_t +1} \sum_{\ibf \in I(n_t,r), \zbf} q_{n_t}^y(\ibf,\zbf)  \prod_{j=1}^r h(z_j; s_{i_j}, s_{i_j +1}), \\
B_2^{y,t} \coleq & \sum_{1 \leq r \leq k(n_t)} \sum_{\ibf \in I_2 (n_t, r), \zbf} q_{n_t}^y(\ibf,\zbf)  \prod_{j=1}^r h(z_j; s_{i_j}, s_{i_j +1}), \\
B_3^{y,t}  \coleq & \sum_{1 \leq r \leq k(n_t)} \sum_{m=1}^{r+1} \sum_{\ibf \in I_1 (n_t, r, m), \zbf} q_{n_t}^y(\ibf,\zbf)  \prod_{j=1}^r h(z_j; s_{i_j}, s_{i_j +1}).  
\end{align*}
With this decomposition in hand, Theorem~\ref{thm:factorization} is a direct consequence of the following lemma.

\begin{lemma}[Central Lemma]
\label{lm:central_lemma}
For an inverse temperature $\beta$ sufficiently small, the following statements are true:
\begin{enumerate}
\item For all $\theta > 0$,	
\begin{equation}
\label{190918115555}
	\lim_{t \to \infty} t^{\theta} 
	\sup_{\|y\| \leq t^{\sublin}} 
	\frac{\langle \lvert \E B_1^{y,t} \rvert \rangle}{p_t^y} 
		= 0. 
\end{equation}
\item There is $\theta = \theta(\sigma) > 0$, independent of $\beta$, such that 
\begin{equation}
\label{190918141213}	
\lim_{t \to \infty} t^{\theta} \sup_{\|y\| \leq t^{\sublin}} \frac{\langle \lvert \E B_2^{y,t} \rvert \rangle}{p_t^y} = 0. 
\end{equation}
\item There is $\theta = \theta(\sigma) > 0$, independent of $\beta$, such that 
\begin{equation}    \label{eq:convergence_3}
\lim_{t \to \infty} t^{\theta} \sup_{\|y\| \leq t^{\sublin}} \left \langle \left \lvert 1 + \frac{\E B_3^{y,t}}{p_t^y}  - Z_{0,0}^{\infty} Z_{-\infty}^{y,t} \right \rvert \right \rangle = 0. 
\end{equation}
\end{enumerate}
\end{lemma}

Sections~\ref{ssec:proof_central_lemma_part12} and~\ref{ssec:proof_central_lemma_part3} are devoted to the proof of this lemma.

\subsection{Proof of the Central Lemma, Parts 1 and 2: Small contributions}
\label{ssec:proof_central_lemma_part12}

In this subsection, we show that the contributions of the terms $\E B_1^{y,t}$ and $\E B_2^{y,t}$ to $Z_{0,0}^{y,t}$ are negligible.
In both cases, the strategy is to reduce the task at hand to showing convergence to $0$ of the $D$-sequence associated with $B_j^{y,t}/q_{n_t}^y$. This reduction is done by applying the Key Lemma~\ref{lm:standard_arg}. Convergence of the $D$-sequence follows immediately from the proof of the discrete-time factorization formula in~\cite{HKNN}.  
Notice that $B_j^{y,t} = 0$ whenever $q^y_{n_t} = 0$. For $j = 1,2$ one has 
$$
	\frac{1}{p^y_t} 
		\Big\langle \big| \E B_j^{y,t} \big| \Big\rangle
		\leq
		\frac{1}{p^y_t} 
		\E \left[ 
			\id_{q^y_{n_t} > 0} \ q_{n_t}^y
			\left\langle 
				\left| 
					\frac{B_j^{y,t}}{q_{n_t}^y}
				\right| \right\rangle
			\right]. 
$$

\subsubsection{Proof of Part 1: Many gaps}
\label{190918135324}

Fix $\vartheta < 1/\alpha_d$. The $D$-sequence associated with $\id_{q^y_{n_t} > 0} \ B_1^{y,t}/q_{n_t}^y$ is given by
$$
	D_n^\vartheta (y)
		\coleq
		\begin{cases}
			\displaystyle \bigg.
				~ 0 & \qquad \text{if $q^y_n = 0$,}
	\\		\displaystyle
			\frac{1}{(q^y_n)^{2}}
			\sum_{k(n) < r \leq n+1} 
			\vartheta^r
			\sum_{\ibf \in I(n,r), \zbf} q_{n}^y(\ibf,\zbf)^2
				& \qquad \text{if $q^y_n > 0$.}
		\end{cases}
$$
Fix $\theta > 0$ and let $\tilde \sublin \in (\sublin, 1)$ be so close to $\sublin$ that the inequalities for $\kappa_1, \kappa_2,$ and $\xi$ continue to hold if $\sublin$ is replaced with $\tilde \sublin$, i.e., suppose 
$$ 
\kappa_1, \kappa_2 > \tfrac{1}{2} (3 \tilde \sublin -1) \quad \text{and} \quad \xi < \min\{1 - \tilde \sublin, \kappa_2 - \kappa_1\}. 
$$ 
As $\vartheta < 1/\alpha_d$, one has  
$$
\lim_{n \to \infty} n^{2 \theta} \sup_{\|y\| \leq n^{\tilde \sublin}} D_n^{\vartheta}(y) = 0,  
$$
see Section~4.2.1 of~\cite{HKNN}. Thanks to the Key Lemma~\ref{lm:standard_arg}, this yields for $\beta \leq \beta_*$, with $\beta_*$ the inverse temperature corresponding to $\vartheta$,  
$$
	\lim_{t \to \infty} t^{\theta} \sup_{\|y\| \leq t^{\sublin}} \frac{1}{p_t^y} \E \left[
			\id_{q^y_{n_t} > 0} \ q_{n_t}^y
			\left\langle 
				\left| 
					\frac{B_1^{y,t}}{q_{n_t}^y}
				\right| \right\rangle
			\right] = 0,
$$
and hence the limit~\eqref{190918115555}.

\subsubsection{Proof of Part 2: No huge gaps}
\label{190918135426}

The $D$-sequence associated with $\id_{q_{n_t}^y > 0} \ B_2^{y,t}/q_{n_t}^y$ is given by
$$
	D_n^\vartheta (y)
		\coleq
		\begin{cases}
			\displaystyle \bigg.
				~ 0 & \qquad \text{if $q^y_n = 0$,}
	\\		\displaystyle
			\frac{1}{(q^y_n)^{2}}
			\sum_{1\leq  r \leq k(n)}
			\vartheta^r
			\; \sum_{\ibf \in I_2(n,r), \zbf} q_{n}^y(\ibf,\zbf)^2
				& \qquad \text{if $q^y_n > 0$.}
		\end{cases}
$$
Choose $\tilde \sublin$ as in the proof of Part 1 and let $\theta \in (0, \xi/8)$. As $\vartheta < 1/\alpha_d$, one has 
$$
\lim_{n \to \infty} n^{2 \theta} \sup_{\|y\| \leq n^{\tilde \sublin}} D_n^{\vartheta}(y) = 0,  
$$
see Section~4.2.2 of~\cite{HKNN}. By Lemma~\ref{lm:standard_arg}, this implies the limit~\eqref{190918141213} if $\beta \leq \beta_*$.

\subsection{Proof of the Central Lemma, Part 3: The main contribution}
\label{ssec:proof_central_lemma_part3}

We follow the strategy from~\cite{HKNN}. For $\ibf \in I_1(n, r,m)$ and $\zbf \in (\Z^d)^r$, define
\begin{equation} \label{eq:def_qminusm}
q_{n,\widehat{m}}^y(\ibf,\zbf) := q_{i_1}^{z_1} \ldots  \widehat{ q_{i_{m} - i_{m-1}}^{z_{m}-z_{m-1}} }\ldots q_{n-i_r}^{y-z_r},
\end{equation}
where the factor with the hat is equal to $1$; in other words, we remove the transition probability corresponding to the huge gap.

Now decompose $B_3^{y,t}$ further, depending on the position of the huge gap 
1) at the beginning, 2) in the middle, or 3) at the end, as follows:

\begin{equation*}
B_3^{y,t} = \id_{q^y_{n_t} > 0} \ q_{n_t}^y \sum_{i=1}^3 \left(F_i^{y,t} + L_i^{y,t} \right),
\end{equation*}
where
\begin{align*}
	F_{1}^{y,t} \coleq & \sum_{1 \leq r \leq k(n_t)}  \sum_{\ibf \in I_1(n_t, r,1), \zbf} 
	q_{n_t,\widehat{1}}^y(\ibf,\zbf)  
	\prod_{j=1}^r h(z_j; s_{i_j}, s_{i_j + 1}), \\ 
	F_2^{y,t} \coleq & \sum_{2 \leq r \leq k(n_t)} \sum_{m=2}^r \; 
	\sum_{\ibf \in I_1(n_t,r,m),  \zbf }  
	q_{n_t,\widehat{m}}^y(\ibf,\zbf)   
	\prod_{j=1}^r h(z_j; s_{i_j}, s_{i_j +1}), \\ 
	F_3^{y,t} \coleq & \sum_{1 \leq r \leq k(n_t)} \; 
	\sum_{\ibf \in I_1(n_t,r,r+1),  \zbf } 
	q_{n_t,\widehat{r+1}}^y(\ibf,\zbf)
	\prod_{j=1}^r h(z_j; s_{i_j}, s_{i_j +1});
\end{align*}
and the error terms are given by
\begin{align*}
	L_1^{y,t} \coleq & \sum_{1 \leq r \leq k(n_t)} \;
	\sum_{\ibf \in I_1(n_t,r,1), \zbf}
	\frac{q_{i_1}^{z_1} - q_{n_t}^{y}}{q^{y}_{n_t}} 
	q_{n_t,\widehat{1}}^y(\ibf,\zbf)  
	\prod_{j=1}^r h(z_j; s_{i_j}, s_{i_j + 1}), \\
	L_2^{y,t} \coleq & \sum_{2 \leq r \leq k(n_t)} \sum_{m=2}^r \;
	\sum_{\ibf \in I_1(n_t,r,m), \zbf} 
	\frac{q_{i_m - i_{m-1}}^{z_m - z_{m-1}} - q_{n_t}^{y}}{q_{n_t}^{y}} 
	q_{n_t,\widehat{m}}^y(\ibf,\zbf)     
	\prod_{j=1}^r h(z_j; s_{i_j}, s_{i_j + 1}),   \\  
	L_3^{y,t} \coleq & \sum_{1 \leq r \leq k(n_t)} \;
	\sum_{\ibf \in I_1(n_t,r,r+1), \zbf} 
	\frac{q_{n_t-i_r}^{y-z_r}- q_{n_t}^{y}}{q^{y}_{n_t}} 
	q_{n_t,\widehat{r+1}}^y(\ibf,\zbf) 
	\prod_{j=1}^r h(z_j; s_{i_j}, s_{i_j + 1}).  
\end{align*}

\bigskip
\noindent 
We first show that the contribution from each error term is negligible.

\begin{lemma} \label{lm:L_convergence}
For an inverse temperature $\beta \leq \beta_*$, the following is true: There is $\theta = \theta(\sigma) > 0$, independent of $\beta$, such that
\begin{equation} 
\lim_{t \to \infty} t^{\theta} \sup_{\|y\| \leq t^{\sublin}} \frac{1}{p_t^y} \biggl \langle \biggl \lvert \E \biggl[\id_{q_{n_t}^y > 0} \ q_{n_t}^y \sum_{i=1}^3 L_i^{y,t}\biggr] \biggr \rvert \biggr \rangle  = 0. \label{eq:conv_1}
\end{equation}
\end{lemma}


\begin{proof}
For $i \in \{1,2,3\}$, let $D^{\vartheta}_n(i,y)$ denote the $D$-sequence associated with $\id_{q^y_{n_t} > 0} \ L_i^{y,t}$. If $q^y_n = 0$, one has $D^{\vartheta}_n(i,y) = 0$. And if $q^y_n > 0$, 
$$ 
D^{\vartheta}_n(i,y) = \sum_{1 \leq r \leq k(n)} \vartheta^r a_i(r) \sum_{\substack{\nbf \in \Xi^i_r, \zbf}} \left(q_{n_1}^{z_1} \right)^2 \ldots \left(q_{n_r}^{z_r} \right)^2 \frac{(q_{n - n_1 - \ldots - n_r}^{y - z_1 - \ldots - z_r} - q^y_n)^2}{(q^y_n)^2}, 
$$ 
where $a_i(r) := 1$ if $i = 1,3$, $a_2(r) \coleq (r-1) \id_{r \geq 2}$,\\ 
\begin{align*}
\Xi_r^i \coleq& \left\{\nbf = (n_1,\ldots, n_r) \in \N_0^r ~:~
			\substack{ \displaystyle  n_1 \geq 0, n_2,\ldots, n_r > 0 
			\\ \displaystyle
			n_1 + \cdots + n_r  \leq r n^\xi }\right\}, \quad i=1,3, \\
\Xi_r^2 \coleq& \left\{\nbf = (n_1,\ldots, n_r) \in \N_0^r  
~:~
			\substack{ \displaystyle  n_1,n_r \geq 0, n_2,\ldots,n_{r-1} > 0 
			\\ \displaystyle
			n_1 + \cdots + n_r  \leq r n^\xi } \right\}.
\end{align*}
Choose $\tilde \sublin$ as in the proof of Part 1 and let 
$$
0 < \theta < \min\{1/5; 1/2 (1 - \tilde \sublin - \xi)\}.
$$
As $\vartheta < 1/\alpha_d$, one has 
$$ 
\lim_{n \to \infty} n^{2 \theta} \sup_{\|y\| \leq n^{\tilde \sublin}} D^{\vartheta}_n(i,y) = 0, \quad i \in \{1,2,3\},  
$$ 
see Lemma~4.2 of~\cite{HKNN}. The Key Lemma~\ref{lm:standard_arg} then implies~\eqref{eq:conv_1}. 
\end{proof} 


In order to deal with the $F_i$'s, we first truncate the partition functions $Z_{0,0}^t$ and $Z_0^{y,t}$.
Fix $\xi_1, \xi_2, \xi_3$ satisfying 
$$
0 < \xi_1 < \xi_2 < \xi_3 < \xi,
$$ 
and notice that since $\xi + \sublin < 1$, we have $\xi_1 + \sublin < 1$. For $n \in \N_0$, set 
$$ 
v(n) \coleq \left \lceil \frac{n}{2} \right \rceil^{\xi_2}, 
\quad w(n) \coleq \left \lceil \frac{n}{2} \right \rceil^{\xi_3}. 
$$

\noindent To avoid heavy notation, we write
\begin{equation}    \label{eq:def_n_symb} 
n \coleq n_t, \quad 
n_{-} \coleq n_{t^{\xi_1}}, \quad
n_{\bullet} \coleq n_{t^{\xi_1}, t-t^{\xi_1}}, \quad 
n_+ \coleq n_{t-t^{\xi_1}, t}
\end{equation} 
and observe that $n = n_- + n_{\bullet} + n_+$. 
Now define
$$
T_{0,0}^t \coleq 
1 + \sum_{1 \leq r \leq v(n_-)+1} 
\sum_{\substack{\ibf \in I(n,r), \zbf \\ i_r \leq w(n_-)}} 
q_{n,\widehat{r+1}}^y(\ibf,\zbf)
\prod_{j=1}^r h(z_j; s_{i_j}, s_{i_j +1})
$$ 
and 
$$
T^{y,t}_0  \coleq
1 + \sum_{1 \leq r \leq v(n_+)+1} 
\sum_{\substack{\ibf \in I(n,r), \zbf \\ n_- w(n_+) \leq i_1}} 
q_{n,\widehat{1}}^y(\ibf,\zbf)
 \prod_{j=1}^r h(z_j; s_{i_j}, s_{i_j +1}). 
$$
Notice that $\E[T_{0,0}^t]$ and $\E[T_0^{y,t}]$ are truncations of the partition functions $Z_{0,0}^t$ and $Z_{0}^{y,t}$, respectively.
In light of Lemma~\ref{lm:L_convergence}, the convergence statement in~\eqref{eq:convergence_3} is implied by the following two lemmas.

\begin{lemma} \label{lm:main_claim1-4}
For an inverse temperature $\beta$ sufficiently small, the following is true: There is $\theta = \theta(\sigma) > 0$, independent of $\beta$, such that 
\begin{align}
	\lim_{t \to \infty} t^{\theta} \sup_{\|y\| \leq t^{\sublin}} \frac{1}{p_t^y} \left \langle \left \lvert \E  \left[ q^y_{n_t} \left(F_2^{y,t} - (T_{0,0}^t - 1) (T^{y,t}_0 - 1 ) \right) \right] \right \rvert \right \rangle &= 0, \label{eq:conv_4} \\
	\lim_{t \to \infty} t^{\theta} \sup_{\|y\| \leq t^{\sublin}} \frac{1}{p_t^y}\left \langle \left \lvert \E \left[ q^y_{n_t} \left(F_1^{y,t} - (T^{y,t}_0 - 1) \right)\right] \right \rvert \right \rangle &= 0, \label{eq:conv_2} \\
	\lim_{t \to \infty} t^{\theta} \sup_{\|y\| \leq t^{\sublin}} \frac{1}{p_t^y} \left \langle \left \lvert \E \left[ q^y_{n_t} \left(F_3^{y,t} - (T_{0,0}^t - 1) \right) \right] \right \rvert \right \rangle &= 0. \label{eq:conv_3} 
\end{align}
\end{lemma}

\begin{lemma} \label{lm:main_claim5}
For an inverse temperature $\beta$ sufficiently small, the following is true: There is $\theta = \theta(\sigma) > 0$, independent of $\beta$, such that
\begin{equation}.  \label{eq:conv_5} 
\lim_{t \to \infty} t^{\theta} \sup_{\|y\| \leq t^{\sublin}} \frac{1}{p_t^y} \left \langle \left \lvert \E \left[ q^y_{n_t} T^{y,t}_0 T_{0,0}^t \right] - p_t^y Z_{0,0}^{\infty} Z_{-\infty}^{y,t} \right \rvert \right \rangle = 0.   
\end{equation}
\end{lemma}

The convergence statement in~\eqref{eq:conv_4} is shown in Section~\ref{ssec:convergence_middle}. We show the convergence statements in~\eqref{eq:conv_2} and~\eqref{eq:conv_3} in Section~\ref{ssec:gap_end}, and the one in~\eqref{eq:conv_5} in Section~\ref{ssec:convergence_ergodicity}.

\section{Main Contribution: \\ Proofs of Lemmas~\ref{lm:main_claim1-4} and~\ref{lm:main_claim5}}
\label{ssec:main_contribution_lemmas}

\subsection{Preliminaries}

Recall that $n, n_-, n_{\bullet}, n_+$ were defined in~\eqref{eq:def_n_symb} and depend on $t > 0$. We set $\oplus := \{-, \bullet, +\}$ and introduce the following shorthands: 
\begin{align*}
\Pp^i(m) =& \Pp(n_i = m), \quad i \in \oplus, \\
\Pp^{i,j}(m,l) =& \Pp(n_i= m, n_j = l), \quad i,j \in \oplus, \ i \neq j, \\
\Pp(m,l,k) =& \Pp(n_- =m, n_{\bullet} = l, n_+ = k), \\
\Pp^{(i,j)}(m) =& \Pp(n_i + n_j = m), \quad i,j \in \oplus, \ i \neq j, \\
\Pp^{-,(\bullet,+)}(m,l) =& \Pp(n_- = m, n_{\bullet} + n_+ = l), \\
\Pp^{\bullet,(-,+)}(m,l,k) =& \Pp(n_{\bullet} = m, n_- + n_+ = l),
\end{align*}
\noindent as well as the analogous definitions for $\E^i[\cdot \vert m]$, $\E^{i,j}[\cdot \vert m,l]$, $\E[\cdot \vert m,l,k]$, etc.

Next, we make the following preliminary observations, which will be applied several times in the proofs of Lemmas~\ref{lm:main_claim1-4} and~\ref{lm:main_claim5}. Let 
$$
R \subseteq \{1, \ldots , n\}.
$$
For $r \in R$, let
$$
M_r \subseteq \{1, \ldots , r+1\}
$$
and for $m \in M_r$, let 
$$
H_{r,m} \subseteq I_1(n,r,m).
$$
Define
\begin{equation}
f \coleq 
\sum_{r \in R} \ \sum_{m \in M_r}  
\sum_{\ibf \in H_{r,m}, \zbf} 
q^{y}_{n,\widehat m}(\ibf,\zbf)^2
\prod_{j=1}^r h(z_j; s_{i_j}, s_{i_j +1}).
\end{equation}
Notice that, for example, each $F_i^{y,t}$ can be expressed in this way.
Then, given $S \subset \N^3$, one has 
\begin{align} \label{eq:E_gen_f}
& \E \left[ q^y_n \left \langle f ^2 \right \rangle 
\id_{(n_-, n_{\bullet}, n_+) \in S} \right] \notag \\
=&  \sum_{(l_-, l_{\bullet}, l_+) \in S} 
q_{l_-+l_\bullet+l_+}^y
 \Pp(l_-,l_{\bullet},l_+) \E \left[ \left \langle f ^2 \right \rangle \bigg| l_-,l_{\bullet},l_+ \right]   
\end{align}
and 
\begin{align}      \label{eq:f_i_expansion}
 \E \left[ \left \langle f^2 \right \rangle \bigg\vert l_-, l_{\bullet}, l_+ \right]   
=& \sum_{r \in R} \ \sum_{m \in M_r}  \ \sum_{\ibf \in H_{r,m}, \zbf} \notag \\
& q^{y}_{l_-+l_\bullet + l_+,\widehat m}(\ibf,\zbf)^2 \ \E \biggl[ \prod_{j=1}^r \left(e^{\beta^2 t_{i_j +1}} - 1 \right) \bigg\vert l_-, l_{\bullet}, l_+ \biggr].      
\end{align}

\noindent For $(l_-, l_{\bullet}, l_+) \in S$, 
$r \in R$, $m \in M_r$, and $\ibf \in H_{r,m}$, let 
\begin{align*}
r_- &= \left \lvert \left\{1 \leq j \leq r: i_j < l_- \right\} \right \rvert, \\
r_{\bullet} &= \left \lvert \left\{1 \leq j \leq r: l_- \leq i_j < l_- + l_{\bullet}  \right\} \right \rvert, \\
r_+ &= r - r_- - r_{\bullet}
\end{align*}
and observe that 
\begin{align}   \label{eq:triple_A}
&\E \biggl[ \prod_{j=1}^r  \left(e^{\beta^2 t_{i_j +1}} - 1 \right) \bigg\vert l_-, l_{\bullet}, l_+ \biggr]  \notag \\
\leq& \ A(t^{\xi_1}, l_-, r_- +1) \ A(t-2t^{\xi_1}, l_{\bullet}, r_{\bullet}+1) \ A(t^{\xi_1},l_+,r_+ +1),  
\end{align}
where $A(t,l,r)$ was defined in~\eqref{eq:def_hat_A}.

In the proofs of of Lemmas~\ref{lm:main_claim1-4} and~\ref{lm:main_claim5}, we will need the following technical results, which we prove in Appendix~\ref{ssec:building_blocks_proofs}.


\begin{lemma}    \label{lm:building_blocks}
Let $\nu \in (\tfrac{1}{2}, 1)$. Recall from~\eqref{eq:Jt_def} that $J(t) = \{n \in \N: \nu t < n < (2-\nu) t\}$, and recall from~\eqref{eq:def_iota} that $\iota(y,n) = n$ if $\|y\|_1 \equiv n$ and $\iota(y,n) = n+1$ if $\|y\|_1 \not\equiv n$. For $\beta$ sufficiently small, the following holds: For any $\theta, c > 0$, we have
{\fontsize{10}{9}\selectfont
\normalfont
\begin{enumerate}[label=(A\arabic*)]
\setcounter{enumi}{-1}
\item \label{A0}
$\displaystyle
\limsup_{t \to \infty} e^{-\beta^2 t^{\xi_1}} \sup_{\|y\| \leq t^{\sublin}} \frac{1}{p_t^y} \sum_{l \in J(t-2t^{\xi_1})} q^y_{\iota(y,l)} \Pp^{\bullet}(l) < \infty;
$
\bigskip
\item \label{A1}
$\displaystyle
\limsup_{t \to \infty} e^{-\beta^2 t^{\xi_1}} \hspace{-3mm} \sup_{\|y\| \leq t^{\sublin}} \frac{1}{p_t^y} \\ \hspace{-1mm} \sum_{l \in J(t-2t^{\xi_1})} \hspace{-4mm} q^y_{\iota(y,l)} \Pp^{\bullet}(l) \hspace{-2mm} \sum_{0 \leq r \leq l} \hspace{-1mm}(r+1) \alpha_d^r \
A(t-2t^{\xi_1},l,r+1) <  \infty;
$
\bigskip
\item \label{A2}
$\displaystyle
\limsup_{t \to \infty} \sum_{l \in J(t^{\xi_1})} \Pp^{-}(l) \sum_{0 \leq r < \nuone l -1} (r+1) \alpha_d^r \ A(t^{\xi_1},l,r+1) < \infty;
$
\bigskip
\item \label{A3}
$\displaystyle 
\lim_{t \to \infty} t^{\theta} e^{\beta^2 t^{\xi_1}} \sum_{l \in J(t^{\xi_1})} \Pp^-(l) \sum_{\nuone l -1 \leq r \leq l} (r+1) \alpha_d^r \ A(t^{\xi_1},l,r+1) = 0$; 
\bigskip
\item \label{A4}
$\displaystyle
\lim_{t \to \infty} t^{\theta} e^{\beta^2 t^{\xi_1}} \sum_{l \notin J(t^{\xi_1})} e^{c t^{\sublin -1} l} \ \Pp^-(l) \sum_{0 \leq r \leq l} (r+1) \alpha_d^r \ A(t^{\xi_1},l,r+1) = 0;
$ 
\bigskip   
\item \label{A5}
$\displaystyle 
\lim_{t \to \infty} t^{\theta} e^{t^{\sublin}} \sum_{l \notin J(t-2t^{\xi_1})} \Pp^{\bullet}(l) \sum_{0 \leq r \leq l} (r+1) \alpha_d^r \ A(t-2t^{\xi_1},l,r+1) = 0;
$
\bigskip   
\item \label{A6}
$\displaystyle 
\lim_{t \to \infty} \sum_{l \notin J(t^{\xi_1})} \Pp^-(l) \sum_{0 \leq r \leq l} (r+1) \alpha_d^r \ A(t^{\xi_1},l,r+1) = 0;
$
\bigskip   
\item \label{A2A3}
$\displaystyle 
\limsup_{t \to \infty} \sum_{l \in J(t^{\xi_1})} \Pp^{-}(l) \sum_{0 \leq r < l } (r+1) \alpha_d^r \ A(t^{\xi_1},l,r+1) < \infty. 
$
\end{enumerate}
}
\end{lemma}
The following lemma is a corollary of Lemma~3.4 in~\cite{HKNN} and can be shown in the same way as Lemma~3.3 (see Appendix~\ref{ssec:transition_prob_proofs}).  
\begin{lemma} \label{lm:ql+._}
There exists a constant $C > 0$ such that for $\| y \| < t^\sigma$, $l_\bullet \in J(t-2t^{\xi_1})$, $l_-,l_+\in \N_0$, one has 
\begin{equation*}
q_{l_- + l_\bullet + l_+}^y \lesssim \prod_{ s \in \{-,+\} } \exp\left(Ct^{\sigma -1}l_s \right) q_{\iota(y,l_\bullet)}^y.
\end{equation*}
\end{lemma}


\subsection{Proof of Lemma~\ref{lm:main_claim1-4}, Part 1: Convergence for one huge gap in the middle}
\label{ssec:convergence_middle}

In this subsection, we prove the convergence statement in~\eqref{eq:conv_4}. Let 
$$ 
f_2^{y,t} \coleq F_2^{y,t} - (T_{0,0}^t -1) (T^{y,t}_0 -1).
$$
In order to prove~\eqref{eq:conv_4}, it is clearly enough to show the following claim. 

\begin{claim} \label{cl:one_huge}
Let $\beta$ be sufficiently small. There is $\theta = \theta(\sigma) > 0$, independent of $\beta$, such that 
\begin{align}
\lim_{t \to \infty} t^{\theta} \sup_{\|y\| \leq t^{\sublin}} \frac{1}{p_t^y} \left \langle \left \lvert \E \left[ q^y_n f_2^{y,t} \id_{n_-, n_+ \in J(t^{\xi_1}), n_{\bullet} \in J(t - 2 t^{\xi_1})} \right] \right \rvert \right \rangle =& 0, \label{eq:ibob_1} \\
\lim_{t \to \infty} t^{\theta} \sup_{\|y\| \leq t^{\sublin}} \frac{1}{p_t^y} \left \langle \left \lvert \E \left[ q^y_n f_2^{y,t} \left(1 - \id_{n_-, n_+ \in J(t^{\xi_1}), n_{\bullet} \in J(t-2 t^{\xi_1})} \right) \right] \right \rvert \right \rangle =& 0. \label{eq:ibob_2} 
\end{align}
\end{claim}


\subsubsection{{Proof of Claim~\ref{cl:one_huge}, Part 1}}

We first show~\eqref{eq:ibob_1}. By symmetry considerations, it suffices to show that there is $\theta = \theta(\sigma) > 0$ such that 
\begin{equation}   \label{eq:n_minus_n_plus}
\lim_{t \to \infty} t^{\theta} \sup_{\|y\| \leq t^{\sublin}} \frac{1}{p_t^y} \E \left[ q^y_n \left \langle \left( f_2^{y,t} \right)^2 \right \rangle \id_{n_-, n_+ \in J(t^{\xi_1}), n_{\bullet} \in J(t-2t^{\xi_1}), n_- \leq n_+} \right] = 0. 
\end{equation} 
We have   
	\begin{align}   \label{eq:T_product_1}
	& (T_{0,0}^t - 1) (T^{y,t}_0 - 1)   \notag \\
	=& \sum_{1 \leq r \leq v(n_-)+1} \sum_{1 \leq s \leq v(n_+)+1} \sum_{\substack{0 \leq i_1 < \ldots < i_r \leq w(n_-), \\ z_1, \ldots, z_r \in \Z^d}} \sum_{\substack{n - w(n_+) \leq l_1 < \ldots < l_s \leq n, \\ c_1, \ldots, c_s \in \Z^d}} \notag \\
	&  q_{i_1}^{z_1} \ldots q_{i_r - i_{r-1}}^{z_r - z_{r-1}} q_{l_2 - l_1}^{c_2 - c_1} \ldots  q_{n - l_s}^{y - c_s} \prod_{j=1}^r h(z_j; s_{i_j}, s_{i_j +1}) \prod_{k=1}^s h(c_k; s_{l_k}, s_{l_k +1}).  
\end{align}
%
Define the set
$$
	V_{r,m}
		\coleq
		\left\{ \ibf = (i_1, \ldots, i_r) \in I_1(n,r,m) ~:~
			\substack{ \displaystyle 
				0 \leq i_1 < \ldots < i_{m-1} \leq w(n_-) 
			\\	\displaystyle
				\quad n - w(n_+) \leq i_m < \ldots < i_r \leq n
			}
			\right\}
$$
and its complement in $I_1(n,r,m)$
$$
	W_{r,m}
		\coleq
		\left\{ \ibf = (i_1, \ldots, i_r) \in I_1(n,r,m) ~:~
			\substack{ \displaystyle 
				i_m -  i_{m-1} < n - w(n_-) - w(n_+) 
			}
			\right\}.
$$
Suppose now that $n_- \leq n_+$, so that $v(n_-) \leq v(n_+)$, and 
recall the notation $q^{y}_{r,\widehat m}(\ibf,\zbf)$ from \eqref{eq:def_qminusm}:
\begin{equation*}
q_{n,\widehat{m}}^y(\ibf,\zbf) := q_{i_1}^{z_1} \ldots  \widehat{ q_{i_{m} - i_{m-1}}^{z_{m}-z_{m-1}} }\ldots q_{n-i_r}^{y-z_r}, \quad \ibf \in I_1(n,r,m), \ \zbf \in (\Z^d)^r. 
\end{equation*}
Making the change of summation indices $r\coleq r+s$ and $m\coleq r+1$, the expression in~\eqref{eq:T_product_1} can be written as
\begin{align}    \label{eq:p6_10}
&\hspace{-3mm}  \sum_{2 \leq r \leq v(n_-)+ 2} \ \sum_{m=2}^r \ 
\sum_{\ibf \in V_{r,m}, \zbf}  
q_{n,\widehat m}^y(\ibf,\zbf) \prod_{j=1}^r h(z_j; s_{i_j}, s_{i_j +1}) \notag \\
+& \sum_{v(n_-)+ 2 < r \leq v(n_+)+2} \ \sum_{2 \leq m \leq v(n_-)+ 2} \ 
\sum_{\ibf \in V_{r,m}, \zbf} 
q_{n,\widehat m}^y(\ibf,\zbf) \prod_{j=1}^r h(z_j; s_{i_j}, s_{i_j +1})  \notag \\
+& \sum_{\substack{v(n_+)+ 2  < r, \\ r \leq v(n_-) + v(n_+)+ 2}} \
 \sum_{\substack{r - v(n_+) \leq m, \\ m \leq v(n_-) + 2}} \ 
\sum_{\ibf \in V_{r,m}, \zbf} 
q_{n,\widehat m}^y(\ibf,\zbf) \prod_{j=1}^r h(z_j; s_{i_j}, s_{i_j +1}).   
	\end{align}
Conditional on $n_-, n_+ \in J(t^{\xi_1})$, $n_{\bullet} \in J(t-2t^{\xi_1})$, and $n_- \leq n_+$, and provided that $t$ is sufficiently large,~\eqref{eq:p6_10} allows us to rewrite $f_2^{y,t}$ as
$
f^{y,t}_{2;1} + f^{y,t}_{2;2} +f^{y,t}_{2;3}, 
$
where 
	\begin{align*}
	f^{y,t}_{2;1} \coleq & 
	\sum_{r \in R^1} \ 
	\sum_{m \in M_r^1}  \ 
	\sum_{\ibf \in W_{r,m}, \zbf} 
	q^{y}_{n,\widehat m}(\ibf,\zbf) 
	\prod_{j=1}^r h(z_j; s_{i_j}, s_{i_j +1}), \\
	f^{y,t}_{2;2} \coleq & 
	\sum_{r \in R^2} 
 	\sum_{m=2}^r 
	\sum_{\ibf \in I_1(n_t,r,m),  \zbf } q^y_{n, \widehat m}(\ibf, \zbf) \prod_{j=1}^r h(z_j; s_{i_j}, s_{i_j +1}) \\
	& - \sum_{r \in R^2} 	 \sum_{2 \leq m \leq v(n_-)+ 2} \ 
\sum_{\ibf \in V_{r,m}, \zbf} 
q_{n,\widehat m}^y(\ibf,\zbf) \prod_{j=1}^r h(z_j; s_{i_j}, s_{i_j +1}),	\\
	f^{y,t}_{2;3} \coleq & 
	\sum_{r \in R^3} \ 
	\sum_{m \in M_r^3} \ 
	\sum_{\ibf \in I_1(n,r,m), \zbf} 
	q^{y}_{n,\widehat m}(\ibf,\zbf) 
	\prod_{j=1}^r h(z_j; s_{i_j}, s_{i_j +1}),
	\end{align*}

\noindent with $R^1 \coleq \{2, \ldots, v(n_-) + 2\}$, $R^2 \coleq \{v(n_-) + 3, \ldots, v(n_-) + v(n_+) + 2\}$, and $R^3 \coleq \{v(n_-) + v(n_+) + 3, \ldots, k(n)\}$; and $M_r^1 \coleq \{2, \ldots, r\}$, $M_r^3 \coleq \{2, \ldots, r\}$.

\bigskip
	
In order to prove~\eqref{eq:n_minus_n_plus}, it is then enough to show existence of $\theta = \theta(\sigma) > 0$ such that for $i=1,2,3$,
\begin{equation}  \label{eq:1_i_7}
\lim_{t \to \infty} t^{\theta} \sup_{\|y\| \leq t^{\sublin}} \frac{1}{p_t^y} \E \left[ q^y_n \left \langle \left( f_{2;i}^{y,t} \right)^2 \right \rangle \id_{n_-, n_+ \in J(t^{\xi_1}), n_{\bullet} \in J(t-2t^{\xi_1}), n_- \leq n_+} \right] = 0. 
\end{equation}
Let $i=1,2,3$. Conditioning on the number of jumps in $[0,t^{\xi_1})$, $[t^{\xi_1}, t-t^{\xi_1})$, and $[t-t^{\xi_1},t)$, equation \eqref{eq:E_gen_f} allows us to write 
\begin{align*}
	& \E \left[ q^y_n 
	\left \langle \left(f_{2;i}^{y,t} \right)^2 \right \rangle 
	\id_{n_-, n_+ \in J(t^{\xi_1}), n_{\bullet} \in J(t-2t^{\xi_1}), n_- \leq n_+} 
	\right] \\
=& \sum_{(l_-, l_{\bullet}, l_+) \in S} q^y_{l_- + l_{\bullet} + l_+} \Pp(l_-,l_{\bullet},l_+) 
\E \left[ \left \langle \left(f_{2;i}^{y,t} \right)^2 \right \rangle \bigg| l_-,l_{\bullet},l_+ \right],    
\end{align*}
\noindent where $S = \{(l_-, l_{\bullet}, l_+): l_-, l_+ \in J(t^{\xi_1}), \ l_{\bullet} \in J(t - 2t^{\xi_1}), \ l_- \leq l_+\}$.

In order to show~\eqref{eq:1_i_7} for $i=1,2$, we will show that the term
$$
 \E [  \langle (f_{2;i}^{y,t} )^2  \rangle | l_-,l_{\bullet},l_+ ]
$$
can be bounded, uniformly in $(l_-, l_{\bullet}, l_+) \in S$, by a function that goes to 0 at least polynomially fast as $t \to \infty$; 
then~\eqref{eq:1_i_7} follows from the fact that 
\begin{equation}	\label{eq:E_of_q_is_p}
\sum_{(l_-, l_{\bullet}, l_+) \in S} q^y_{l_- + l_{\bullet} + l_+} \Pp(l_-,l_{\bullet},l_+) \leq p_t^y.
\end{equation}
Notice that for $i = 1, 3$, the expression $ \E [  \langle (f_{2;i}^{y,t})^2  \rangle \vert l_-, l_{\bullet}, l_+ ]$ can be expanded in the following way, using~\eqref{eq:f_i_expansion}:

\begin{equation}     \label{eq:f2_i_expansion}
 \sum_{r \in R^i} \ \sum_{m \in M_r^i}  \ \sum_{\ibf \in H_{r,m}^i, \zbf} 
q^{y}_{l_-+l_\bullet + l_+,\widehat m}(\ibf,\zbf)^2 \ \E \biggl[ \prod_{j=1}^r \left(e^{\beta^2 t_{i_j +1}} - 1 \right) \bigg\vert l_-, l_{\bullet}, l_+ \biggr], 
\end{equation}
where $H_{r,m}^1 = W_{r,m}$ and $H_{r,m}^3 = I_1(n,r,m)$. Notice further that 
$$
 \E [  \langle (f_{2;2}^{y,t})^2  \rangle \vert l_-, l_{\bullet}, l_+ ]
 $$
is bounded by~\eqref{eq:f2_i_expansion} with $i=2$ and $M_r^2 = \{2,\ldots,r\}$, $H_{r,m}^2 = I_1(n,r,m)$.

Recall that $\sublin$ is the coefficient associated with the family of sets $J(t), t > 0$, and fix $\nuone \in (\Jtconst^{-1} -1,1)$. For $i = 1,2$, for $t$ sufficiently large, and for $l_- \leq l_+$ in $J(t^{\xi_1})$ and $l_{\bullet} \in J(t-2 t^{\xi_1})$, we have 
$$ 
r+1 \leq v(l_-) + v(l_+) + 3 = \Big\lceil\frac{l_-}{2}\Big\rceil^{\xi_2} + \Big\lceil\frac{l_+}{2}\Big\rceil^{\xi_2} + 3 < \nuone l_-,  \quad r \in R^i, 
$$ 
so in particular 
$
r_s + 1 < \nuone l_s  
$ 
for $s \in \oplus$.
Therefore, since $r = r_- + r_\bullet + r_+$, by Lemma~\ref{lm:A_estimate} and for $\beta$ so small that $\psi(\beta) > 0$, 
\begin{align}\label{eq:triple_A_bound}
A(t^{\xi_1}, l_-, r_- +1)  \  A&(t-2t^{\xi_1}, l_{\bullet}, r_{\bullet}+1) \notag \\
& A(t^{\xi_1},l_+,r_+ +1) \lesssim \psi(\beta)^{r_- + r_{\bullet} + r_+} = \psi(\beta)^r. 
\end{align}
Now, we take up cases $i=1,2,3$ separately.
\paragraph*{\textsc{Case} $i = 1$.} Using the bounds in~\eqref{eq:triple_A} and~\eqref{eq:triple_A_bound}, and taking $H_{r,m} = W_{r,m}$, the term in~\eqref{eq:f2_i_expansion} is less than a constant times
\begin{align*} 
& \sum_{r \in R^1}
\sum_{m \in M_r^1}
\sum_{\substack{\jbf = (j_1, \ldots, j_{r}) \in \N^{r}, \\ j_{m}-j_{m-1} < n - w(l_-) - w(l_+)}} 
\sum_{\zbf} \left(q_{j_1}^{z_1} \right)^2 \ldots \left(q_{j_r}^{z_r} \right)^2 
\psi(\beta)^{r}  \\
\lesssim& \sum_{r \in R^1} r  
\sum_{\substack{j_1, \ldots, j_{r-1} \in \N, \\ j_1+ \ldots + j_{r-1} > \frac{1}{3} (w(l_-) + w(l_+))}} 
\sum_{\zbf} 
\left(q_{j_1}^{z_1} \right)^2 \ldots \left(q_{j_r}^{z_r} \right)^2 \psi(\beta)^{r} \\ 
\lesssim& \sum_{r=1}^{\infty} r^2 (\alpha_d \psi(\beta))^{r} \sum_{j > \frac{w(l_-) + w(l_+)}{3 (v(l_-)+2)}} \frac{1}{j^{\frac{d}{2}}} \lesssim \left(\frac{w(l_-) + w(l_+)}{v(l_-)} \right)^{1-\frac{d}{2}}. 
\end{align*} 
For the last estimate we used that $0 < \psi(\beta) < \vartheta$ and thus $\alpha_d \psi(\beta) < \alpha_d \vartheta < 1$. Since 
$$
\frac{w(l_-) + w(l_+)}{v(l_-)} \gtrsim l_-^{\xi_3-\xi_2} \gtrsim t^{\xi_1 (\xi_3-\xi_2)}, 
$$
it follows that 
$$
\sup_{\|y\| \leq t^{\sublin}} \frac{1}{p_t^y} \E \left[ q^y_n \left \langle \left( f^{y,t}_{2;1} \right)^2 \right \rangle \id_{n_-, n_+ \in J(t^{\xi_1}), n_{\bullet} \in J(t-2 t^{\xi_1}), n_- \leq n_+} \right] \lesssim t^{\xi_1 (\xi_3-\xi_2) (1-\frac{d}{2})}. 
$$
This implies~\eqref{eq:1_i_7} for $i = 1$ and $\theta < \xi_1 (\xi_3-\xi_2) (\tfrac{d}{2}-1)$.

\paragraph*{\textsc{Case} $i=2$.} Again, using~\eqref{eq:triple_A_bound}, we bound~\eqref{eq:f2_i_expansion} by 

\begin{align*}
		\sum_{r \in R^2} r &
		\sum_{\jbf = (j_1, \ldots, j_{r}) \in \N^{r}} 
		\sum_{\zbf}
			\left(q_{j_1}^{z_1} \right)^2 \ldots \left(q_{j_{r}}^{z_{r}} \right)^2 
			\psi(\beta)^{r} 
	\\
	&\lesssim \sum_{r > \frac{1}{3} v(l_-)} r (\alpha_d \psi(\beta))^r
	\lesssim  (\alpha_d \psi(\beta))^{\frac{1}{3} v(l_-)}\frac{1}{3} v(l_-)
	\lesssim  (\alpha_d \psi(\beta))^{C t^{\xi_1\xi_2}}
\end{align*}
for some constant $C>0$. Here, we used that $l_-\in J(t^{\xi_1})$.
From this estimate we deduce~\eqref{eq:1_i_7} for $i = 2$ and for any $\theta > 0$. 

\paragraph*{\textsc{Case} $i = 3$.} For $t$ sufficiently large, $l_- \leq l_+$ in $J(t^{\xi_1})$, $l_{\bullet} \in J(t-2t^{\xi_1})$, we have 
$$
r_{\bullet} + 1 \leq k(l_- + l_{\bullet} + l_+) < \nuone l_{\bullet}, \quad r \in R^3. 
$$
However, it is not true in general that $r_- + 1 < \nuone l_-$ and $r_+ + 1 < \nuone l_+$. Thus, Lemma~\ref{lm:A_estimate} only gives  
$$
\E \biggl[ \prod_{j=1}^r \left(e^{\beta^2 t_{i_j +1}} - 1 \right) \bigg\vert l_-, l_{\bullet}, l_+ \biggr] \lesssim \psi^{r_{\bullet}} A(t^{\xi_1},l_-,r_- +1) A(t^{\xi_1},l_+, r_+ + 1). 
$$
Consequently, for $i =3$, the expression in~\eqref{eq:f2_i_expansion} is less than a constant times
\begin{align}
\label{eq:i210}
	& \sum_{\substack{r_-, r_{\bullet}, r_+ \in \N_0, \\ r_- + r_{\bullet} + r_+ \in R^3}} 
		(r_- + r_{\bullet} + r_+) (\alpha \psi)^{r_- + r_{\bullet} + r_+} 
\\ \label{eq:i21}
	+& \sum_{\substack{r_- \leq l_-, \ r_+ \leq l_+, \ r_{\bullet} \in \N_0, \\ r_- \geq \nuone l_- \ \text{or} \ r_+ \geq \nuone l_+}} 
		\hspace{-3mm} (r_- + r_{\bullet} + r_+) 
		(\alpha \psi)^{r_{\bullet}} \\
		& \alpha^{r_- + r_+} 
		A(t^{\xi_1},l_-, r_- +1) A(t^{\xi_1},l_+,r_+ +1). \label{eq:i22} 
\end{align}
%
Notice that the expression in~\eqref{eq:i210} converges to $0$ as $t\to\infty$ faster than any polynomial by the same argument as in the case $i=2$. So the limit statement for this term is also yielded by~\eqref{eq:E_of_q_is_p}. Therefore, it remains only to consider the expression in~\eqref{eq:i21} and~\eqref{eq:i22}; that is, we need to show that for any $\theta > 0$, 
\begin{align}    
\lim_{t \to \infty} t^{\theta} \sup_{\|y\| \leq t^{\sublin}} \frac{1}{p_t^y} & \sum_{\substack{l_-, l_+ \in J(t^{\xi_1}), \\ l_- \leq l_+}} \sum_{\l_{\bullet} \in J(t-2 t^{\xi_1})}  q^y_{l_- + l_{\bullet} + l_+} \Pp(l_-, l_{\bullet}, l_+) \notag \\
\sum_{\substack{r_- \leq l_-, \ r_+ \leq l_+, \ r_{\bullet} \in \N_0, \\ r_- \geq \nuone l_- \ \text{or} \ r_+ \geq \nuone l_+}} 
		\hspace{-3mm} & (r_- + r_{\bullet} + r_+) 
		(\alpha \psi)^{r_{\bullet}} \notag \\
		& \alpha^{r_- + r_+} 
		A(t^{\xi_1},l_-, r_- +1) A(t^{\xi_1},l_+,r_+ +1) = 0.      \label{eq:R_sums_1} 
\end{align}
Recall from Section~\ref{sec:transition_prob} that for $y \in \Z^d$ and $n \in \N_0$, we define
$$
\iota(y,n) = \begin{cases}
                          n, & \ \|y\|_1 \equiv n, \\
                        n+1, & \ \|y\|_1 \not\equiv n.
\end{cases}
$$
By Lemma~\ref{lm:q_iota}, for $t$ sufficiently large, $y \in \Z^d$ such that $\|y\| \leq t^{\sublin}$, $l_-, l_+ \in J(t^{\xi_1})$, and $l_{\bullet} \in J(t-2t^{\xi_1})$,
$$
q^y_{l_- + l_{\bullet} + l_+} \lesssim q^y_{\iota(y,l_{\bullet})}.
$$
%
Also, note that in~\eqref{eq:R_sums_1} at least one of the conditions $r_- \geq \nuone l_-$ or  $r_+ \geq \nuone l_+$ must be satisfied. Without loss of generality, assume $r_- \geq \nuone l_-$  (the other case is argued similarly). Therefore, since $r_- + r_\bullet + r_+ \leq (r_- + 1)(r_\bullet + 1)(r_+ + 1)$, we can bound the expression in~\eqref{eq:R_sums_1} from above by
%
\begin{align}   \label{eq:R_sums_3_1}
	& \sup_{\|y\| \leq t^{\sublin}} \frac{1}{p_t^y} \ 
	e^{-\beta^2t^{\xi_1}}
	\sum_{l_{\bullet} \in J(t-2t^{\xi_1})} 
	q^y_{\iota(y,{l_\bullet})} 
	\Pp(l_{\bullet})  \\  
	\label{eq:R_sums_3_2}
	& \sum_{r_\bullet\in \N_0} (r_\bullet+1) (\alpha \psi)^{r_\bullet}  \\
	\label{eq:R_sums_3_3}
	&e^{\beta^2t^{\xi_1}} \hspace{-2mm} \sum_{\substack{l_- \in J(t^{\xi_1})}} 
	\sum_{\substack{\nuone l_- \leq r_- \leq l_-}}
	(r_-+1) \alpha^{r_-} \Pp(l_-, l) A(t^{\xi_1},l_-,r_-+1)  \\  
	\label{eq:R_sums_3_4}
	&\phantom{e^{\beta^2t^{\xi_1}} } \sum_{\substack{l_+ \in J(t^{\xi_1})}} 
	\sum_{\substack{0\leq r_+ \leq l_+}}
	(r_++1) \alpha^{r_+} \Pp(l_+, l) A(t^{\xi_1},l_+,r_++1). 
\end{align}
Here, note that we multiplied and divided by $e^{\beta^2t^{\xi_1}} $, so that the term in~\eqref{eq:R_sums_3_1} is finite in the limit by~(A0) in Lemma~\ref{lm:building_blocks}; the term in~\eqref{eq:R_sums_3_2} is bounded provided that $\alpha\psi <1$. Finally, the term in~\eqref{eq:R_sums_3_3} goes to 0 as $t\to \infty$ by~(A3) in Lemma~\ref{lm:building_blocks} and the term in~\eqref{eq:R_sums_3_4} goes to 0 as $t\to \infty$ by~(A7) in Lemma~\ref{lm:building_blocks}.

%
%
%


\subsubsection{{Proof of Claim~\ref{cl:one_huge}, Part 2}}

With regard to the convergence statement in~\eqref{eq:ibob_2}, we first note that by the triangle inequality it is enough to show that for $\beta$ sufficiently small and for any $\theta > 0$,  
\begin{equation}    \label{eq:ibob_2_1}
\lim_{t \to \infty} t^{\theta} \sup_{\|y\| \leq t^{\sublin}} \frac{1}{p_t^y} 
\left \langle \E \left[ q^y_n \left \lvert F_2^{y,t} \right \rvert 
\left(1 - \id_{n_-, n_+ \in J(t^{\xi_1}), n_{\bullet} \in J(t-2t^{\xi_1})} \right) \right] \right \rangle  = 0, 
\end{equation}
and 
\begin{align}     
&\lim_{t \to \infty} t^{\theta}  \sup_{\|y\| \leq t^{\sublin}} \frac{1}{p_t^y}  \notag \\
&\left \langle \E \left[ q^y_n \left \lvert (T_{0,0}^t -1) (T^{y,t}_0 -1) \right \rvert 
\left(1 - \id_{n_-, n_+ \in J(t^{\xi_1}), n_{\bullet} \in J(t-2t^{\xi_1})} \right) \right] \right \rangle = 0. \label{eq:ibob_2_2} 
\end{align}
We first show~\eqref{eq:ibob_2_1}, which will follow from  
\begin{equation} \label{eq:lim_F2_sq}
\lim_{t \to \infty} t^{\theta} \sup_{\|y\| \leq t^{\sublin}} \frac{1}{p_t^y} \E \left[ q^y_n \left \langle \left(F_2^{y,t} \right)^2 \right \rangle \left(1-\id_{n_-, n_+ \in J(t^{\xi_1}), n_{\bullet} \in J(t-2t^{\xi_1})} \right) \right] = 0 
\end{equation}
for any $\theta > 0$. 
From~\eqref{eq:E_gen_f}, for given families of sets $S_-(t), S_{\bullet}(t), S_+(t) \subset \N_0$, we have 
\begin{align*}
& \E \left[ q^y_n \left \langle \left(F_2^{y,t} \right)^2 \right \rangle \id_{n_s \in S_s(t), s \in \oplus} \right] \\
=& \sum_{l_s \in S_s(t), s \in \oplus} q^y_{l_- + l_{\bullet} + l_+} \Pp(l_-, l_{\bullet}, l_+) \E \left[ \left \langle \left(F_2^{y,t} \right)^2 \right \rangle \bigg\vert l_-, l_{\bullet}, l_+ \right]. 
\end{align*}
Also, \eqref{eq:f_i_expansion} and \eqref{eq:triple_A}, imply that for $l_-, l_{\bullet}, l_+ \in \N_0$,  
\begin{align*}
& \E \left[ \left \langle \left(F_2^{y,t} \right)^2 \right \rangle \bigg\vert l_-, l_{\bullet}, l_+ \right] \\
\lesssim& \hspace{-2mm}\sum_{0 \leq r_s \leq l_s, s \in \oplus} (r_- + r_{\bullet} + r_+)  
 \alpha^{r_- + r_{\bullet} + r_+} \\
 &A(t^{\xi_1},l_-, r_- +1) A(t-2t^{\xi_1}, l_{\bullet}, r_{\bullet} +1) A(t^{\xi_1},l_+, r_+ +1) \\
\lesssim& \sum_{0 \leq r_\bullet \leq l_{\bullet}} (r_\bullet + 1) \alpha^{r_\bullet} A(t-2t^{\xi_1},l_{\bullet}, r_\bullet+1) \\
&\prod_{s \in \{-,+\}} \sum_{0 \leq r_s \leq l_s} (r_s +1) \alpha^{r_s} A(t^{\xi_1},l_s,r_s +1). 
\end{align*} 
By Lemma~\ref{lm:lclt_con_time}, for $t$ sufficiently large $1/p_t^y \lesssim e^{t^\sigma}$ for $\| y \| \leq t^\sigma$. Therefore,
\begin{align*}
\sup_{\|y\| \leq t^{\sublin}} \frac{1}{p_t^y} & \E \left[ q^y_n \left \langle \left(F_2^{y,t} \right)^2 \right \rangle \id_{n_s \in S_s(t), s \in \oplus} \right] \\
\lesssim e^{t^{\sublin}} & \sum_{l \in S_{\bullet}(t)} \Pp^{\bullet}(l) \sum_{0 \leq r \leq l} (r + 1) \alpha^{r} A(t-2t^{\xi_1},l,r + 1) \\
& \prod_{s \in \{-,+\}} \sum_{l_s \in S_s(t)} \Pp^s(l_s) \sum_{0 \leq r_s \leq l_s} (r_s +1) \alpha^{r_s} A(t^{\xi_1},l_s, r_s +1).  
\end{align*}  

In order to complete the proof of~\eqref{eq:ibob_2_1}, we will consider two main cases: (1) $S_{\bullet}(t)$ is the complement of $J(t-2t^{\xi_1})$ and (2) $S_{\bullet}(t)$ is $J(t-2t^{\xi_1})$.

\medskip

\paragraph*{\textsc{Case 1.}} If $S_{\bullet}(t)$ is the complement of $J(t-2t^{\xi_1})$, then~(A5)  in Lemma~\ref{lm:building_blocks} yields 
$$
\lim_{t \to \infty} t^{\theta} e^{t^{\sublin}} \sum_{l \in S_{\bullet}(t)} \Pp^{\bullet}(l) \sum_{0 \leq r \leq l} (r+1) \alpha^r A(t-2t^{\xi_1}, l, r+1) = 0, \quad \theta > 0. 
$$
Then either of two possibilities occur:\\
If either $S_-$ or $S_+$ is the complement of $J(t^{\xi_1})$, then (A6)  in Lemma~\ref{lm:building_blocks} implies
		$$ 
		\lim_{t \to \infty} \sum_{l \notin J(t^{\xi_1})} \Pp^-(l) \sum_{0 \leq r \leq l} (r+1) \alpha^r A(t^{\xi_1},l,r+1) = 0. 
		$$ 
If either $S_-$ or $S_+$ is of $J(t^{\xi_1})$, (A7) implies
		$$ 
		\lim_{t \to \infty} \sum_{l \in J(t^{\xi_1})} \Pp^-(l) \sum_{0 \leq r \leq l} (r+1) \alpha^r A(t^{\xi_1},l,r+1) < \infty. 
		$$ 	
Therefore,
$$
\lim_{t \to \infty} t^{\theta} \sup_{\|y\| \leq t^{\sublin}} \frac{1}{p_t^y} \E \left[ q^y_n \left \langle \left(F_2^{y,t} \right)^2 \right \rangle \id_{n_{\bullet} \notin J(t-2t^{\xi_1})} \right] = 0, \quad \theta > 0. 
$$

\paragraph*{\textsc{Case 2.}} If $S_\bullet = J(t-2t^{\xi_1})$, then
by Lemma~\ref{lm:ql+._},
for $\|y \| \leq t^{\sublin}$, $l_{\bullet} \in S_\bullet(t)$, and $l_-, l_+ \in \N_0$, we have 
\begin{equation}    \label{eq:q_exp_bound} 
q^y_{l_- + l_{\bullet} + l_+} \lesssim \prod_{s \in \{-,+\}} \exp \left(C t^{\sublin -1} l_s \right) q^y_{\iota(y,l_{\bullet})}.
\end{equation} 
Therefore, 
\begin{align*}
	& \sup_{\|y\| \leq t^{\sublin}} \frac{1}{p_t^y} \sum_{l_s \in S_s(t), s \in \{-,+\}} \sum_{l_{\bullet} \in J(t-2t^{\xi_1})} q^y_{l_- + l_{\bullet} + l_+} \Pp(l_-, l_{\bullet}, l_+)  \\ 
	& \sum_{0 \leq r \leq l_{\bullet}} (r +1) \alpha^{r} A(t-2t^{\xi_1},l_{\bullet},r +1) \prod_{s \in \{-,+\}} \sum_{0 \leq r_s \leq l_s} (r_s +1) \alpha^{r_s} A(t^{\xi_1},l_s,r_s +1) \\
	\lesssim& H_1(t) H_2(t), 
	\end{align*} 
where 
$$
H_1(t) := e^{-\beta^2 t^{\xi_1}} \sup_{\|y\| \leq t^{\sublin}} \frac{1}{p_t^y} \sum_{l \in J(t-2t^{\xi_1})} q^y_{\iota(y,l)} \Pp^{\bullet}(l) \sum_{0 \leq r \leq l} (r +1) \alpha^{r} A(t-2t^{\xi_1},l,r +1)
$$
and 
\begin{align*} 
H_2(t) :=& e^{\beta^2 t^{\xi_1}} \prod_{s \in \{-,+\}} \sum_{l_s \in S_s(t)} \\
&\exp \left(C t^{\sublin -1} l_s \right) \Pp^s(l_s) \sum_{0 \leq r_s \leq l_s} (r_s +1) \alpha^{r_s} A(t^{\xi_1},l_s,r_s +1). 
\end{align*} 

Notice $\lim_{t \to \infty} H_1(t) < \infty$ by (A1) in Lemma~\ref{lm:building_blocks}. To deal with $H_2(t)$, first note that at least one of $S_-$ or $S_+$ must be the complement of $J(t^{\xi_1})$; without loss of generality, assume it is $S_-(t)$, in which case $S_+(t)$ is allowed to be $J(t^{\xi_1})$ or its complement. Then, (A4) in Lemma~\ref{lm:building_blocks} implies that the factor in $H_2(t)$ corresponding to $s=-$ satisfies
$$
\lim_{t \to \infty} t^{\theta} e^{\beta^2 t^{\xi_1}} \sum_{l \notin J(t^{\xi_1})} e^{c t^{\sublin -1} l} \Pp^-(l) \sum_{0 \leq r \leq l} (r+1) \alpha^r A(t^{\xi_1},l,r+1) = 0.
$$
If $S_+(t)$ is the complement of $J(t^{\xi_1})$, then the factor in $H_2(t)$ corresponding to $s=+$ also satisfies the above. However, if $S_+(t)=J(t^{\xi_1})$, then $\exp(Ct^{\sigma - 1} l_+) \leq \exp(C't^{\sigma - 1 + \xi_1})$ which remains bounded because $\sigma - 1 + \xi_1 < 0$. Finally, (A7) implies
$$ 
\limsup_{t \to \infty} \sum_{l \in J(t^{\xi_1})} \Pp^{-}(l) \sum_{0 \leq r < l } (r+1) \alpha^r A(t^{\xi_1},l,r+1) < \infty. 
$$
Therefore, we also have 
$$
\lim_{t \to \infty} t^{\theta} \sup_{\|y\| \leq t^{\sublin}} \frac{1}{p_t^y} \E \left[ q^y_n \left \langle \left(F_2^{y,t} \right)^2 \right \rangle \id_{n_{\bullet} \in J(t-2t^{\xi_1})} \left(1 - \id_{n_-, n_+ \in J(t^{\xi_1})} \right) \right] = 0, \quad \theta > 0. 
$$
This completes the proof of~\eqref{eq:ibob_2_1}. 

\medskip

For the convergence statement in~\eqref{eq:ibob_2_2}, we only need to note that by the Cauchy--Schwarz inequality, Fubini Theorem, and symmetry, 
\begin{align*}
& \left \langle \E \left[q^y_n \left \lvert (T_{0,0}^t -1) (T^{y,t}_0 -1) \right \rvert \left(1-\id_{n_-, n_+ \in J(t^{\xi_1}), n_{\bullet} \in J(t-2t^{\xi_1})} \right) \right] \right \rangle \\
\leq& \E \left[ q^y_n \left \langle \left(T_{0,0}^t -1 \right)^2 \right \rangle \left(1 - \id_{n_-, n_+ \in J(t^{\xi_1}), n_{\bullet} \in J(t-2t^{\xi_1})} \right) \right]. 
\end{align*}
The rest of the proof can be carried out in full analogy to the proof of~\eqref{eq:lim_F2_sq}.

\subsection{Proof of Lemma~\ref{lm:main_claim1-4} Part 2 and 3: Convergence for one huge gap at the start or the end}
\label{ssec:gap_end}

We only show the convergence statement in~\eqref{eq:conv_3} as the proof of~\eqref{eq:conv_2} is analogous. We write 
$$
F_3^{y,t} - T_{0,0}^t + 1 = f_{3;1}^t  + f_{3;2}^t  
$$
where for $i= 1,2$,
\begin{align}    \label{eq:p6_3}
	f_{3;i}^t \coleq &
	\sum_{r \in R^i} \sum_{\substack{\ibf \in H_{r,n}^i, \zbf}} 
	q_r(\ibf,\zbf)
	\prod_{j=1}^r h(z_j; s_{i_j}, s_{i_j +1}) 
\end{align}
where
$q_r(\ibf,\zbf) \coleq q_{n, \widehat{r+1}}^y = q_{i_1}^{z_1} \ldots q_{i_r - i_{r-1}}^{z_r - z_{r-1}}$,\\
$R^1 \coleq \{1,\ldots, v(n_-)+1 \}$, 
$R^2 \coleq \{v(n_-)+1, \ldots k(n) \}$, \\
$H_{r,n}^1 \coleq 
		\left\{ \ibf = (i_1, \ldots, i_r)  ~:~
			\substack{ \displaystyle 
				0 \leq i_1 < \ldots < i_{r} \leq rn^{\xi} 
			\\	\displaystyle
				i_r > w(n_-) 
			}
			\right\},
$
$
H_{r,n}^2 \coleq I_1(n,r,r+1).
$\\

\smallskip
By Jensen's inequality, it is enough to prove the following claim.

\begin{claim} \label{cl:one_huge_start}
For $\beta >0$ sufficiently small, there is $\theta > 0$ such that 
for $i \in \{1,2\}$,  
\begin{align}   
\lim_{t \to \infty} t^{\theta} \sup_{\|y\| \leq t^{\sublin}} \frac{1}{p_t^y} \E \left[ q^y_n \left \langle \left(f_{3;i}^t \right)^2 \right \rangle \id_{n_- \in J(t^{\xi_1}), \ n_{\bullet} + n_+ \in J(t-t^{\xi_1})} \right] =& 0, \label{eq:ib_J_I} \\
\lim_{t \to \infty} t^{\theta} \sup_{\|y\| \leq t^{\sublin}} \frac{1}{p_t^y} \E \left[ q^y_n \left \langle \left(f_{3;i}^t \right)^2 \right \rangle \left(1 - \id_{n_- \in J(t^{\xi_1}), n_{\bullet} + n_+ \in J(t-t^{\xi_1})} \right) \right] =& 0. \label{eq:ib_J_ob_I}
\end{align}
\end{claim}

Let $i \in \{1,2\}$. Conditioning on the number of jumps in $[0,t^{\xi_1})$ and $[t^{\xi_1},t)$, we write
\begin{align}
& \E \left[ q^y_n \left \langle \left(f_{3;i}^t \right)^2 \right \rangle \id_{n_- \in J(t^{\xi_1}), n_{\bullet} + n_+ \in J(t-t^{\xi_1})} \right] \notag \\
=& \sum_{l_- \in J(t^{\xi_1})} \sum_{l_{\oplus} \in J(t-t^{\xi_1})} q^y_{l_- +l_{\oplus}} \notag \\
& \Pp^{-,(\bullet,+)}(l_-,l_{\oplus}) \ \E^{-,(\bullet,+)}\left[ \left \langle \left(f_{3;i}^t \right)^2 \right \rangle \big\vert l_-,l_{\oplus} \right]. \label{eq:exp_Ef3} 
\end{align} 
For $l_-, l_{\oplus} \in \N_0$, 
\begin{align}    
&\E^{-,(\bullet,+)} \left[ \left \langle \left(f_{3;i}^t \right)^2 \right \rangle \big\vert l_-, l_{\oplus} \right] \notag \\
=& \sum_{r \in R^i} \sum_{\ibf \in I^i_r, \zbf} q_r(\ibf, \zbf)^2 \E^{-,(\bullet,+)} \biggl[\prod_{j=1}^r \left(e^{\beta^2 t_{i_j +1}} - 1 \right) \bigg\vert l_-,l_{\oplus} \biggr],  \label{eq:exp_val_bd} 
\end{align}
and for $r \in R^i$, and $\ibf \in H_{r,n}^i$, let 
$$
r_- = \left \lvert \left\{1 \leq j \leq r: i_j < l_- \right\} \right \rvert, \quad r_{\oplus} = r - r_-. 
$$
Observe that 
\begin{align}   
\E^{-,(\bullet,+)} & \biggl[ \prod_{j=1}^r \left(e^{\beta^2 t_{i_j +1}} - 1 \right) \bigg\vert l_-, l_{\oplus} \biggr] \\
\leq& A(t^{\xi_1}, l_-, r_- +1) A(t-t^{\xi_1},l_{\oplus}, r_{\oplus}+1). \label{eq:double_A} 
\end{align}

\subsubsection{{Proof of Claim~\ref{cl:one_huge_start}, Part 1}}
We will consider the cases $i=1$ and $i=2$ separately.

\paragraph*{\textsc{Case} $i=1$.} In order to show~\eqref{eq:ib_J_I} for $i=1$, we will show that the term
$ \E [  \langle (f_{3;i}^{t} )^2  \rangle | l_-,l_{\oplus} ]
$
can be bounded by a function that goes to 0 as $t \to \infty$ at least polynomially fast; then the fact that 
\begin{equation} \label{eq:ql.+_pty}
\sum_{\substack{l_s\in S_s(t) \\ l_-\leq l_+ }} q^y_{l_- + l_{\bullet} + l_+} \Pp(l_-,l_{\bullet},l_+) \leq p_t^y
\end{equation}
yields the limit in~\eqref{eq:ib_J_I}.\\
For $t$ sufficiently large, $l_- \in J(t^{\xi_1}), l_{\oplus} \in J(t-t^{\xi_1})$, we have for any $r \in R^1$ 
$$ 
r+1 \leq v(l_-) + 2  < \nuone l_{-},
$$ 
so in particular
$
r_s + 1 < \nuone l_s
$
for  $s \in \{-,\oplus\}$. 
By Lemma~\ref{lm:A_estimate}, the expression on the right side of~\eqref{eq:double_A} is less than a constant times 
$$
\psi^{r_- + r_{\oplus}}  = \psi^r. 
$$
As a result, for $i=1$, the expression in~\eqref{eq:exp_val_bd} is less than a constant times

    \begin{align*}
    \sum_{r \in R^1} \sum_{\substack{\ibf \in H_{r,m}^1, \zbf}} 
    \psi^r
    q_r(\ibf,\zbf)^2
    \lesssim &
    \sum_{1\leq r\leq v(n_-)+1} 
    \psi^r
    \sum_{\substack{j_1, \ldots, j_r \in \N, \\ j_1 + \ldots + j_r > w(l_-)}} \sum_{c_1, \ldots, c_r \in \Z^d} \left(q_{j_1}^{c_1} \right)^2 \ldots \left(q_{j_r}^{c_r} \right)^2 
    \\
       \lesssim &
     \sum_{1\leq r\leq v(n_-)+1} 
     \psi^r
    \sum_{l=1}^r \sum_{\substack{j_1, \ldots, j_r \in \N, \\ j_l \geq \frac{1}{2} \frac{w(l_-)}{v(l_-)}}} \sum_{k=1}^r \biggl(\sum_{c_k \in \Z^d} \left(q_{j_k}^{c_k} \right)^2 \biggr)
    \\
    \lesssim &
    \sum_{1\leq r\leq \infty}   r  (\alpha\psi)^r
    \sum_{j > \frac{1}{2}\frac{w(l_-)}{v(l_-)} }
    \frac{1}{j^{d/2}}\\
    \lesssim &
    \left(\frac{w(l_-)}{v(l_-)} \right)^{1-\frac{d}{2}}. 
    \end{align*}
 
Since 
$$ 
\frac{w(l_-)}{v(l_-)} \gtrsim l_-^{\xi_3-\xi_2} \gtrsim t^{\xi_1 (\xi_3-\xi_2)},  
$$ 
this implies~\eqref{eq:ib_J_I} for $i=1$ and $\theta < \xi_1 (\xi_3-\xi_2) (\tfrac{d}{2}-1)$. 

\medskip

\paragraph*{\textsc{Case} $i=2$.}
For $l_- \in J(t^{\xi_1})$ and $l_{\oplus} \in J(t-t^{\xi_1})$, 
$$ 
r+1 \leq k(l_- +l_{\oplus}) + 1 < \nuone l_{\oplus}, \quad r \in R^2, 
$$ 
as long as $t$ is sufficiently large. Then, 
$$ 
\E^{-,(\bullet,+)} \biggl[ \prod_{j=1}^r \left(e^{\beta^2 t_{i_j +1}} - 1 \right) \bigg\vert l_-, l_{\oplus} \biggr] \lesssim \psi^{r_{\oplus}} A(t^{\xi_1},l_-,r_- +1). 
$$ 
Consequently, for $i=2$, the expression in~\eqref{eq:exp_val_bd} is less than a constant times 
$$
\sum_{\substack{0 \leq r_- \leq l_-, r_{\oplus} \geq 0, \\ v(l_-) + 1 < r_- + r_{\oplus}}} (\alpha \psi)^{r_{\oplus}} A(t^{\xi_1},l_-,r_- +1)  \alpha^{r_-}. 
$$
If $r_- + 1 < \nuone l_-$, we have as before 
$ A(t^{\xi_1},l_-,r_- +1) \lesssim (\psi)^{r_-}, $ 
so the term in~\eqref{eq:exp_val_bd} is less than a constant times
\begin{align*}
 \sum_{r > (v(l_-)+1)/2} (\alpha \psi)^r \lesssim (\alpha\psi)^{(v(l_-)+1)/2}
\end{align*}   
For $l_- \in J(t^{\xi_1})$, we have that
$
v(l_-) \gtrsim l_-^{\xi_2} \gtrsim t^{\xi_1 \xi_2}
$.
Therefore, for $\alpha\psi<1$, the estimate in~\eqref{eq:ql.+_pty} yields

\begin{align*}
\lim_{t \to \infty} t^{\theta} \sup_{\|y\| \leq t^{\sublin}} \frac{1}{p_t^y} 
\hspace{-1mm}
\sum_{\substack{l_- \in J(t^{\xi_1}) \\ l_{\oplus} \in J(t-t^{\xi_1})}}
\hspace{-3mm}
q^y_{l_- + l_{\oplus}} \Pp^{-,(\bullet,+)}(l_-, l_{\oplus}) 
\sum_{\substack{0 \leq r_- < \nuone l_- \\ r_- + r_\oplus \in R^2 }} (\alpha\psi)^{r_- + r_\oplus}
=0.
\end{align*}
To complete the proof of~\eqref{eq:ib_J_I} for $i=2$, it only remains to consider the case $r_- + 1 \geq \nuone l_-$: We show that for any $\theta > 0$, 
\begin{align}    
\lim_{t \to \infty} t^{\theta} & \sup_{\|y\| \leq t^{\sublin}} \frac{1}{p_t^y} \sum_{l_-\in J(t^{\xi_1})} \sum_{l_{\oplus} \in J(t-t^{\xi_1})} q^y_{l_- + l_{\oplus}} \Pp^{-,(\bullet,+)}(l_-,l_{\oplus})  \notag \\
& \sum_{\nuone l_- - 1 \leq r_- \leq l_-, r_{\oplus} \in \N_0} (\alpha \psi)^{r_{\oplus}} \alpha^{r_-} A(t^{\xi_1},l_-,r_- +1) = 0.  \label{eq:sums_1}
\end{align}
By Lemma~\ref{lm:q_iota},  for $t$ sufficiently large, $y \in \Z^d$ such that $\|y\| \leq t^{\sublin}$, $l_- \in J(t^{\xi_1})$, and $l_{\oplus} \in J(t-t^{\xi_1})$, we have that
$
q^y_{l_- + l_{\oplus}} \lesssim q^y_{\iota(y,l_{\oplus})}.
$
Therefore, 
\begin{align*}
	& \sup_{\|y\| \leq t^{\sublin}} \frac{1}{p_t^y} \sum_{l_- \in J(t^{\xi_1})} \sum_{l_{\oplus} \in J(t-t^{\xi_1})} q^y_{l_- + l_{\oplus}} \Pp^{-,(\bullet,+)}(l_-, l_{\oplus}) \\
	& \sum_{r \in \N_0} (\alpha \psi)^r \sum_{\nuone l_- - 1 \leq r_- \leq l_-} (r_- +1) \alpha^{r_-} A(t^{\xi_1},l_-,r_- +1) \\
\lesssim& \sup_{\|y\| \leq t^{\sublin}} \frac{1}{p_t^y} \sum_{l_{\oplus} \in J(t-t^{\xi_1})} q^y_{\iota(y,l_{\oplus})} \Pp^{(\bullet,+)}(l_{\oplus}) \\
	&\sum_{l_- \in J(t^{\xi_1})} \Pp^-(l_-) \sum_{\nuone l_- - 1 \leq r_- \leq l_-} \alpha^{r_-} A(t^{\xi_1},l_-,r_- +1).
\end{align*}
and convergence to 0 as $t \to \infty$ follows from~(A0) and~(A3) in Lemma~\ref{lm:building_blocks}. 

\subsubsection{{Proof of Claim~\ref{cl:one_huge_start}, Part 2}}
Notice that by~\eqref{eq:exp_val_bd} and~\eqref{eq:double_A}, for $l_-, l_{\oplus} \in \N_0$, we have
\begin{align*}
&\E^{-,(\bullet,+)} \left[ \left \langle \left(f_{3;i}^t \right)^2 \right \rangle \bigg\vert l_-, l_{\oplus} \right] \\
\lesssim& \sum_{0 \leq r_- \leq l_-} \alpha^{r_-} A(t^{\xi_1},l_-, r_- +1) \sum_{0 \leq r_{\oplus} \leq l_{\oplus}} \alpha^{r_{\oplus}} A(t-t^{\xi_1},l_{\oplus},r_{\oplus} +1). 
\end{align*}
Therefore, since Lemma~\ref{lm:lclt_con_time} implies that $1/p_t^y \lesssim e^{t^\sigma}$ for $\|y\| \leq t^{\sigma}$, using the expansion in~\eqref{eq:exp_Ef3}, we obtain the following bound
\begin{align*}
& \sup_{\|y\| \leq t^{\sublin}} \frac{1}{p_t^y} \E \left[ q^y_n \left \langle \left(f_{3;i}^t \right)^2 \right \rangle \id_{n_s \in S_s(t), s \in \{-, \oplus\}} \right] \\
\lesssim& e^{t^{\sublin}} \sum_{l_- \in S_-(t)} \Pp^-(l_-) \sum_{0 \leq r_- \leq l_-} \alpha^{r_-} A(t^{\xi_1},l_-,r_-+1) \\
& \sum_{l_{\oplus} \in S_{\oplus}(t)} \Pp^{(\bullet,+)}(l_{\oplus}) \sum_{0 \leq r_{\oplus} \leq l_{\oplus}} \alpha^{r_{\oplus}} A(t-t^{\xi_1},l_{\oplus},r_{\oplus}+1). 
\end{align*}

We now consider two cases: (1) $S_{\oplus}(t)$ is the complement of $J(t-t^{\xi_1})$ and (2) $S_{\oplus}(t)$ is $J(t-t^{\xi_1})$.

\paragraph*{\textsc{Case 1.}} If $S_{\oplus}(t)$ is the complement of $J(t-t^{\xi_1})$,~(A5)  in Lemma~\ref{lm:building_blocks} yields 
$$
\lim_{t \to \infty} t^{\theta} e^{t^{\sublin}} \sum_{l \in S_{\oplus}(t)} \Pp^{(\bullet,+)}(l) \sum_{0 \leq r \leq l} \alpha^r A(t-t^{\xi_1},l,r+1) = 0, \quad \theta > 0.
$$
Notice that $S_-(t)$ is allowed to be either $J(t^{\xi_1})$ or its complement. If $S_-(t)$ is the complement of $J(t^{\xi_1})$, (A6) in Lemma~\ref{lm:building_blocks} yields
$$
\lim_{t \to \infty} \sum_{l \notin J(t^{\xi_1})} \Pp^-(l) \sum_{0 \leq r \leq l} \alpha^r A(t^{\xi_1},l,r+1) = 0. 
$$
If on the contrary, $S_-(t)=J(t^{\xi_1})$, by (A7) in Lemma~\ref{lm:building_blocks} we have

$$
\limsup_{t \to \infty} \sum_{l \in J(t^{\xi_1})} \Pp^-(l) \sum_{0 \leq r \leq l} \alpha^r A(t^{\xi_1},l,r+1) < \infty. 
$$

Therefore, this establishes for $i \in \{1,2\}$ and any $\theta > 0$:
$$
\lim_{t \to \infty} t^{\theta} \sup_{\|y\| \leq t^{\sublin}} \frac{1}{p_t^y} \E \left[q^y_n \left \langle \left( f_{3;i}^t \right)^2 \right \rangle \id_{n_{\bullet} + n_+ \notin J(t-t^{\xi_1})} \right] = 0. 
$$

\paragraph*{\textsc{Case 2.}} If $S_{\oplus}(t)$ is $J(t-t^{\xi_1})$, then $S_-(t)$ must be the complement of $J(t^{\xi_1})$. Also, notice that by Lemma~\ref{lm:ql+._}, for $\|y\| \leq t^{\sublin}$, $l_{\oplus} \in J(t-t^{\xi_1})$, and $l_- \in \N_0$, we have
$$
q^y_{l_- + l_{\oplus}} \lesssim \exp \left(C t^{\sublin -1} l_- \right) q^y_{\iota(y,l_{\oplus})}.
$$
Hence, 
\begin{align*}
& \sup_{\|y\| \leq t^{\sublin}} \frac{1}{p_t^y} \sum_{l_{\oplus} \in J(t-t^{\xi_1})} \sum_{l_- \in S_-(t)} q^y_{l_- + l_{\oplus}} \Pp^{-,(\bullet,+)}(l_-, l_{\oplus}) \\
& \sum_{0 \leq r_{\oplus} \leq l_{\oplus}} \alpha^{r_{\oplus}} A(t-t^{\xi_1},l_{\oplus}, r_{\oplus}+1) \sum_{0 \leq r_- \leq l_-} \alpha^{r_-} A(t^{\xi_1},l_-,r_-+1) \\
\lesssim& e^{-\beta^2 t^{\xi_1}} \sup_{\|y\| \leq t^{\sublin}} \frac{1}{p_t^y} \sum_{l_{\oplus} \in J(t-t^{\xi_1})} q^y_{\iota(y,l_{\oplus})} \Pp^{\oplus}(l_{\oplus}) \sum_{0 \leq r_{\oplus} \leq l_{\oplus}} \alpha^{r_{\oplus}} A(t-t^{\xi_1},l_{\oplus},r_{\oplus}+1) \\
& e^{\beta^2 t^{\xi_1}} \sum_{l_- \notin J(t^{\xi_1})} \exp \left(C t^{\sublin -1} l_- \right) \Pp^-(l_-) \sum_{0 \leq r_- \leq l_-} \alpha^{r_-} A(t^{\xi_1},l_-,r_-+1). 
\end{align*}
Then, by~(A1) and~(A4) in Lemma~\ref{lm:building_blocks} , we have for $i \in \{1,2\}$
$$
\lim_{t \to \infty} t^{\theta} \sup_{\|y\| \leq t^{\sublin}} \frac{1}{p_t^y} \E \left[ q^y_n \left \langle \left( f_{3;i}^t \right)^2 \right \rangle \id_{n_{\bullet} + n_+ \in J(t-t^{\xi_1}), n_- \notin J(t^{\xi_1})} \right] = 0, \quad \theta > 0. 
$$
This completes the proof of~\eqref{eq:ib_J_ob_I}.

\subsection{Proof of Lemma~\ref{lm:main_claim5}: Convergence to limiting partition functions}
\label{ssec:convergence_ergodicity}

Recall that we write $n$ for $n_t$, $n_-$ for $n_{t^{\xi_1}}$, $n_{\bullet}$ for $n_{t^{\xi_1}, t - t^{\xi_1}}$, and $n_+$ for $n_{t-t^{\xi_1},t}$. We also maintain the notational shorthands $\Pp^-$ $\Pp^{\bullet}$, etc., introduced at the beginning of Section~\ref{ssec:gap_end}. 

\begin{claim} \label{lm:final_conv_step_1}
For $\beta$ sufficiently small there is $\theta > 0$ such that 
$$ 
\lim_{t \to \infty} t^{\theta} \sup_{\|y\| \leq t^{\sublin}} \frac{1}{p_t^y} \left \langle \left \lvert \E \left[\left(q_n^y - q_{n_{\bullet}}^y \right) T^{y,t}_0 T_{0,0}^t \right] \right \rvert \right \rangle = 0. 
$$ 
$$ 
\lim_{t \to \infty} t^{\theta} \sup_{\|y\| \leq t^{\sublin}} \frac{1}{p_t^y} \left \langle \left \lvert \E \left[ q^y_{n_{\bullet}} T^{y,t}_0 T_{0,0}^t \right] - p_t^y Z^{y,t}_{-\infty} Z_{0,0}^{\infty} \right \rvert \right \rangle = 0. 
$$
\end{claim}

\subsubsection{{Proof of Claim~\ref{lm:final_conv_step_1}, Part 1}}
Let $\chi \in (0,\tfrac{1}{2} (1-\sublin))$ and $\Jtconst \in (\tfrac{1}{2},1)$. For $t > 0$ and $y \in \Z^d$, let $\chi_1(t)$ be the smallest even integer $\geq t (1- t^{-\chi})$, and let $\chi_2(t)$ be the largest odd integer $\leq t (1+t^{-\chi})$. Let 
$$ 
K(t) = \{l \in \N: \chi_1(t) \leq l \leq \chi_2(t)\}, \quad t > 0. 
$$ 
We will show that there is $\theta > 0$ such that 
\begin{align}    
\lim_{t \to \infty} t^{\theta} & \sup_{\|y\| \leq t^{\sublin}} \frac{1}{p_t^y} \notag \\
&\left \langle \left \lvert \E \left[ \left(q_n^y - q_{n_{\bullet}}^y \right) T^{y,t}_0 T_{0,0}^t \id_{n_- + n_+ \in J(2t^{\xi_1}), n_{\bullet} \in K(t- 2 t^{\xi_1})} \right] \right \rvert \right \rangle = 0,  \label{eq:fcv_1} 
\end{align}
and 
\begin{align}    
\lim_{t \to \infty} & t^{\theta}  \sup_{\|y\| \leq t^{\sublin}} \frac{1}{p_t^y} \notag \\
& \left \langle \left \lvert \E \left[ \left(q_n^y - q_{n_{\bullet}}^y \right) T^{y,t}_0 T_{0,0}^t \left(1 - \id_{n_- + n_+ \in J(2 t^{\xi_1}), n_{\bullet} \in K(t-2 t^{\xi_1})} \right)  \right] \right \rvert \right \rangle = 0.  \label{eq:fcv_2} 
\end{align}
Let us first show~\eqref{eq:fcv_1}.  
If $t$ is sufficiently large, a point $y \in \Z^d$ such that $\|y\| \leq t^{\sublin}$ can be connected to the origin by a path of length less than $(t-2 t^{\xi_1}) (1-(t-2 t^{\xi_1})^{-\chi})$. 
For such $t$, for $l_{\pm} \in J(2 t^{\xi_1})$, and $l_{\bullet} \in K(t-2 t^{\xi_1})$, 
it follows that 
$q_{l_{\bullet}}^y > 0$ if $l_{\bullet} \equiv \|y\|_1$, 
$q_{l_{\bullet}}^y = 0$ if $l_{\bullet} \not \equiv \|y\|_1$, 
$q^y_{l_{\bullet} +l_{\pm}} > 0$ if $l_{\bullet}+l_{\pm} \equiv \|y\|_1$, and 
$q^y_{l_{\bullet}+l_{\pm}} = 0$ if $l_{\bullet}+l_{\pm} \not \equiv \|y\|_1$. 
Thus,  
\begin{equation}     \label{eq:pfq_3}
\E \left[ (q^y_n - q^y_{n_{\bullet}}) T^{y,t}_0 T_{0,0}^t \id_{n_- + n_+ \in J(2 t^{\xi_1}), n_{\bullet} \in K(t- 2 t^{\xi_1})} \right]  = A(y,t) + B(y,t), 
\end{equation}
where 
\begin{align*}
A(y,t) \coleq & 
\sum_{\substack{l_{\pm} \in J(2 t^{\xi_1}), \\ l_{\pm} \equiv 0}} 
\sum_{\substack{l_{\bullet} \in K(t- 2 t^{\xi_1}), \\ l_{\bullet} \equiv \|y\|_1}} 
\left(q^y_{l_{\bullet}+l_{\pm}}-q^y_{l_{\bullet}} \right) 
\E[T^{y,t}_0 T_{0,0}^t \id_{n_{\bullet} = l_{\bullet}, n_- + n_+ = l_{\pm}}], \\
B(y,t) \coleq & 
\sum_{\substack{l_{\pm} \in J(2 t^{\xi_1}), \\ l_{\pm} \equiv 1}} 
\sum_{l_{\bullet} \in K(t- 2 t^{\xi_1})} \\
&\left(\id_{l_{\bullet} \not \equiv \|y\|_1} q^y_{l_{\bullet}+l_{\pm}} - \id_{l_{\bullet} \equiv \|y\|_1} q^y_{l_\bullet} \right) \E[T^{y,t}_0 T_{0,0}^t \id_{n_{\bullet} = l_{\bullet}, n_- + n_+ = l_{\pm}}].  
\end{align*}


We will show that there is $\theta > 0$ such that 
\begin{align}
\lim_{t \to \infty} t^{\theta} \sup_{\|y\| \leq t^{\sublin}} \frac{1}{p_t^y} \langle \lvert A(y,t) \rvert \rangle =& 0, \label{eq:pfq_5} \\
\lim_{t \to \infty} t^{\theta} \sup_{\|y\| \leq t^{\sublin}} \frac{1}{p_t^y} \langle \lvert B(y,t) \rvert \rangle =& 0. \label{eq:pfq_6} 
\end{align}
To do so, we first prove the following claim, which gives us an upper bound for the factor
$
\E[T^{y,t}_0 T_{0,0}^t \id_{n_{\bullet} = l_{\bullet}, n_- + n_+ = l_{\pm}}].
$
\begin{claim} \label{cl: E_T_bound}
\begin{equation*} 
\E[T^{y,t}_0 T_{0,0}^t \id_{n_{\bullet} = l_{\bullet}, n_- + n_+ = l_{\pm}}] 
\lesssim \sum_{m=0}^{l_{\pm}} 
 \Pp(m, l_{\bullet}, l_{\pm}-m) 
 \left( \hspace{-1mm}1 + \hspace{-2mm}
 \sum_{1 \leq r  \leq m+1} \hspace{-2mm} A(t^{\xi_1},m,r) \alpha^{r}\right)\hspace{-1mm}.
\end{equation*}
\end{claim}

\paragraph*{\textit{Proof.}}
Since $T^{y,t}_0 T_{0,0}^t$ and $n_{\bullet}$ are independent, we have that
\begin{equation}    \label{eq:A_upper_bound}
\E[T^{y,t}_0 T_{0,0}^t \id_{n_{\bullet} = l_{\bullet}, n_- + n_+ = l_{\pm}}]
=
\Pp^{\bullet, (-,+)}(l_{\bullet},l_{\pm})  \E^{(-,+)} \left[ \left \langle \left \lvert T^{y,t}_0 T_{0,0}^t \right \rvert \right \rangle \vert l_{\pm} \right].  
\end{equation}
By the Cauchy--Schwarz inequality and symmetry, 
$$ 
\E^{(-,+)} \left[ \left \langle \left \lvert T^{y,t}_0 T_{0,0}^t \right \rvert \right \rangle \vert l_{\pm} \right] \leq \E^{(-,+)} \left[ \left \langle (T_{0,0}^t)^2 \right \rangle \vert l_{\pm} \right]. 
$$ 
If we also condition on $n_-$, this implies that the right side of~\eqref{eq:A_upper_bound} is bounded from above by  
\begin{equation}  \label{eq:pfq_7}
\sum_{m=0}^{l_{\pm}} 
 \Pp(m, l_{\bullet}, l_{\pm}-m) \E^- \left[\left \langle (T_{0,0}^t)^2 \right \rangle \vert m \right].  
\end{equation} 
Now, for a fixed $m$, we have  
\begin{equation}    \label{eq:E_T_est}   
\E^-[ \langle (T_{0,0}^t)^2 \rangle \vert m] 
\leq 2 + 2 \sum_{r \in R} \sum_{\ibf \in H_r, \zbf} q_r(\ibf,\zbf)^2   
\E^- \biggl[ \prod_{j=1}^r \left(e^{\beta^2 t_{i_j +1}} - 1 \right) \bigg\vert m \biggr],  
\end{equation}
where 
\begin{align*}
R \coleq & \left\{1, \ldots, v(m) + 1 \right\},  \\
H_r \coleq & \left\{\ibf = (i_1, \ldots, i_r) \in \N_0^r: 0 \leq i_1 < \ldots < i_r \leq w(m) \right\}, \\
q_r(\ibf,\zbf) =& q_{i_1}^{z_1} \ldots q_{i_r - i_{r-1}}^{z_r - z_{r-1}}. 
\end{align*}

Since
$$
\E^- \biggl[ \prod_{j=1}^r \left(e^{\beta^2 t_{i_j +1}} -1 \right) \bigg\vert m \biggr] \leq A(t^{\xi_1},m,r),  
$$
we have 
$$
\sum_{r \in R} \sum_{\ibf \in I_r, \zbf} q_r(\ibf,\zbf)^2 \E^- \biggl[ \prod_{j=1}^r \left(e^{\beta^2 t_{i_j +1}} - 1 \right) \bigg\vert m \biggr]  
\lesssim 
\sum_{1 \leq r  \leq m+1} A(t^{\xi_1},m,r) \alpha^{r},
$$
which proves the claim.
\epf

Now we show~\eqref{eq:pfq_5}. By Claim~\ref{cl: E_T_bound}, we have that $A(y,t)$ is less than a constant times
\begin{align*}
\sum_{\substack{l_{\pm} \in J(2 t^{\xi_1}), \\ l_{\pm} \equiv 0}} 
\sum_{\substack{l_{\bullet} \in K(t- 2 t^{\xi_1}), \\ l_{\bullet} \equiv \|y\|_1}} &
\left(q^y_{l_{\bullet}+l_{\pm}}-q^y_{l_{\bullet}} \right) \notag \\
& \sum_{m=0}^{l_{\pm}} 
 \Pp(m, l_{\bullet}, l_{\pm}-m) 
 \left( \hspace{-1mm}1 + \hspace{-2mm}
 \sum_{1 \leq r  \leq m+1} \hspace{-2mm} A(t^{\xi_1},m,r) \alpha^{r}\right)\hspace{-1mm}. 
\end{align*}
We first show that 
\begin{equation}\label{eq:sum_w_const_term}
\lim_{t \to \infty} 
t^{\theta} \sup_{\|y\| \leq t^{\sublin}} \frac{1}{p_t^y}
\sum_{\substack{l_{\pm} \in J(2 t^{\xi_1}), \\ l_{\pm} \equiv 0}} 
\sum_{\substack{l_{\bullet} \in K(t- 2 t^{\xi_1}), \\ l_{\bullet} \equiv \|y\|_1}} 
\hspace{-2mm}
\left \lvert q^y_{l_{\bullet}+l_{\pm}}-q^y_{l_{\bullet}} \right\lvert
 \Pp(l_{\bullet}, l_{\pm}) 
 = 0.
\end{equation}

Using the fact that $q_{l_\bullet + l_\pm} > 0$, we can write
$$
\left\lvert q^y_{l_{\bullet}+l_{\pm}}-q^y_{l_{\bullet}} \right\lvert  = q^y_{l_{\bullet} + l_{\pm}} \left \lvert \frac{q^y_{l_{\bullet}}}{q^y_{l_{\bullet}+l_{\pm}}} - 1 \right \rvert .
$$
But as
$$
\sum_{\substack{l_{\pm} \in J(2 t^{\xi_1}), \\ l_{\pm} \equiv 0}} 
\sum_{\substack{l_{\bullet} \in K(t- 2 t^{\xi_1}), \\ l_{\bullet} \equiv \|y\|_1}} 
\hspace{-2mm}
q^y_{l_{\bullet}+l_{\pm}}
 \Pp(l_{\bullet}, l_{\pm}) \leq p_t^y,
$$
in order to prove~\eqref{eq:sum_w_const_term}, it is therefore enough to show that there exists a function $g$ such that $t^{\theta}g(t)$ converges to $0$ as $t\to\infty$ for some $\theta >0$ and which satisfies

\begin{equation}
\label{191126225309}
 \left \lvert \frac{q^y_{l_{\bullet}}}{q^y_{l_{\bullet}+l_{\pm}}} - 1 \right \rvert \leq g(t).
\end{equation}

By Lemma~\ref{lm:ratio_estimate}, there is $c > 0$ such that for $t$ sufficiently large, for $y \in \Z^d$ with $\|y\| \leq t^{\sublin}$, and for $l_{\bullet} \in K(t-2 t^{\xi_1})$, $l_{\pm} \in J(2 t^{\xi_1})$ with $q^y_{l_{\bullet}}, q^y_{l_{\bullet}+l_{\pm}} > 0$, we have
$$ 
\frac{q^y_{l_{\bullet}}}{q^y_{l_{\bullet}+l_{\pm}}} - 1 \leq \left(1 + O(t^{-\frac{2}{5}}) \right) \exp(c t^{\sublin + \xi_1 -1}) - 1
$$ 
and 
$$ 
1 - \frac{q^y_{l_{\bullet}}}{q^y_{l_{\bullet}+l_{\pm}}} = 1 - \frac{1}{q^y_{l_{\bullet}+l_{\pm}} / q^y_{l_{\bullet}}} \leq 1 - \left(1 +O(t^{-\frac{2}{5}}) \right) \exp(-c t^{\sublin + \xi_1 - 1}).
$$ 
Therefore, the inequality \eqref{191126225309} can be satisfied by
\begin{equation}\label{eq: g_funct}
g(t) \coleq \left(1 + O(t^{-\frac{2}{5}}) \right) \exp(c t^{\sublin +\xi_1-1}) - 1,
\end{equation}
and $\lim_{t \to \infty} t^{\theta} g(t) = 0$ for $\theta \in (0, 1 - \sublin - \xi_1)$. 

Let us now show that for $\beta$ sufficiently small, there is $\theta > 0$ such that for $S_-(t) = J(t^{\xi_1})$ and for $S_-(t)$ equal to the complement of $J(t^{\xi_1})$,
\begin{align}  
\lim_{t \to \infty} t^{\theta}  \sup_{\|y\| \leq t^{\sublin}} \frac{1}{p_t^y} \sum_{\substack{l_{\pm} \in J(2 t^{\xi_1}), \\ l_{\pm} \equiv 0}}  & \sum_{\substack{0 \leq m \leq l_{\pm}, \\ m \in S_-(t)}}  \sum_{\substack{l_{\bullet} \in K(t-2 t^{\xi_1}), \\ l_{\bullet} \equiv \|y\|_1}} \left \lvert q^y_{l_{\bullet}+l_{\pm}} - q^y_{l_{\bullet}} \right \rvert  \notag \\
& \Pp(m, l_{\bullet}, l_{\pm} - m) 
\sum_{1 \leq r  \leq m+1} A(t^{\xi_1},m,r) \alpha^{r}
= 0.  \label{eq:pfq_10} 
\end{align} 
Notice that if $g(t)$ is as in~\eqref{eq: g_funct}, then for $l_{\bullet} \in K(t-2 t^{\xi_1})$ such that $l_{\bullet} \equiv \|y\|_1$ and for $l_{\pm} \in J(2 t^{\xi_1})$ such that $l_{\pm} \equiv 0$, 
$$ 
\left \lvert q^y_{l_{\bullet}+l_{\pm}} - q^y_{l_{\bullet}} \right \rvert = q^y_{l_{\bullet}} \left \lvert \frac{q^y_{l_{\bullet}+ l_{\pm}}}{q^y_{l_{\bullet}}} - 1 \right \rvert \leq q_{l_{\bullet}}^y g(t).
$$
As a result, if $S_-$ is the complement of $J(t^{\xi_1})$, we obtain~\eqref{eq:pfq_10} from~(A0) and~(A4)   from Lemma~\ref{lm:building_blocks} as well as the following estimate:
\begin{align*}
& \sup_{\|y\| \leq t^{\sublin}} \frac{1}{p_t^y} \sum_{\substack{l_{\pm} \in J(2 t^{\xi_1}), \\ l_{\pm} \equiv 0}} \sum_{\substack{0 \leq m \leq l_{\pm}, \\ m \notin J(t^{\xi_1})}} \sum_{\substack{l_{\bullet} \in K(t-2 t^{\xi_1}), \\ l_{\bullet} \equiv \|y\|_1}} \left \lvert q^y_{l_{\bullet}+l_{\pm}} - q_{l_{\bullet}}^y \right \rvert \Pp(m, l_{\bullet}, l_{\pm} - m) \\
& \sum_{r \in R} \sum_{\ibf \in I_r, \zbf} q_r(\ibf,\zbf)^2 \E^- \biggl[ \prod_{j=1}^r \left(e^{\beta^2 t_{i_j +1}} - 1 \right) \bigg\vert m \biggr] \\
\lesssim& g(t) e^{-\beta^2 t^{\xi_1}} \sup_{\|y\| \leq t^{\sublin}} \frac{1}{p_t^y} \sum_{l \in J(t-2t^{\xi_1})} q_{\iota(y,l)}^y \Pp^{\bullet}(l) \\
&e^{\beta^2 t^{\xi_1}} \sum_{m \notin J(t^{\xi_1})} \Pp^-(m) 
\sum_{1 \leq r \leq m+1} \alpha^r A(t^{\xi_1},m,r). 
\end{align*}

If $S_-$ is $J(t^{\xi_1})$, consider two subcases: (1) $1\leq r < \nuone m$ and (2) $\nuone m \leq r \leq m + 1$.
In the first subcase, Lemma~\ref{lm:A_estimate} gives
\begin{align*}
& \sup_{\| y \| \leq t^{\sublin}} \frac{1}{p_t^y} \sum_{\substack{l_{\pm} \in J(2 t^{\xi_1}), \\ l_{\pm} \equiv 0}} \sum_{\substack{0 \leq m \leq l_{\pm}, \\ m \in J(t^{\xi_1})}} \sum_{\substack{l_{\bullet} \in K(t-2 t^{\xi_1}), \\ l_{\bullet} \equiv \|y\|_1}} \\
&\left \lvert q^y_{l_{\bullet} + l_{\pm}} - q^y_{l_{\bullet}} \right \rvert \Pp(m, l_{\bullet}, l_{\pm}-m) 
\sum_{1 \leq r < \nuone m} A(t^{\xi_1},m,r) \alpha^{r} \\
\lesssim& \sup_{\|y\| \leq t^{\sublin}} \frac{1}{p_t^y} \sum_{\substack{l_{\pm} \in J(2 t^{\xi_1}), \\ l_{\pm} \equiv 0}} \sum_{\substack{l_{\bullet} \in K(t-2 t^{\xi_1}), \\ l_{\bullet} \equiv \|y\|_1}} \left \lvert q^y_{l_{\bullet} + l_{\pm}}  - q^y_{l_{\bullet}} \right \rvert \Pp^{\bullet, (-,+)}(l_{\bullet}, l_{\pm}).
\end{align*}
The desired limit then follows from~\eqref{eq:sum_w_const_term} for $\theta \in (0, 1 - \sublin - \xi_1)$.

In the second subcase, we estimate as follows:

\begin{align*}
& \sup_{\|y\| \leq t^{\sublin}} \frac{1}{p_t^y} \sum_{\substack{l_{\pm} \in J(2 t^{\xi_1}), \\ l_{\pm} \equiv 0}} \sum_{\substack{0 \leq m \leq l_{\pm}, \\ m \in J(t^{\xi_1})}}  \sum_{\substack{l_{\bullet} \in K(t-2 t^{\xi_1}), \\ l_{\bullet} \equiv \|y\|_1}} \\
&\left \lvert q^y_{l_{\bullet}+l_{\pm}} - q^y_{l_{\bullet}} \right \rvert \Pp(m, l_{\bullet}, l_{\pm} - m) \sum_{\nuone m \leq r \leq m+1} A(t^{\xi_1},m,r) \alpha^{r} \\
\lesssim& g(t) e^{-\beta^2 t^{\xi_1}} \sup_{\|y\| \leq t^{\sublin}} \frac{1}{p_t^y} \sum_{l \in J(t-2t^{\xi_1})} q^y_{\iota(y,l)} \Pp^{\bullet}(l) \\
&e^{\beta^2 t^{\xi_1}} \sum_{m \in J(t^{\xi_1})} \Pp^-(m) \sum_{\nuone m \leq r \leq m+1} \alpha^r A(t^{\xi_1},m,r).
\end{align*} 
By~(A0) and~(A3), this tends to $0$ as $t \to \infty$ faster than $t^{-\theta}$ for any $\theta > 0$.
This completes the proof of~\eqref{eq:pfq_5}. 

   
Let us now show~\eqref{eq:pfq_6}, which will imply~\eqref{eq:fcv_1}. For convenience, we assume that $\|y\|_1 \equiv 0$. In the case $\|y\|_1 \equiv 1$, the proof proceeds analogously. We have 
$$
B(y,t) = B^{(1)}(y,t) + B^{(2)}(y,t), 
$$ 
where 
\begin{align*}
B^{(1)}(y,t) \coleq & \sum_{\substack{l_{\pm} \in J(2 t^{\xi_1}), \\ l_{\pm} \equiv 1}} \sum_{p = \frac{1}{2} \chi_1(t-2 t^{\xi_1})}^{\frac{1}{2} (\chi_2(t-2 t^{\xi_1})-1)} \\ 
&q^y_{2p+1+l_{\pm}} \E \left[ T^{y,t}_0 T_{0,0}^t \id_{n_- + n_+ = l_{\pm}}\left(\id_{n_{\bullet} = 2p+1} - \id_{n_{\bullet} = 2p} \right) \right], \\
B^{(2)}(y,t) \coleq & \sum_{\substack{l_{\pm} \in J(2 t^{\xi_1}), \\ l_{\pm} \equiv 1}} \sum_{p = \frac{1}{2} \chi_1(t-2 t^{\xi_1})}^{\frac{1}{2} (\chi_2(t-2 t^{\xi_1})-1)} \\
& \left(q^y_{2p+1+l_{\pm}} - q^y_{2p} \right) \E \left[ T^{y,t}_0 T_{0,0}^t \id_{n_{\bullet} = 2p, n_- + n_+ = l_{\pm}} \right].   
\end{align*} 
Following the proof of~\eqref{eq:pfq_5}, one shows 
$$
\lim_{t \to \infty} t^{\theta} \sup_{\|y\| \leq t^{\sublin}} \frac{1}{p_t^y} \left \langle \left \lvert B^{(2)}(y,t) \right \rvert \right \rangle = 0 
$$
for $\theta \in (0, 1 - \sublin - \xi_1)$. To establish~\eqref{eq:pfq_6}, it remains to show that there is $\theta > 0$ such that 
\begin{equation}     \label{eq:pfq_2_14} 
\lim_{t \to \infty} t^{\theta} \sup_{\|y\| \leq t^{\sublin}} \frac{1}{p_t^y} \left \langle \left \lvert B^{(1)}(y,t) \right \rvert \right \rangle = 0. 
\end{equation}

By Claim~\ref{cl: E_T_bound}, we have
\begin{align*}
\left \langle \left \lvert B^{(1)}(y,t) \right \rvert \right \rangle 
\leq \sum_{\substack{l_{\pm} \in J(2 t^{\xi_1}), \\ l_{\pm} \equiv 1}} \sum_{m=0}^{l_{\pm}} & \sum_{p=\frac{1}{2} \chi_1(t-2 t^{\xi_1})}^{\frac{1}{2} (\chi_2(t-2 t^{\xi_1})-1)} \\
q^y_{2p+1+l_{\pm}} &
\left \lvert \Pp^{\bullet}(2p+1) - \Pp^{\bullet}(2p) \right \rvert \\
&\Pp^{-,+}(m,l_{\pm}-m) 
\left( 1 + \sum_{1 \leq r \leq m+1} \alpha^r A(t^{\xi_1},m,r)\right). 
\end{align*}
Since $n_{\bullet}$ is Poisson distributed with intensity $t-2t^{\xi_1}$, 
$$ 
\Pp^{\bullet}(2p+1) - \Pp^{\bullet}(2p) = \Pp^{\bullet}(2p+1) \left(1 - \frac{2p+1}{t-2t^{\xi_1}} \right). 
$$
Therefore, for $\tfrac{1}{2} \chi_1(t-2 t^{\xi_1}) \leq p \leq \tfrac{1}{2} (\chi_2(t-2 t^{\xi_1})-1)$, we have
$$ 
\lvert \Pp^{\bullet}(2p+1) - \Pp^{\bullet}(2p) \lvert
\leq
\Pp^{\bullet}(2p+1) \hat g(t),$$ 
where 
$ 
\hat g(t) = \left(t-2 t^{\xi_1} \right)^{-\sigma}.
$
Since
$$
\sum_{\substack{l_{\pm} \in J(2 t^{\xi_1}), \\ l_{\pm} \equiv 1}} \sum_{p = \frac{1}{2} \chi_1(t-2 t^{\xi_1})}^{\frac{1}{2} (\chi_2(t- 2 t^{\xi_1})-1)} q^y_{2p+1+l_{\pm}} \Pp^{\bullet, (-,+)}(2p+1,l_{\pm}) \leq p_t^y, 
$$
we have 
$$
\lim_{t \to \infty} t^{\theta} \hat g(t) \sup_{\|y\| \leq t^{\sublin}} \frac{1}{p_t^y} \sum_{\substack{l_{\pm} \in J(2 t^{\xi_1}), \\ l_{\pm} \equiv 1}}  \sum_{p=\frac{1}{2} \chi_1(t-2 t^{\xi_1})}^{\frac{1}{2} (\chi_2(t- 2 t^{\xi_1})-1)} q^y_{2p+1+l_{\pm}} \Pp^{\bullet, (-,+)}(2p+1, l_{\pm}) = 0 
$$
for $\theta \in (0, \sigma)$. Since we also have that
$$
q^y_{2p+1+l_{\pm}}  \leq q^y_{2p+1} (1+g(t)) \lesssim q^y_{2p+1} = q_{\iota(y,2p+1)}, 
$$
we can complete the proof of~\eqref{eq:pfq_2_14} by following the reasoning for~\eqref{eq:pfq_5}. 

\bigskip 

Next, we show~\eqref{eq:fcv_2}.  Let $\| y \| \leq t^{\sigma}$, then by Lemma~\ref{lm:lclt_con_time}, we have that $1/p_t^y \leq e^{t^{\sigma}}$. For any $n, n_\bullet \in \N_0$,  we also have that $| q_n^y - q_{n_\bullet} | \leq 2$. Therefore, using Claim~\ref{cl: E_T_bound}, we have that for any given families of sets $S_{\bullet}(t), S_{\pm}(t) \subset \N_0$,
\begin{align}    
& \sup_{\|y\| \leq t^{\sublin}} \frac{1}{p_t^y} \left \langle \left \lvert \E \left[ \left(q^y_n - q^y_{n_{\bullet}} \right) T^{y,t}_0 T_{0,0}^t \id_{n_{\bullet} \in S_{\bullet}(t), n_- + n_+ \in S_{\pm}(t)} \right] \right \rvert \right \rangle \notag \\
\lesssim& e^{t^{\sublin}} \sum_{l_{\bullet} \in S_{\bullet}(t)} \Pp^{\bullet}(l_{\bullet}) \notag \\
&\sum_{l_{\pm} \in S_{\pm}(t)} \sum_{m=0}^{l_{\pm}} \Pp^{-,+}(m,l_{\pm}-m) \biggl(1 + \sum_{1 \leq r \leq m+1} \alpha^r A(t^{\xi_1},m,r) \biggr).    \label{eq:l_k_sum}
\end{align}


\noindent If $S_{\bullet}(t)$ is the complement of $K(t-2 t^{\xi_1})$, the following claim, which we prove in Appendix~\ref{ssec:building_blocks_proofs}, allows us to deal with the term in the second line of~\eqref{eq:l_k_sum}.
\begin{claim}\label{cl:sum_notin_K}
$$
\lim_{t \to \infty} t^{\theta} e^{t^{\sublin}} \sum_{l \notin K(t-2 t^{\xi_1})} \Pp^{\bullet}(l) = 0, \quad \theta > 0. 
$$
\end{claim}
\noindent In addition, the expression in the third line of~\eqref{eq:l_k_sum} is bounded from above by  
$$
1 + \sum_{m \in \N_0} \Pp^-(m) \sum_{1 \leq r \leq m+1} \alpha^r A(t^{\xi_1},m,r), 
$$
and~(A7) and~(A4)   imply 
$$
\limsup_{t \to \infty} \sum_{m \in \N_0} \Pp^-(m) \sum_{1 \leq r \leq m+1} \alpha^r A(t^{\xi_1},m,r) < \infty. 
$$
From this, we may already infer 
$$
\lim_{t \to \infty} t^{\theta} \sup_{\|y\| \leq t^{\sublin}} \frac{1}{p_t^y} \left \langle \left \lvert \E \left[ \left(q^y_n - q^y_{n_{\bullet}} \right) T^{y,t}_0 T_{0,0}^t \id_{n_{\bullet} \notin K(t-2 t^{\xi_1})} \right] \right \rvert \right \rangle = 0, \quad \theta > 0.  
$$
If $S_\bullet(t)$ is $K(t-2 t^{\xi_1})$, then $S_\pm(t)$ must be the complement of $J(2t^{\xi_1})$.
For $\|y\| \leq t^{\sublin}$, $l_{\bullet} \in K(t-2 t^{\xi_1}) \subset J(t-2t^{\xi_1})$, and $l_{\pm} \in \N_0$, we have as in~\eqref{eq:q_exp_bound}  
$$
q^y_{l_{\bullet}+l_{\pm}} \lesssim \exp \left(C t^{\sublin-1} l_{\pm} \right) q^y_{\iota(y,l_{\bullet})}. 
$$
Thus,
\begin{align}    \label{eq:second_to_last} 
& \sup_{\|y\| \leq t^{\sublin}} \frac{1}{p_t^y} \left \langle \left \lvert \E \left[ \left(q^y_n - q^y_{n_{\bullet}} \right) T^{y,t}_0 T_{0,0}^t \id_{n_{\bullet} \in K(t-2 t^{\xi_1}), n_- + n_+ \notin J(2 t^{\xi_1})} \right] \right \rvert \right \rangle  \\
\lesssim& e^{-\beta^2 t^{\xi_1}} \sup_{\|y\| \leq t^{\sublin}} \frac{1}{p_t^y} \sum_{l \in K(t-2 t^{\xi_1})} q^y_{\iota(y,l)} \Pp^{\bullet}(l) \notag \\
& e^{\beta^2 t^{\xi_1}} \sum_{\substack{l_-, l_+ \in \N_0, \\ l_- + l_+ \notin J(2t^{\xi_1})}} 
\exp \left(Ct^{\sublin-1} (l_- + l_+) \right) \Pp^{-,+}(l_-, l_+) \sum_{1 \leq r \leq l_- +1} \alpha^r A(t^{\xi_1},l_-,r). \notag
\end{align}
By~(A1) of Lemma~\ref{lm:building_blocks}, 
$$
\limsup_{t \to \infty} e^{-\beta^2 t^{\xi_1}} \sup_{\|y\| \leq t^{\sublin}} \frac{1}{p_t^y} \sum_{l \in K(t-2 t^{\xi_1})} q^y_{\iota(y,l)} \Pp^{\bullet}(l) < \infty. 
$$
If $l_-, l_+ \in \N_0$ such that $l_- + l_+ \notin J(2t^{\xi_1})$, we have $l_- \notin J(t^{\xi_1})$ or $l_+ \notin J(t^{\xi_1})$. Consequently, for any $\theta > 0$, the product of $t^{\theta}$ and the expression in the third line of~\eqref{eq:second_to_last} is less than a constant times  
\begin{align}     
&t^{\theta} e^{\beta^2 t^{\xi_1}} \sum_{l \notin J(t^{\xi_1})} \exp \left(C t^{\sublin -1} l \right) \Pp^-(l) \sum_{1 \leq r \leq l+1} \alpha^r A(t^{\xi_1},l,r)  \notag \\
& \biggl(\sum_{l \in J(t^{\xi_1})} \Pp^-(l) \sum_{1 \leq r \leq l+1} \alpha^r A(t^{\xi_1},l,r)  \notag \\
& + \sum_{l \notin J(t^{\xi_1})} \exp \left(C t^{\sublin -1} l \right) \Pp^-(l) \sum_{1 \leq r \leq l+1} \alpha^r A(t^{\xi_1},l,r) \biggr).  \label{eq:fe_1}
\end{align}
By~(A4) in Lemma~\ref{lm:building_blocks}, as $t \to \infty$, the expression in the first line of~\eqref{eq:fe_1} tends to $0$ for any $\theta > 0$ and the expression in the third line of~\eqref{eq:fe_1} tends to $0$. The second line of~\eqref{eq:fe_1} is bounded on account of~(A7). This completes the proof of~\eqref{eq:fcv_2}. 
\epf 

\subsubsection{{Proof of Claim,~\ref{lm:final_conv_step_1} Part 2}}

Since $q^y_{n_{\bullet}}$, $T^{y,t}_0$ and $T_{0,0}^t$ are independent with respect to $\Pp$, we have 
$$
\E \left[ q^y_{n_{\bullet}} T^{y,t}_0 T_{0,0}^t \right] = p^y_{t-2t^{\xi_1}} \E T^{y,t}_0 \E T_{0,0}^t. 
$$
It is hence enough to show existence of $\theta > 0$ such that 
\begin{align}
\lim_{t \to \infty} t^{\theta} \sup_{\|y\| \leq t^{\sublin}} \frac{\lvert p^y_{t-2t^{\xi_1}} - p^y_t \rvert}{p_t^y} \left \langle  \left \lvert \E T_{0,0}^t \E T^{y,t}_0 \right \rvert \right \rangle =& 0,   \label{eq:end_pf_1}  \\
\lim_{t \to \infty} t^{\theta} \sup_{\|y\| \leq t^{\sublin}} \left \langle \left \lvert \E T_{0,0}^t \E T^{y,t}_0 - Z_{0,0}^{\infty} Z^{y,t}_{-\infty} \right \rvert \right \rangle =& 0. \label{eq:end_pf_2}
\end{align}
Let us point out that there is $\theta > 0$ such that 
$$
\lim_{t \to \infty} t^{\theta} \sup_{\|y\| \leq t^{\sublin}} \frac{\lvert p^y_{t-2t^{\xi_1}} - p^y_t \rvert}{p^y_t} = 0. 
$$
Since $p^y_{t-2t^{\xi_1}} = \E q^y_{n_{\bullet}}$ and $p^y_t = \E q^y_n$, 
this can be shown similarly to~\eqref{eq:sum_w_const_term}. Hence, in order to prove~\eqref{eq:end_pf_1}, it is enough to show that 
$$
\limsup_{t \to \infty} \left \langle \left \lvert \E T_{0,0}^t \E T^{y,t}_0 \right \rvert \right \rangle  < \infty.
$$
By the Cauchy--Schwarz inequality and symmetry,  
$$
\left \langle \left \lvert \E T_{0,0}^t \E T^{y,t}_0 \right \rvert \right \rangle \leq \left \langle \left(\E T_{0,0}^t \right)^2 \right \rangle.
$$

Let us now show that the truncated partition function $\E T_{0,0}^t$ converges to the limiting partition function $Z_{0,0}^\infty$ in the $L^2$ sense and obtain a rate of convergence. In order to do this we will first prove that there is $\theta > 0$ such that 
\begin{equation}    \label{eq:convergence_int_1} 
\lim_{t \to \infty} t^{\theta} \left \langle \left(\E T_{0,0}^t - Z_{0,0}^{t^{\xi_1}} \right)^2 \right \rangle = 0. 
\end{equation} 
We keep the notation introduced at the beginning of Subsection~\ref{ssec:convergence_middle}. Let $t > 0$. For $\Pp_{0,0}$-almost every realization of the continuous-time simple symmetric random walk $\eta$ on $\Z^d$, we have 
$$
e^{-\frac{\beta^2}{2} t^{\xi_1}} e^{\beta \AC_0^{t^{\xi_1}}} - T_{0,0}^t = N_1^{t^{\xi_1}} + N_2^{t^{\xi_1}},
$$
where
\begin{align}     \label{eq:M_T_diff_1} 
N_1^{t^{\xi_1}}
=& \sum_{1 \leq r \leq v(n_-) + 1} \sum_{\substack{0 \leq i_1 < \ldots < i_r \leq n_-, i_r > w(n_-), \\ z_1, \ldots, z_r \in \Z^d}} q_r(\ibf,\zbf) \prod_{j=1}^r h(z_j; s_{i_j}, s_{i_j +1})  \notag \\
N_2^{t^{\xi_1}} 
= & \sum_{v(n_-) + 1 < r \leq n_- +1} \sum_{\substack{0 \leq i_1 < \ldots < i_r \leq n_-, \\ z_1, \ldots, z_r \in \Z^d}} q_r(\ibf,\zbf) \prod_{j=1}^r h(z_j; s_{i_j}, s_{i_j +1}).  \notag 
\end{align}
By Jensen's inequality, 
$$
\left \langle \left(\E T_{0,0}^t - Z_{0,0}^{t^{\xi_1}} \right)^2 \right \rangle \leq \E \left[\left \langle \left( N_1^{t^{\xi_1}} + N_2^{t^{\xi_1}} \right)^2 \right \rangle \right], 
$$
so it is enough to show existence of $\theta > 0$ such that 
$$
\lim_{t \to \infty} t^{\theta} \E \left[ \left \langle \left(N_i^t \right)^2 \right \rangle \right] = 0, \quad i \in \{1,2\}, 
$$
i.e., we need to check convergence of the $D$-sequences 
\begin{align*}
D^{\vartheta}_n(1) =& \sum_{1 \leq r \leq v(n) +1} \vartheta^r \sum_{\substack{0 \leq i_1 < \ldots < i_r \leq n, i_r > w(n), \\ z_1, \ldots, z_r \in \Z^d}} q_r(\ibf,\zbf)^2, \\
D^{\vartheta}_n(2) =& \sum_{v(n)+ 1 < r \leq n+1} \vartheta^r \sum_{\substack{0 \leq i_1 < \ldots i_r \leq n, \\ z_1, \ldots, z_r \in \Z^d}} q_r(\ibf,\zbf)^2. 
\end{align*}
For $\vartheta < \alpha^{-1}$, 
$$
D^{\vartheta}_n(2) \lesssim \sum_{v(n)+1 < r \leq n+1} (\vartheta \alpha)^r \leq \frac{(\vartheta \alpha)^{n/2}}{1-\vartheta \alpha}, 
$$
so 
$$
\lim_{n \to \infty} n^{\theta} D_n^{\vartheta}(2) = 0, \quad \theta > 0, \ \vartheta < \alpha^{-1}. 
$$
Moreover, 
\begin{align*}
D^{\vartheta}_n(1) \lesssim& \sum_{1 \leq r \leq v(n) +1} \vartheta^r \sum_{\substack{j_1, \ldots, j_r \in \N, \\ j_1 + \ldots + j_r > w(n)}} \sum_{c_1, \ldots, c_r \in \Z^d} \left(q_{j_1}^{c_1} \right)^2 \ldots \left(q_{j_r}^{c_r} \right)^2 \\
\leq& \sum_{1 \leq r \leq v(n)+1} \vartheta^r \sum_{l=1}^r \sum_{\substack{j_1, \ldots, j_r \in \N, \\ j_l \geq \frac{1}{2} \frac{w(n)}{v(n)}}} \sum_{k=1}^r \biggl(\sum_{c_k \in \Z^d} \left(q_{j_k}^{c_k} \right)^2 \biggr) \\
\lesssim& \sum_{1 \leq r \leq v(n)+1} r (\vartheta \alpha)^r \sum_{j \geq \frac{1}{2} \frac{w(n)}{v(n)}} \frac{1}{j^{\frac{d}{2}}} 
\lesssim \sum_{r=1}^{\infty} r (\vartheta \alpha)^r \left(\frac{1}{2} \frac{w(n)}{v(n)} - 1 \right)^{1-\frac{d}{2}}, 
\end{align*}
and  
$$
\lim_{n \to \infty} n^{\theta} \left(\frac{1}{2} \frac{w(n)}{v(n)} - 1 \right)^{1-\frac{d}{2}} = 0, \quad \theta \in (0, (\xi_3 - \xi_2) (\tfrac{d}{2}-1)), 
$$ 
from which we deduce~\eqref{eq:convergence_int_1}. If we combine this result with Theorem~\ref{thm:limiting_part_fun}, we obtain in particular that there is $\theta > 0$ such that 
\begin{equation}    \label{eq:truncated_limiting} 
\lim_{t \to \infty} t^{\theta} \left \langle \left(\E T_{0,0}^t - Z_{0,0}^{\infty} \right)^2 \right \rangle = 0.
\end{equation} 
Finally, since $\left \langle \left(Z_{0,0}^{\infty} \right)^2 \right \rangle < \infty$, we obtain~\eqref{eq:end_pf_1}. 

In order to prove~\eqref{eq:end_pf_2}, first notice that for $y \in \Z^d$ such that $\|y\| \leq t^{\sublin}$, we have 
$$
\left \langle \left \lvert \E T_{0,0}^t \E T^{y,t}_0 - Z_{0,0}^{\infty} Z_{-\infty}^{y,t} \right \rvert \right \rangle \leq \left \langle \left \lvert \E T^{y,t}_0 \left(\E T_{0,0}^t - Z_{0,0}^{\infty} \right) \right \rvert \right \rangle + \left \langle \left \lvert Z_{0,0}^{\infty} \left(\E T^{y,t}_0 - Z_{-\infty}^{y,t} \right) \right \rvert \right \rangle. 
$$
Therefore, we obtain~\eqref{eq:end_pf_2} by applying Cauchy--Schwarz to the two summands on the right, and using~\eqref{eq:truncated_limiting} together with  
$$
\lim_{t \to \infty} \left \langle \left(\E T_{0,0}^t \right)^2 \right \rangle = \left \langle \left(Z_{0,0}^{\infty} \right)^2 \right \rangle < \infty.
$$

\section{Proof of Theorem \protect\ref{thm:limiting_part_fun}}
\label{sec:proof_thm_limiting_part_fun}

Before presenting the proof of Theorem \ref{thm:limiting_part_fun}, we introduce some additional notation and establish two lemmas. Let $n \in \N_0$ and $1 \leq r \leq n+1$. For $\ibf = (i_1, \ldots, i_r) \in \N_0^r$ such that $0 \leq i_1 < \ldots < i_r \leq n$, and for $\zbf = (z_1, \ldots, z_r) \in (\Z^d)^r$, define
$q(\ibf,\zbf) \coleq q_{i_1}^{z_1} \ldots q_{i_r - i_{r-1}}^{z_r - z_{r-1}}$. 
For $\tbf = (t_1, t_2, \ldots)$ an infinite sequence of positive numbers, let
	\begin{equation*}
		M_n(\tbf) \coleq \sum_{r=1}^{n+1} 
		\sum_{\substack{0 \leq i_1 < \ldots < i_r \leq n, \\
		 z_1, \ldots, z_r \in \Z^d}}
		q(\ibf,\zbf)^2 \prod_{j=1}^r\left(e^{\beta^2 t_{i_j +1}} -1\right),
	\end{equation*}
which is monotone increasing in $n$.  
Set 
	\begin{equation*}
		M(\tbf) \coleq \lim_{n \to \infty} M_n(\tbf) \in (0,+\infty].
	\end{equation*}
\begin{lemma}     \label{lm:finite_Q}
Let $\beta$ be so small that $\alpha_d \lambda < 1$, and let $\tau = (\tau_k)_{k \in \N}$ be a sequence of independent exponentially distributed random variables with rate $1$. Then
\begin{equation*}
\lim_{n \to \infty} n^{\theta} \ \E \left[M\left(\tau \right) - M_{n-1}(\tau) \right]  = 0, \quad \forall \theta \in \left(0, \min\{\tfrac{d}{2} -1; -\ln(\alpha_d \lambda)\} \right). 
\end{equation*} 
\end{lemma}

\begin{proof}
With $\lambda$ defined by~\eqref{eq:def_lambda}, one has 
\begin{align}      
 \E \left[M \left(\tau \right) - M_{n-1}(\tau) \right]  
\leq&  \sum_{1 \leq r \leq \ln(n)} \lambda^r \sum_{\substack{0 \leq i_1 < \ldots i_r, i_r \geq n, \\ z_1, \ldots, z_r \in \Z^d}} 
q(\ibf,\zbf)^2 \notag \\
&+ \sum_{r > \ln(n)} \lambda^r  \sum_{\substack{0 \leq i_1 < \ldots < i_r, \\ z_1, \ldots, z_r \in \Z^d}} 
		q(\ibf,\zbf)^2. \label{eq:series_remainder}
\end{align}
The expression in the second line of~\eqref{eq:series_remainder} is dominated by 
$$
\sum_{r > \ln(n)} (\alpha_d \lambda)^r \lesssim n^{\ln(\alpha_d \lambda)}, 
$$
and 
$$
\lim_{n \to \infty} n^{\theta} n^{\ln(\alpha_d \lambda)} = 0, \quad \forall \theta \in (0, -\ln(\alpha_d \lambda)). 
$$
As explained in more detail in~\cite{HKNN} (proof of Theorem 2.2), the expression in the first line of~\eqref{eq:series_remainder}, to the right of the inequality sign, is dominated by 
$$
\sum_{r=1}^{\infty} r (\alpha_d \lambda)^r \left(\frac{n}{\ln(n)} \right)^{1 - \frac{d}{2}}, 
$$
and 
$$
\lim_{n \to \infty} n^{\theta} \left(\frac{n}{\ln(n)} \right)^{1 - \frac{d}{2}} = 0, \quad \forall \theta \in (0, \tfrac{d}{2}-1). 
$$
\end{proof} 
\vspace{-3mm}

\begin{lemma}     \label{lm:Z_increment}
Let $\beta$ be so small that Key Lemma~\ref{lm:standard_arg} holds. Let $T > t > 0$. For $\Pp_{0,0}$-almost every realization of $\eta$ and for every sequence $\tbf = (t_1, t_2, \ldots)$ of positive numbers such that $t_i = s_i - s_{i-1}$, $1 \leq i \leq n_{T}$ and $t_{n_T + 1} \geq T - s_{n_T}$, one has 
$$   
\biggl \langle \biggl(e^{-\frac{\beta^2}{2} T} e^{\beta \AC_0^T} - e^{-\frac{\beta^2}{2} t} e^{\beta \AC_0^t} \biggr)^2 \biggr \rangle \\
\leq 3 \left(M \left(\tbf \right) - M_{n_{t}-1} \left(\tbf \right) \right) + 3 \left \langle \left(R^t \right)^2 \right \rangle,
$$
where $R^t$ is a random variable depending on $\eta$ and $(W^x)_{x \in \Z^d}$ that satisfies
\begin{equation}     \label{eq:saw_2}
\lim_{t \to \infty} t^{\theta} \ \E_{0,0} \left \langle \left(R^t \right)^2 \right \rangle = 0, \quad \forall \theta \in \left(0, \tfrac{d}{2}\right). 
\end{equation}
\end{lemma} 
%
%

\begin{proof}[Proof of Lemma~\ref{lm:Z_increment}]
Since $t$ is $\Pp_{0,0}$-almost surely not a jump time for $\eta$, we have 
\begin{align*}       
		&e^{-\frac{\beta^2}{2} T} e^{\beta \AC_0^T} - e^{-\frac{\beta^2}{2} t} e^{\beta \AC_0^t}   \\
		=& \sum_{r=1}^{n_{t}} \sum_{\substack{0 \leq i_1 < \ldots < i_r \leq n_T,\ i_r \geq n_t, \\ z_1, \ldots, z_r \in \Z^d}} 
		q(\ibf,\zbf)
		\prod_{j=1}^r h(z_j; s_{i_j}, s_{i_j +1}) \\
		& + \sum_{r=n_t+1}^{n_T+1} \sum_{\substack{0 \leq i_1 < \ldots < i_r \leq n_T, \\ z_1, \ldots, z_r \in \Z^d}} q(\ibf,\zbf)
		 \prod_{j=1}^r h(z_j; s_{i_j}, s_{i_j + 1})   - R^t,  
	\end{align*}
where $s_{n_T +1} = T$ and 
$$
R^t \coleq \sum_{r=1}^{n_t +1} \sum_{\substack{0 \leq i_1 < \ldots < i_r = n_t, \\ z_1, \ldots, z_r \in \Z^d}} q(\ibf,\zbf)
 \prod_{j=1}^{r-1} h(z_j; s_{i_j}, s_{i_j +1}) h(z_r; s_{n_t}, t). 
$$
If we set $t_i \coleq s_i - s_{i-1}$ for $1 \leq i \leq n_T$ and fix $t_{n_T + 1} \geq T - s_{n_T}$, we obtain for any positive $(t_i)_{i > n_T+1}$ the estimate 
	\begin{align*}
		& \biggl \langle \biggl( e^{-\frac{\beta^2}{2} T} e^{\beta \AC_0^T} - e^{-\frac{\beta^2}{2} t} e^{\beta \AC_0^t} \biggr)^2 \biggr \rangle  \\
		\leq& 3 \sum_{r=1}^{n_t} \sum_{\substack{0 \leq i_1 < \ldots < i_r \leq n_T, i_r \geq n_t, \\ z_1, \ldots, z_r \in \Z^d}} 
		q(\ibf,\zbf)^2  
		\prod_{j=1}^r \left \langle  h(z_j; s_{i_j}, s_{i_j + 1})^2 \right\rangle  \\
		& + 3 \sum_{r=n_t+1}^{n_T+1} \sum_{\substack{0 \leq i_1 < \ldots < i_r \leq n_T, \\ z_1, \ldots, z_r \in \Z^d}} 
		q(\ibf,\zbf)^2  \prod_{j=1}^r \left \langle h(z_j; s_{i_j}, s_{i_j + 1})^2 \right \rangle + 3 \left \langle \left(R^t \right)^2 \right \rangle \\
		\leq& 3 \left(M\left(\tbf \right) - M_{n_t-1}\left(\tbf \right) \right) + 3 \left \langle \left(R^t \right)^2 \right \rangle.  
	\end{align*}
The only point left to show is~\eqref{eq:saw_2}.  The $D$-sequence associated with $R^t$ is 
$$ 
D_n^{\vartheta} = \sum_{r=1}^{n+1} \vartheta^r \sum_{\substack{0 \leq i_1 < \ldots < i_r = n, \\ z_1, \ldots, z_r \in \Z^d}} 
q(\ibf,\zbf)^2, \quad \vartheta > 0.
$$ 
\noindent For $n \geq 2$ and $\vartheta \leq 1$,  
	\begin{equation*}
		D_n^{\vartheta} = \vartheta U_{n,1} + \sum_{r=2}^n \vartheta^r (U_{n,r} + U_{n,r-1}) + \vartheta^{n+1} U_{n,n} \leq 2 \sum_{r=1}^n \vartheta^r U_{n,r},   
	\end{equation*} 
\noindent where  
	\begin{equation*}
		U_{n,r} \coleq \begin{cases}
		\displaystyle
		 \sum_{\substack{0 < i_1 < \ldots < i_{r-1} < n, \\ z_1, \ldots, z_r \in \Z^d}} \left(q_{i_1}^{z_1} \right)^2 \ldots \left(q_{n - i_{r-1}}^{z_r - z_{r-1}}\right)^2, & \quad 2 \leq r \leq n, \\
		 \displaystyle
		\sum_{z \in \Z^d} \left(q_n^z \right)^2, & \quad r = 1. 
			\end{cases}
	\end{equation*}
Let $C > 0$ such that $q_n^y \leq C n^{-\frac{d}{2}}$ for all $n \in \N$ and $y \in \Z^d$. Then,  
	\begin{equation*}
		U_{n,r} \leq \begin{cases}
                                      \displaystyle  C^r  \sum_{0 < i_1 < \ldots < i_{r-1} < n} i_1^{-\frac{d}{2}} \ldots (n - i_{r-1})^{-\frac{d}{2}}, & \quad 2 \leq r \leq n, \\ 
C n^{-\frac{d}{2}}, & \quad r = 1.
\end{cases} 
	\end{equation*}
By Lemma~\ref{lm:lclt_lemma} in the appendix, there is a constant $c > 1$, depending only on the dimension $d$, such that for any $r \in \N$ and $n \geq r+1$,     
	\begin{equation}    \label{eq:lclt_lemma}
		\sum_{0 < i_1 < \ldots < i_r < n} i_1^{-\frac{d}{2}} (i_2 - i_1)^{-\frac{d}{2}} \ldots (i_r - i_{r-1})^{-\frac{d}{2}} (n - i_r)^{-\frac{d}{2}} \leq c^r  n^{-\frac{d}{2}}.   
	\end{equation}  
Hence,	
	\begin{equation*}
		U_{n,r} \leq (Cc)^r n^{-\frac{d}{2}}, \quad 1 \leq r \leq n.
	\end{equation*}
As a result, 
	\begin{equation}       \label{eq:D_seq_R_est}
		D_n^{\vartheta} \leq \frac{2}{n^{\frac{d}{2}}} \sum_{r=1}^{\infty} (\vartheta Cc)^r. 
	\end{equation}
For $\vartheta < (Cc)^{-1}$ and $\theta < \tfrac{d}{2}$, we obtain $\lim_{n \to \infty} n^{\theta} D_n^{\vartheta} = 0$. Remark~\ref{rm:standard_arg} yields~\eqref{eq:saw_2} for $\beta$ sufficiently small. 
\end{proof}

\bigskip


\begin{proof}[Proof of Theorem~\ref{thm:limiting_part_fun}] 
Let $\beta$ be so small that the conclusion of Key Lemma~\ref{lm:standard_arg} holds, that $\alpha_d \lambda < 1$ and that 
$$
\E \left[ \left(e^{\beta^2 \tau} - 1 \right)^2 \right]^{\frac{1}{2}} < \alpha_d^{-1} 
$$
for $\tau$ an exponential random variable of rate $1$. Assume without loss of generality that $s = 0$ and that $x$ is the zero vector in $\R^d$, and set $\E \coleq \E_{0,0}$. Let $T > t > 0$ and let $b \in (0,1)$.  
Since \\ $Z_{0,0}^t = e^{-\frac{\beta^2}{2} t} \E e^{\beta \AC_0^t}$, we obtain with Jensen's inequality and Lemma~\ref{lm:Z_increment}
	\begin{align*}     
		& \left\langle \left(Z^T_{0,0} - Z^t_{0,0}\right)^2 \right\rangle \\
\leq&  \E \biggl[ \biggl \langle \biggl(e^{-\frac{\beta^2}{2} T} e^{\beta \AC_0^T} - e^{-\frac{\beta^2}{2} t} e^{\beta \AC_0^t} \biggr)^2 \biggr \rangle \left(\id_{s_{\lfloor t^b \rfloor} < t} + \id_{s_{\lfloor t^b \rfloor} \geq t} \right) \biggr] \\
		\leq& \E \left[3 \left(M \left(\tbf \right) - M_{\lfloor t^b \rfloor -1}\left(\tbf \right)\right) \right] + \E\left[3 M\left(\tbf \right) \ \id_{s_{\lfloor t^b \rfloor} \geq t}\right] + 3 \E \left \langle \left(R^t \right)^2 \right \rangle,     
	\end{align*} 
where $\tbf = (t_1, t_2, \ldots)$ is defined as $t_i \coleq s_i - s_{i-1}$ for $i \geq 1$, and is thus an i.i.d. sequence of exponentially distributed random variables with rate $1$. As a consequence of Lemma~\ref{lm:finite_Q}, 
\begin{equation*}
\lim_{t \to \infty} t^{\theta} \ \E \left[ 3 \left(M(\tbf) - M_{\lfloor t^b \rfloor - 1}(\tbf) \right) \right] = 0, \quad \forall \theta \in \left(0, b \min \left\{\tfrac{d}{2}-1; -\ln(\alpha_d \lambda) \right\} \right). 
\end{equation*}
Furthermore,~\eqref{eq:saw_2} gives 
\begin{equation*}
\lim_{t \to \infty} t^{\theta} \ \E \left \langle \left(R^t \right)^2 \right \rangle = 0, \quad \forall \theta \in \left(0, \tfrac{d}{2} \right).
\end{equation*} 
Finally, the Cauchy--Schwarz inequality implies  
\begin{align*}
\E \left[M(\tbf) \mathbbm{1}_{s_{\lfloor t^b \rfloor} \geq t} \right] 
\leq& \sum_{r=1}^{\infty} \sum_{\substack{0 \leq i_1 < \ldots < i_r, \\ z_1, \ldots, z_r \in \Z^d}} 
q(\ibf,\zbf)^2 \E \biggl[ \prod_{j=1}^r \left( e^{\beta^2 t_{i_j +1}} - 1 \right) \mathbbm{1}_{s_{\lfloor t^b \rfloor} \geq t} \biggr] \\
\leq& \ \Pp \left( s_{\lfloor t^b \rfloor} \geq t \right)^{\frac{1}{2}} \sum_{r=1}^{\infty} \rho^r \sum_{0 \leq i_1 < \ldots < i_r} \sum_{z_1, \ldots, z_r \in \Z^d} 
q(\ibf,\zbf)^2, 
\end{align*} 
where 
$$ 
\rho \coleq \E\left[ \left(e^{\beta^2 t_1} - 1 \right)^2 \right]^{\frac{1}{2}} < \alpha_d^{-1}, 
$$ 
so 
$$
\sum_{r=1}^{\infty} \rho^r \sum_{0 \leq i_1 < \ldots < i_r} \sum_{z_1, \ldots, z_r \in \Z^d} q(\ibf,\zbf)^2 \lesssim \sum_{r=1}^{\infty} (\rho \alpha_d)^r < \infty. 
$$
We have 
\begin{equation*}
\Pp \left( s_{\lfloor t^b \rfloor} \geq t \right) = e^{-t} \sum_{k=0}^{\lfloor t^b \rfloor - 1} \frac{t^k}{k!}. 
\end{equation*}
By Stirling's formula, 
$$
t^{\theta} e^{-t} \sum_{k=0}^{\lfloor t^b \rfloor -1} \frac{t^k}{k!} \leq \frac{t^{\theta+1} e^{-t} e^{\lfloor t^b \rfloor -1} t^{\lfloor t^b \rfloor -1}}{\sqrt{2 \pi (\lfloor t^b \rfloor -1)} (\lfloor t^b \rfloor -1)^{\lfloor t^b \rfloor -1}} \lesssim t^{\theta+1-\frac{b}{2}} e^{-t} \left(\frac{et}{\lfloor t^b \rfloor -1} \right)^{\lfloor t^b \rfloor -1},    
$$
and the right-hand side tends to $0$ as $t \to \infty$ for every $\theta > 0$.
Since $b$ was arbitrarily chosen from $(0,1)$, we obtain 
$$
\lim_{t \to \infty} t^{\theta} \left \langle \left(Z^t_{0,0} - Z_{0,0}^{\infty} \right)^2 \right \rangle = 0, \qquad \forall \theta \in (0, \min\{\tfrac{d}{2} -1; -\ln(\alpha_d \lambda)\}). 
$$
\end{proof}


\section{Proof of Theorem~\ref{thm:positivity}}     \label{sec:proof_uniformity} 

Theorem~\ref{thm:positivity} follows directly from Proposition~\ref{prop:uniform_as} (Section~\ref{ssec:uniform_as}) and Proposition~\ref{prop:uniform_positivity} (Section~\ref{ssec:positivity}):  

Let $\beta$ be so small that $\alpha_d \lambda < 1$. By Proposition~\ref{prop:uniform_as}, there exists $\Omega^{\lim} \in \F$ with $Q(\Omega^{\lim}) = 1$ such that for all $(x,s) \in \Z^d \times \R$ and all $\omega \in \Omega^{\lim}$, $Z_{x,s}^t(\omega)$ converges to a limit $Z_{x,s}^{\infty}(\omega)$ as $t \to \infty$. By Proposition~\ref{prop:uniform_positivity}, there exists  $\Omega^+ \in \F$ with $Q(\Omega^+) = 1$ such that for all $(x,s) \in \Z^d \times \R$ and all $\omega \in \Omega^+$, one has $Z_{x,s}^{\infty}(\omega) > 0$.

\subsection{Uniform almost sure convergence of the partition functions} \label{ssec:uniform_as} 

\begin{proposition}      \label{prop:uniform_as} 
Let $\beta$ be so small that $\alpha_d \lambda < 1$. Then, there exists $\Omega^{\lim}  \in \F$ with $Q(\Omega^{\lim}) = 1$ such that for all $(x,s) \in \Z^d \times \R$ and all $\omega \in \Omega^{\lim}$, $Z_{x,s}^t(\omega)$ converges to a limit $Z_{x,s}^{\infty}(\omega)$ in $[0, \infty)$ as $t \to \infty$.
\end{proposition} 

In order to prove Proposition~\ref{prop:uniform_as}, we first establish two lemmas. 

\begin{lemma}     \label{lm:Z_Feynman_Kac} 
Let $x, y \in \Z^d$, $\omega \in \Omega$, and $s < u < t$. Then 
$$
Z_{x,s}^{y,t}(\omega) = \sum_{z \in \Z^d} Z_{x,s}^{z,u}(\omega) Z_{z,u}^{y,t}(\omega) \quad \text{and} \quad Z_{x,s}^t(\omega) = \sum_{z \in \Z^d} Z_{x,s}^{z,u}(\omega) Z_{z,u}^t(\omega). 
$$
\end{lemma} 

\begin{proof}
For any $z, z' \in \Z^d$ and any $n \in \N_0$, let $\Gamma_z^{z', n}$ be the collection of all discrete-time random walk paths $\gamma = (\gamma_0, \gamma_1, \ldots, \gamma_n)$ with starting point $\gamma_0 = z$ and endpoint $\gamma_n = z'$. Let $\sbf = (s_n)_{n \in \N_0}$ be a sequence of jump times with $s = s_0 < s_1 < s_2 < \ldots$ and $u$ not a term in $\sbf$. Let $n_{s,u}$, $n_{u,t}$, and $n_{s,t}$ denote the number of jumps within the time intervals $(s,u)$, $(u,t)$, and $(s,t)$, respectively. For $z \in \Z^d$ and $\gamma = (\gamma_0, \ldots, \gamma_{n_{s,t}}) \in \Gamma_x^{y, n_{s,t}}$ with $\gamma_{n_{s,u}} = z$, let $\gamma' \in \Gamma_x^{z, n_{s,u}}$ and $\gamma'' \in \Gamma_z^{y,n_{u,t}}$ be the unique paths which concatenate to $\gamma$, with the endpoint of $\gamma'$ and the starting point of $\gamma''$ overlapping. By a slight abuse of notation, let the sample path for the continuous-time random walk described by $\gamma$ (or $\gamma'$, $\gamma''$) and $\sbf$ be denoted by $\gamma$ (or $\gamma'$, $\gamma''$) as well. Then  
$$
\AC_s^t(\gamma, \omega) = \AC_s^u(\gamma', \omega) + \AC_u^t(\gamma'', \omega). 
$$
Using the fact that it is enough to average over the sequences of jump times $\sbf$ not containing $u$ (because the complement has measure zero), one has 
\begin{align*}
Z_{x,s}^{y,t}(\omega) =& e^{-\frac{\beta^2}{2} (t-s)} \E_{\sbf} \biggl[ \sum_{\gamma \in \Gamma_x^{y,n_{s,t}}} \frac{1}{(2d)^{n_{s,t}}} e^{\beta \AC_s^t(\gamma, \omega)} \biggr] \\
=& e^{-\frac{\beta^2}{2} (u-s)} e^{-\frac{\beta^2}{2} (t-u)} \E_{\sbf} \biggl[ \sum_{z \in \Z^d} \sum_{\gamma' \in \Gamma_x^{z, n_{s,u}}} \sum_{\gamma'' \in \Gamma_z^{y,n_{u,t}}} \frac{e^{\beta \AC_s^u(\gamma', \omega)}}{(2d)^{n_{s,u}}}  \frac{e^{\beta \AC_u^t(\gamma'', \omega)}}{(2d)^{n_{u,t}}} \biggr] \\
=& \sum_{z \in \Z^d} Z_{x,s}^{z,u}(\omega) Z_{z,u}^{y,t}(\omega),    
\end{align*}
where $\E_{\sbf}$ denotes averaging with respect to $\sbf$. The identity 
$$
Z_{x,s}^t(\omega) = \sum_{z \in \Z^d} Z_{x,s}^{z,u}(\omega) Z_{z,u}^t(\omega)
$$
is obtained by summing on both sides over all $y \in \Z^d$. 
\end{proof}

Let $\beta$ be so small that $\alpha_d \lambda < 1$ and define 
$$
\Omega^{\lim} := \bigcap_{x \in \Z^d} \bigcap_{n \in \Z} \{\omega \in \Omega: \ \lim_{t \to \infty} Z_{x,n}^t(\omega) \ \text{exists in} \ [0, \infty)\}. 
$$
Since for any $(x,s) \in \Z^d \times \R$, the limit $\lim_{t \to \infty} Z_{x,s}^t(\omega)$ exists in $[0, \infty)$ $Q$-almost surely, one has $Q(\Omega^{\lim}) = 1$. 

\begin{lemma}     \label{lm:Z_Feynman_Kac_infinite}
Assume that $\beta$ is so small that $\alpha_d \lambda < 1$. Let $x \in \Z^d$ and let $s < u$. Then, $Q$-almost surely, 
$$
Z_{x,s}^{\infty} = \sum_{z \in \Z^d} Z_{x,s}^{z,u} Z_{z,u}^{\infty}. 
$$
\end{lemma} 

\begin{proof}
By Lemma~\ref{lm:Z_Feynman_Kac}, for any $t > u$, 
	\begin{equation}    \label{eq:Z_lim_mid_step}
		Z_{x,s}^t= 
		\sum_{z\in \Z^d} Z_{x,s}^{z,u} Z_{z,u}^t, 
	\end{equation}	
so the expression on the right converges $Q$-almost surely to $Z_{x,s}^{\infty}$ as $t \to \infty$. To complete the proof of the lemma, it is then enough to show that the right-hand side of~\eqref{eq:Z_lim_mid_step} converges to $\sum_{z \in \Z^d} Z_{x,s}^{z,u} Z_{z,u}^{\infty}$ in $L^1$ as $t \to \infty$. 
One has 
\begin{align*}
	\biggl\langle \biggl \lvert  \sum_{z\in \Z^d} Z_{x,s}^{z,u} Z_{z,u}^t  - \sum_{z\in \Z^d} Z_{x,s}^{z,u} Z_{z,u}^\infty  \biggr \rvert  \biggr\rangle
	& \leq \biggl\langle   \sum_{z\in \Z^d} Z_{x,s}^{z,u} \left | Z_{z,u}^t - Z_{z,u}^\infty  \right|  \biggr\rangle \\
	& = \sum_{z\in \Z^d} \left\langle    Z_{x,s}^{z,u} \left | Z_{z,u}^t - Z_{z,u}^\infty  \right|  \right\rangle \\
	& = \sum_{z\in \Z^d} \left\langle    Z_{x,s}^{z,u} \right\rangle \left\langle \left | Z_{z,u}^t  - Z_{z,u}^\infty  \right|  \right\rangle\\
	& = \left\langle \left | Z_{0,u}^t - Z_{0,u}^\infty  \right|  \right\rangle,
\end{align*}
which converges to 0 as $t\to\infty$ by virtue of Theorem~\ref{thm:limiting_part_fun_exists}.
\end{proof}

\begin{proof}[Proof of Proposition~\ref{prop:uniform_as}] Let $\Omega^{\lim}$ be the set of $\omega \in \Omega$ such that 
$$
Z_{x,n}^{\infty}(\omega) := \lim_{t \to \infty} Z_{x,n}^t(\omega) \ \text{exists in} \ [0, \infty) \ \forall (x,n) \in \Z^d \times \Z 
$$
and 
$$
Z_{x,n}^{\infty}(\omega) = \sum_{z \in \Z^d} Z_{x,n}^{z,m}(\omega) Z_{z,m}^{\infty}(\omega) \ \forall x \in \Z^d \ \text{and} \ m, n \in \Z \ \text{with} \ n < m. 
$$
Then $\Omega^{\lim} \in \F$ and, by Lemma~\ref{lm:Z_Feynman_Kac_infinite}, $Q(\Omega^{\lim}) = 1$. 
Fix $\omega \in \Omega^{\lim}$, $x \in \Z^d$, and $s \in \R \setminus \Z$. Let $n \in \Z$ such that $n-1 < s < n$. 
For $y \in \Z^d$ and $t > n$, set 
$$
a_y := Z_{x,n-1}^{y,n}(\omega), \ a_y(s) := Z_{x,s}^{y,n}(\omega), \ b_y(t) := Z_{y,n}^t(\omega), \ b_y := Z_{y,n}^{\infty}(\omega). 
$$
Since $\lim_{t \to \infty} Z_{x,n-1}^t(\omega)$ exists in $[0, \infty)$, there is $T > n$ such that $Z_{x,n-1}^t(\omega) < \infty$ for every $t \geq T$. By Lemma~\ref{lm:Z_Feynman_Kac}, one has for every $t \geq T$ 
$$
c(t) := \sum_{z \in \Z^d} a_z b_z(t) = Z_{x,n-1}^t(\omega) < \infty. 
$$
Furthermore, 
$$
c := \sum_{z \in \Z^d} a_z b_z = Z_{x,n-1}^{\infty}(\omega) < \infty
$$
and $\lim_{t \to \infty} c(t) = c$. 
For any $y \in \Z^d$, 
\begin{equation} \label{eq:a_alpha_estimate} 
a_y = \sum_{z \in \Z^d} Z_{x,n-1}^{z,s}(\omega) Z_{z,s}^{y,n}(\omega) \geq Z_{x,n-1}^{x,s}(\omega) a_y(s) \geq \alpha(s) a_y(s),   
\end{equation} 
where 
$$
\alpha(s) := e^{-\frac{\beta^2}{2} (s-n+1)} e^{-(s-n+1)} e^{\beta (\omega(x,s) - \omega(x,n-1))} > 0. 
$$
Thus, since $a_y(s) \leq a_y/\alpha(s)$, the convergence of the series $c(t)$ ($t \geq T$) and $c$ yields the convergence of 
$$
Z_{x,s}^t(\omega) = \sum_{y \in \Z^d} a_y(s) b_y(t), \ t \geq T, \quad \text{and} \quad S:= \sum_{y \in \Z^d} a_y(s) b_y. 
$$
We claim that 
$$
\lim_{t \to \infty} Z_{x,s}^t(\omega) = S,  
$$
which implies the desired convergence of $Z_{x,s}^t(\omega)$ to a finite limit. 

Let $\eps > 0$. As the series $c$ converges, there exists $R > 0$ such that 
\begin{equation}    \label{eq:convergent_c} 
\sum_{y \in \Z^d: \|y\| > R} a_y b_y < \frac{\eps \alpha(s)}{9}. 
\end{equation} 
In addition, since $\lim_{t \to \infty} b_y(t) = b_y$ for every $y \in \Z^d$, there exists $T_1 > T$ such that 
$$
\biggl \lvert \sum_{y \in \Z^d: \|y\| \leq R} a_y b_y(t) - \sum_{y \in \Z^d: \|y\| \leq R} a_y b_y \biggr \rvert < \frac{\eps \alpha(s)}{9}, \quad \forall t \geq T_1. 
$$
Finally, since $\lim_{t \to \infty} c(t) = c$, there exists $T_2 > T$ such that 
$$
\lvert c(t) - c \rvert < \frac{\eps \alpha(s)}{9}, \quad \forall t \geq T_2. 
$$
Let us show that for every $t \geq \max\{T_1, T_2\}$, one has 
\begin{equation}    \label{eq:a_b_t_estimate}
\sum_{y \in \Z^d: \|y\| > R} a_y b_y(t) < \frac{\eps \alpha(s)}{3}. 
\end{equation} 
To obtain a contradiction, assume that there is $t \geq \max\{T_1, T_2\}$ such that 
$$
\sum_{y \in \Z^d: \|y\| > R} a_y b_y(t) \geq \frac{\eps \alpha(s)}{3}. 
$$
Then 
\begin{align*}
 \frac{\eps \alpha(s)}{9} >& \lvert c(t) - c \rvert \\
 \geq& \sum_{\substack{y \in \Z^d: \\ \|y\| > R}} a_y b_y(t) - \biggl \lvert \sum_{\substack{y \in \Z^d: \|y\| \leq R}} a_y b_y(t) - \sum_{\substack{y \in \Z^d: \\ \|y\| \leq R}} a_y b_y \biggr \rvert - \sum_{\substack{y \in \Z^d: \\ \|y\| > R}} a_y b_y \\
 >& \frac{\eps \alpha(s)}{3} - \frac{\eps \alpha(s)}{9} - \frac{\eps \alpha(s)}{9} = \frac{\eps \alpha(s)}{9}, 
\end{align*}
which is absurd. From~\eqref{eq:a_b_t_estimate} and~\eqref{eq:a_alpha_estimate}, we infer that 
$$
\sum_{y \in \Z^d: \|y\| > R} a_y(s) b_y(t) < \frac{\eps}{3}, \quad \forall t \geq \max\{T_1, T_2\}. 
$$
And~\eqref{eq:convergent_c} together with~\eqref{eq:a_alpha_estimate} yields 
$$
\sum_{y \in \Z^d: \|y\| > R} a_y(s) b_y < \frac{\eps}{9}. 
$$
Let $T_3 > T$ be so large that 
$$
\biggl \lvert \sum_{y \in \Z^d: \|y\| \leq R} a_y(s) b_y(t) - \sum_{y \in \Z^d: \|y\| \leq R} a_y(s) b_y \biggr \rvert < \frac{\eps}{3}, \quad \forall t \geq T_3. 
$$
Then, for $t \geq \max\{T_1, T_2, T_3\}$, 
\begin{align*}
\lvert Z_{x,s}^t(\omega) - S \rvert \leq& \sum_{\substack{y \in \Z^d: \\ \|y\| > R}} a_y(s) b_y(t) + \sum_{\substack{y \in \Z^d: \\ \|y\| > R}} a_y(s) b_y \\
&+ \biggl \lvert \sum_{\substack{y \in \Z^d: \\ \|y\| \leq R}} a_y(s) b_y(t) - \sum_{\substack{y \in \Z^d: \\ \|y\| \leq R}} a_y(s) b_y \biggr \rvert \\
<& \frac{\eps}{3} + \frac{\eps}{9} + \frac{\eps}{3} < \eps. 
\end{align*}
\end{proof}

\subsection{Positivity of the Limiting Partition Functions}
\label{ssec:positivity}

From the definition of the partition functions $Z_{x,s}^{y,t}(\omega)$, it is clear that the limiting partition functions $Z_{x,s}^{\infty}(\omega)$ are nonnegative $Q$-almost surely. In this section, we show that they are in fact positive $Q$-almost surely. 

\begin{proposition}     \label{prop:uniform_positivity}
Let $\beta$ be so small that $\alpha_d \lambda < 1$. Then there is an $\F$-measurable set $\Omega^+ \subset \Omega$ with $Q(\Omega^+) = 1$ such that for all $(x,s) \in \Z^d \times \R$ and all $\omega \in \Omega^+$, one has $Z_{x,s}^{\infty}(\omega) > 0$. 
\end{proposition}

%
%
\begin{proof}
We follow to a large extent the proof strategy from~\cite{Sinai_95}.   Recall the set $\Omega^{\lim}$ from Proposition~\ref{prop:uniform_as}, whose precise definition is given at the beginning of the proof of Proposition~\ref{prop:uniform_as}. By definition of $\Omega^{\lim}$, one has for every $\omega \in \Omega^{\lim}$, $x \in \Z^d$, and $m, n \in \Z$ with $n < m$ that 
$$
Z_{x,n}^{\infty}(\omega) = \sum_{z \in \Z^d} Z_{x,n}^{z,m}(\omega) Z_{z,m}^{\infty}(\omega). 
$$
Thus, for $\omega \in \Omega^{\lim}$, $x \in \Z^d$, and $n \in \Z$,  
\begin{equation}    \label{eq:Z_positivity_implication} 
Z_{x,n}^{\infty}(\omega) = 0 \quad \Longrightarrow \quad Z_{z,m}^{\infty}(\omega) = 0 \ \forall z \in \Z^d, m \in \{n+1, n+2, \ldots\}. 
\end{equation} 
Let us show the following claim: 

\medskip 

For every $x \in \Z^d$ and $n \in \Z$, there exists $\Omega^+_{x,n} \in \F$ with $Q(\Omega^+_{x,n}) = 1$ such that $\Omega^+_{x,n} \subset \Omega^{\lim}$ and $Z_{x,n}^{\infty}(\omega) > 0$ for every $\omega \in \Omega^+_{x,n}$. 

\medskip 

Let $e_1 := (1,0, \ldots, 0) \in \Z^d$. For $\omega \in \Omega$, we define $\theta^{e_1}\omega$ as the element of $\Omega$ for which 
$$
\theta^{e_1} \omega(x,t) = \omega(x-e_1, t), \quad \forall x \in \Z^d, \ t \in \R. 
$$
In words, $\theta^{e_1}$ is the spatial shift in the direction of $e_1$. Notice that $\theta^{e_1}$ preserves the measure $Q$ on $\Omega$ and is mixing. In addition, $\Omega^{\lim}$ is invariant under $\theta^{e_1}$. Fix $n \in \Z$. For $x \in \Z^d$, define the random variable 
$$
I_x(\omega) := \inf\{k \in \Z: \ Z_{x,n+k}^{\infty}(\omega) = 0\}, \quad \omega \in \Omega^{\lim}. 
$$
For $x \in \Z^d$, let us prove that $I_x$ is almost surely infinite. To obtain a contradiction, assume that $Q(I_x \in \Z) > 0$. Then there is $k \in \Z$ such that $Q(I_x = k) > 0$. As, for any $j \in \Z$, the discrete-time stochastic processes $(Z_{x,n+i}^{\infty})_{i \in \Z}$ and $(Z_{x,n+j+i}^{\infty})_{i \in \Z}$ have the same law, one has $Q(I_x = j) = Q(I_x = k) > 0$ for all $j \in \Z$. This leads to a contradiction because the events $\{I_x = j\}$, $j \in \Z$, are disjoint. Next, we show that $Q(I_x = -\infty) = 0$. One has 
\begin{align*}
    Q(I_x = -\infty) =& Q \biggl(\bigcap_{k \in \Z} \bigcup_{j < k} \{Z_{x,n+j}^{\infty} = 0\} \biggr) \\
    =& Q \biggl(\bigcap_{k \in \Z} \bigcap_{y \in \Z^d} \{Z_{y,n+k}^{\infty} = 0\} \biggr) = Q \biggl(\bigcap_{y \in \Z^d} \{I_y = -\infty\} \biggr),  
\end{align*}
where we used~\eqref{eq:Z_positivity_implication} to obtain the second identity. Since $\bigcap_{y \in \Z^d} \{I_y = -\infty\}$ is invariant under the shift $\theta^{e_1}$ introduced earlier and since $\theta^{e_1}$ is mixing, in particular ergodic, we have 
$$
Q \biggl(\bigcap_{y \in \Z^d} \{I_y = -\infty\} \biggr) \in \{0,1\}. 
$$
If the set $\{I_x = -\infty\}$ had measure $1$, we would have $Z_{x,n}^{\infty} = 0$, $Q$-almost surely. But this is not the case because $$
\langle Z_{x,n}^{\infty} \rangle = \lim_{t \to \infty} \langle Z_{x,n}^t \rangle = 1. 
$$
Hence, $Q(I_x = -\infty) = 0$. Together with $Q(I_x \in \Z) = 0$, this yields $Q(I_x = \infty) = 1$ and in particular $Q(Z_{x,n}^{\infty} > 0) = 1$. This completes the proof of the claim. 

\medskip 

Let $\Omega^+$ be the set of $\omega \in \Omega^{\lim}$ such that 
$$
Z_{x,n}^{\infty}(\omega) > 0 \quad \forall (x,n) \in \Z^d \times \Z.  
$$
Our claim implies that $Q(\Omega^+) = 1$. 
Fix $\omega \in \Omega^+$, $x \in \Z^d$, and $s \in \R \setminus \Z$. Let $n \in \Z$ such that $n-1 < s < n$. In the proof of Proposition~\ref{prop:uniform_as}, we saw that 
$$
Z_{x,s}^{\infty}(\omega) = \sum_{y \in \Z^d} Z_{x,s}^{y,n}(\omega) Z_{y,n}^{\infty}(\omega) 
$$
and the expression on the right-hand side is positive. 
\end{proof}


\section{A Lower Tail of the Partition Function: Talagrand's Method}
\label{sec:proof_continuous_Talagrand}

This section is devoted to the proof of Theorem~\ref{thm:continuous_Talagrand}. The first step is to prove a discrete-time version of Theorem~\ref{thm:continuous_Talagrand}, which is carried out in Subsection~\ref{ssec:Talagrand}. Then, in Subsection~\ref{continuous_Talagrand}, we prove Theorem~\ref{thm:continuous_Talagrand} by a suitable limiting procedure to go from discrete time to continuous time. To prove the discrete-time version, we follow closely the strategy laid out in \cite[Section 4]{CarmonaHu}, which goes back to Talagrand~\cite{Talagrand03}. 

\subsection{The discrete-time case}    \label{ssec:Talagrand}

Let $N$ be a positive integer and let $\Sc = (S_n)_{n \geq 0}$ be a random walk on $\Z^d$, $d \geq 3$, that starts at the origin and has transition probabilities 
\begin{equation}\label{eq:trans_prob_dis}
\Psf(S_{n+1} = y \; | S_n = x) := 
	\begin{cases}
	\frac{N}{N+1} &\text{ if \; $y = x$},  \\
	\frac{1}{2d (N+1)} &\text{ if \; $\| y - x \| = 1$}.
	\end{cases}
\end{equation}
In addition, let $(\omega(z,k))_{z \in \Z^d, k \in \N_0}$ be an i.i.d. collection of Gaussian random variables with mean $0$ and variance $\tfrac{1}{N}$ that is independent of $\Sc$.  We denote the probability measure corresponding to $(\omega(z,k))_{z \in \Z^d, k \in \N_0}$ by $Q^N$ and the expected value by $\langle \cdot \rangle_{\hspace{-0.5mm}N}$.  
For any $t > 0$ such that $tN \in \N$, we let $\Gamma_{tN}$ denote the set of possible realizations of $(S_0, \ldots, S_{tN-1})$. Given a path $\gamma = (\gamma_0, \ldots, \gamma_{tN-1}) \in \Gamma_{tN}$, we let $p(\gamma)$ denote the probability that $(S_0, \ldots, S_{tN-1}) = \gamma$.  Then, we define the random partition function 
\begin{equation*}
Z_t^N(\omega) := e^{- \frac{\beta^2}{2} t} \E \biggl[ \exp \biggl( \beta \sum_{i=0}^{tN - 1} \omega(S_i,i) \biggr) \biggr], 
\end{equation*}
where $\beta > 0$ is a small parameter and $\E$ is the expectation taken with respect to $\Sc$. 
Then we have
\begin{align*} 
Z_t^N(\omega) =& e^{-\frac{\beta^2}{2} t} \sum_{\gamma \in \Gamma_{tN}} p(\gamma) \exp \biggl(\beta \sum_{i=0}^{tN-1} \omega(\gamma_i,i) \biggr) \\
=& \sum_{\gamma \in \Gamma_{tN}} p(\gamma) \prod_{i=0}^{tN-1} \left[ \exp(\beta \omega(\gamma_i,i)) \exp \left(-\frac{\beta^2}{2N} \right)\right]. 
\end{align*}
%
We now state and prove a discrete-time version of Theorem~\ref{thm:continuous_Talagrand}.
\bigskip
\begin{theorem}       \label{thm:gci_discrete}
For $\beta$ sufficiently small, there exists a constant $c > 0$ such that for all $N \in \N$ and $t > 0$ with $tN \in \N$, 
\begin{equation*}
Q^N \left(Z_t^N < e^{-u} \right) \leq c e^{- u^2 / c}, \quad \forall u > 0. 
\end{equation*}
\end{theorem}
%

The proof of this theorem relies on three technical lemmas (Lemmas~\ref{lm:q_bounds}--\ref{lm:Gaussian_ci}) following the approach in~\cite{CarmonaHu} and~\cite{Talagrand03}.

Let $\Sc^1$ and $\Sc^2$ be two independent copies of the random walk $\Sc$.  Set 
\begin{equation*}
L_n := \sum_{i=0}^{n-1} \mathbbm{1}_{\{S^1_i = S^2_i\}} 
\quad \text{and} \quad
L_{\infty} := \sum_{i=0}^{\infty} \mathbbm{1}_{\{S^1_i = S^2_i\}}.
\end{equation*}
Since $\Sc$ is transient on $\Z^d$ for $d \geq 3$, we have $L_{\infty} < \infty$ almost surely. Also consider the random walk $\Dc := \Sc^1 - \Sc^2 = (D_n)_{n\geq0}$. It is transient because $\Sc^1$ and $\Sc^2$ are. Let $\tau_1$ denote the time of first return of the random walk $\Dc$ to $0$ and set
\begin{equation*}
q := \Psf(\tau_1 < \infty).
\end{equation*}
It is important to note that $q$ depends on $N$. Clearly, $q > 0$ and because of transience we also have $q < 1$. The following lemma gives more precise bounds on $q$. 

\begin{lemma}   \label{lm:q_bounds}
There are positive constants $c_1$ and $c_2$ with $c_1 < c_2$, depending only on the dimension $d$, such that 
\begin{equation*}
1 - \frac{c_2}{N} \leq q \leq 1 - \frac{c_1}{N}.
\end{equation*}
\end{lemma}



\begin{proof}
For $\tbf\in \R^d$, let $\varphi(\tbf)$ denote the characteristic function of $D_1$. By properties of characteristic functions (see e.g.~\cite[p.194]{Durrett}),
\begin{equation*}
\Psf(D_n = 0) = \frac{1}{(2\pi)^d} \int_{(-\pi, \pi)^d} \varphi(\tbf)^n \ d \tbf, \quad n \in \N_0.
\end{equation*}
It is not hard to see that
\begin{equation}   \label{eq:f_formula}
\frac{1}{1 - q} = \sum_{n=0}^{\infty} q^n = \sum_{n = 0}^{\infty} \Psf(D_n = 0) = \frac{1}{(2 \pi)^d} \int_{(-\pi, \pi)^d} \frac{1}{1 - \varphi(\tbf)} \ d \tbf. 
\end{equation}
For any $\tbf \in (-\pi, \pi)^d$, 
\begin{align*} 
1 - \varphi(\tbf) =& \E\left[ 1 - e^{i \langle \tbf, D_1 \rangle} \right] = \sum_{y \in \Z^d} \Psf(D_1 = y) \left( 1 - e^{i \langle \tbf, y \rangle} \right) \\
=&  \sum_{y \in \Z^d} \Psf(D_1 = y) \left(1 - \cos(\langle \tbf, y \rangle) \right).
\end{align*}
Let us decompose the integral on the right-hand side of~\eqref{eq:f_formula} into 
\begin{equation}    \label{eq:two_integral}
\int_{(-\frac{\pi}{3}, \frac{\pi}{3})^d} \frac{d \tbf}{1 - \varphi(\tbf)} + \int_{(-\pi, \pi)^d \setminus (-\frac{\pi}{3}, \frac{\pi}{3})^d} \frac{d \tbf}{1 - \varphi(\tbf)}.
\end{equation}
Fix $\tbf \in (-\tfrac{\pi}{3}, \tfrac{\pi}{3})^d$. Since 
$
1 - \cos(x) \geq \frac{x^2}{4} 
$
for $x \in (-\tfrac{\pi}{3}, \tfrac{\pi}{3})$, we have
\begin{equation*}
1 - \varphi(\tbf) \geq \frac{1}{4} \sum_{y \in \Z^d: \langle \tbf, y \rangle \in (-\frac{\pi}{3}, \frac{\pi}{3})} \Psf(D_1 = y) \ \langle \tbf, y \rangle^2.
\end{equation*}
As long as $\tbf$ is not the zero vector, the expression on the right can be written as 
\begin{equation}   \label{eq:lower_bd_char_fun}
\frac{\| \tbf \|^2}{4} \sum_{y \in \Z^d: \langle \tbf, y \rangle \in (-\frac{\pi}{3}, \frac{\pi}{3})} \Psf(D_1 = y) \left \langle  \frac{\tbf}{\| \tbf \|}, y \right \rangle^2.
\end{equation}
As $\tfrac{\tbf}{\| \tbf \|}$ lies on the unit sphere in $\R^d$, there is $i \in \{1, \ldots, d\}$ such that 
\begin{equation*}
\left \langle \frac{\tbf}{\| \tbf \|}, e_i \right \rangle^2 \geq \frac{1}{d}, 
\end{equation*}
where $e_i$ is the $i^{\rm{th}}$ unit vector in $\R^d$.
Accordingly, using \eqref{eq:trans_prob_dis}, the expression in~\eqref{eq:lower_bd_char_fun} is bounded below by   
\begin{equation*}
\frac{\| \tbf \|^2}{4d} \Psf(D_1 = e_i) = \frac{\| \tbf \|^2}{4 d^2} \frac{N}{(N+1)^2}
\end{equation*}

\noindent and the first integral in~\eqref{eq:two_integral} is bounded above by
\begin{equation*}
4 d^2 \frac{(N+1)^2}{N} \int_{(-\frac{\pi}{3}, \frac{\pi}{3})^d} \frac{d \tbf}{\| \tbf \|^2}
<\hat c_1 N
\end{equation*}
for some constant $\hat c_1$ depending on $d$.

Now, let $\tbf \in (-\pi, \pi)^d \setminus (-\tfrac{\pi}{3}, \tfrac{\pi}{3})^d$. In this case, there is $i \in \{1, \ldots, d\}$ such that $\lvert t_i \rvert \geq \tfrac{\pi}{3}$, and $1-\varphi(\tbf)$ is bounded below by 
$$
\Psf(D_1 = e_i) \left(1 - \cos(\langle \tbf, e_i \rangle) \right) = \frac{N}{d (N+1)^2} (1 - \cos(t_i)) \geq \frac{N}{2d (N+1)^2}.
$$
Therefore, the second integral in~\eqref{eq:two_integral} is bounded above by
\begin{equation*}
2d \frac{(N+1)^2}{N} \int_{(-\pi, \pi)^d \setminus (-\frac{\pi}{3}, \frac{\pi}{3})^d} \ d \tbf
<\tilde c_1 N
\end{equation*}
for some constant $\tilde c_1$ depending on $d$. Letting $c_1:= \hat c_1 + \tilde c_1$, from \eqref{eq:f_formula} we obtain the following upper bound for $q$:
\begin{equation*}
q \leq 1 - \frac{c_1}{N}. 
\end{equation*}
To obtain a lower bound for $q$, we estimate
$$
1 - \varphi(\tbf) \leq \sum_{y \in \Z^d} \Psf(D_1 = y) \ \langle \tbf, y \rangle^2 \leq \sum_{y \in \Z^d \setminus \{0\}} \Psf(D_1 = y) \ \| \tbf \|^2 \| y \|^2.
$$
Since
\begin{equation*}
\sum_{y \in \Z^d \setminus \{0\}} \Psf(D_1 = y) \ \| y \|^2 = \frac{2N}{(N+1)^2} + \frac{2}{d (N+1)^2}  + \frac{2d -2}{d(N+1)^2} \lesssim \frac{1}{N},
\end{equation*}
it follows from \eqref{eq:f_formula} that
\begin{equation*}
\frac{1}{1 - q} \gtrsim \frac{N}{(2 \pi)^d} \int_{(-\pi, \pi)^d} \frac{d \tbf}{\| \tbf \|^2} = \tilde c_2 N
\end{equation*}
for some $\tilde c_2 > 0$ depending on $d$. Therefore, there is a constant $c_2$ such that 
$$
q \geq 1 - \frac{c_2}{N}.
$$
\end{proof}

For the constant $c_1$ from Lemma~\ref{lm:q_bounds}, one has 
$$
0 < c_1 \leq \inf_N (N (1-q)) \leq \inf_N (N \ln(1/q)). 
$$
We will see that in order for the conclusions of Theorems~\ref{thm:gci_discrete} and~\ref{thm:continuous_Talagrand} to hold, it is enough to assume $\beta^2 < \inf_N (N \ln(1/q))$ as this implies $e^{\frac{\beta^2}{N}} q < 1$ for every $N$ (see also Theorem~1.5 in~\cite{CarmonaHu}). For the remainder of Section~\ref{sec:proof_continuous_Talagrand}, we shall therefore make the following assumption. 

\begin{assumption}
    One has $\beta^2 < \inf_N (N \ln(1/q))$. 
\end{assumption}

It is easy to see that the expected value of the partition function $Z_t^N$ with respect to the disorder $(\omega(z,k))_{z \in \Z^d, k \in \N_0}$ is equal to $1$. Now, we show that the second moment of $Z_t^N$ with respect to the disorder is bounded uniformly in $t$ and $N$.

\begin{lemma}    \label{lm:bounded_second_mom}
There is a constant $c_3 > 0$ such that for all $N \in \N$ and $t > 0$ with $tN \in \N$, 
\begin{equation*}
\left \langle (Z_t^N)^2 \right \rangle_N \leq c_3.
\end{equation*}
\end{lemma}

\begin{proof} We have  
\begin{align*}
\left \langle (Z_t^N)^2 \right \rangle_N 
=& e^{- \beta^2 t} \biggl \langle \E \biggl[ \exp \biggl( \beta \sum_{j=1}^2 \sum_{i=0}^{tN-1} \omega(S_i^j, i) \biggr) \biggr] \biggr \rangle_{\hspace{-1.5mm}N} \\
=&  e^{-\beta^2 t} \E \biggl[ \prod_{i=0}^{tN-1} \biggl \langle \exp \biggl( \beta \Bigl( \omega(S^1_i, i) + \omega(S^ 2_i, i) \Bigr) \biggr) \biggr \rangle_{\hspace{-1.5mm}N} \biggr] \\
=& \E \biggl[ \prod_{i=0}^{tN-1} \left( e^{\frac{\beta^2}{N}} \mathbbm{1}_{\{S^1_i = S^2_i\}} + \mathbbm{1}_{\{S^1_i \neq S^2_i\}} \right) \biggr] = \E \left[ e^{\frac{\beta^2}{N} L_{tN}} \right] 
\leq \E \left[ e^{\frac{\beta^2}{N} L_{\infty}} \right].
\end{align*}
Since 
$\Psf(L_{\infty} = k) = q^{k-1} (1 - q)$ and $e^{\frac{\beta^2}{N}} q < 1$, 
we have  
\begin{equation}   \label{eq:exp_L_inf}
\E \left[ e^{\frac{\beta^2}{N} L_{\infty}} \right] 
= e^{\frac{\beta^2}{N}} (1-q) \sum_{k = 1}^{\infty} \left( e^{\frac{\beta^2}{N}} q \right)^{k-1} = e^{\frac{\beta^2}{N}} \frac{1-q}{1 - e^{\frac{\beta^2}{N}} q}. 
\end{equation}
By Lemma~\ref{lm:q_bounds}, the right-hand side of~\eqref{eq:exp_L_inf} is dominated by 
$$ 
\frac{c_2}{N} \frac{1}{e^{-\frac{\beta^2}{N}} - 1 + \frac{c_1}{N}},  
$$
which converges to $c_2/(c_1 - \beta^2)$ as $N \to \infty$ and is thus bounded in $N$. 
\end{proof}


\begin{lemma}        \label{lm:c_2} 
There exists a constant $C > 0$ such that for all $N \in \N$ and $t > 0$ with $tN \in \N$,  
\begin{equation*}
Q^N \biggl(Z_t^N \geq \frac{1}{2}; \; \E \biggl[ L_{tN} \exp \biggl( \beta \sum_{i=0}^{tN-1} \sum_{j=1}^2 \omega(S^j_i,i) \biggr) \biggr] \leq C N e^{\beta^2 t} (Z_t^N)^2 \biggr) 
> \frac{1}{C}. 
\end{equation*}
\end{lemma}

\begin{proof}	
By the Paley--Zygmund inequality~\cite{paley_zygmund_1932}, $\langle Z_t^N \rangle_N = 1$, and Lemma~\ref{lm:bounded_second_mom}, we have
\begin{equation*}
Q^N \left(Z_t^N \geq \frac{1}{2} \right) \geq \frac{1}{4\langle (Z_t^N)^2 \rangle_N} \geq \frac{1}{4 c_3}. 
\end{equation*}
Let $C$ be any positive number. We will only impose a restriction on $C$ towards the end of the proof. Then the left-hand side of the asserted inequality in Lemma~\ref{lm:c_2} is greater than or equal to 
\begin{equation}      \label{eq:lower_bd_mu} 
\frac{1}{4 c_3} - Q^N \biggl( \E \biggl[ L_{tN} \exp \biggl(\beta \sum_{i=0}^{tN-1} \sum_{j=1}^2 \omega(S^j_i, i) \biggr) \biggr] > \frac{C}{4} N e^{\beta^2 t} \biggr)\hspace{-0.5mm}.    
\end{equation}
Markov's inequality implies 
\begin{align*}
& Q^N \biggl( \E \biggl[ L_{tN} \exp \biggl(\beta \sum_{i=0}^{tN-1} \sum_{j=1}^2 \omega(S^j_i, i) \biggr) \biggr] > \frac{C}{4} N e^{\beta^2 t} \biggr) \\
\leq& \frac{4}{C N} e^{-\beta^2 t} \biggl \langle \E\biggl[ L_{tN} \exp \biggl(\beta \sum_{i=0}^{tN-1} \sum_{j=1}^2 \omega(S^j_i, i) \biggr) \biggr] \biggr \rangle_{\hspace{-1.5mm}N} \\
=& \frac{4}{C N} e^{-\beta^2 t} \ \E\biggl[ L_{tN} \prod_{i=0}^{tN-1} \biggl \langle \exp \Big(
\beta (\omega(S^1_i, i) + \omega(S_i^2, i) \Big) \biggr \rangle_{\hspace{-1.5mm}N} \biggr] \\
\leq& \frac{4}{C N} \E \left[ L_{\infty} e^{\frac{\beta^2}{N} L_{\infty}} \right], 
\end{align*}
where the last step follows from the proof of Lemma~\ref{lm:bounded_second_mom}. Continuing to proceed as in the proof of Lemma~\ref{lm:bounded_second_mom}, one has 
\begin{align*}
\frac{4}{N}    \E \left[L_{\infty} e^{\frac{\beta^2}{N} L_{\infty}} \right] =& 
\frac{4}{N} e^{\frac{\beta^2}{N}} (1-q) \sum_{k=1}^{\infty} k \left( e^{\frac{\beta^2}{N}} q \right)^{k-1} \\
=& \frac{4}{N} e^{\frac{\beta^2}{N}} \frac{1-q}{(1 - e^{\frac{\beta^2}{N}} q)^2}
\\
\leq& e^{\frac{\beta^2}{N}} \frac{4 c_2}{N^2} \frac{1}{(1 - e^{\frac{\beta^2}{N}} + e^{\frac{\beta^2}{N}} \frac{c_1}{N})^2}, 
\end{align*}
and as $N \to \infty$, the expression on the right-hand side converges to 
$$
\frac{4 c_2}{(c_1 - \beta^2)^2}. 
$$
Hence, there is $\tilde C > 0$, independent of $C$ and $N$, such that the expression in~\eqref{eq:lower_bd_mu} is greater than 
\begin{equation*}
\frac{1}{4 c_3} - \frac{\tilde C}{C}, 
\end{equation*}
which is greater than $\tfrac{1}{C}$ if $C$ is chosen sufficiently large. 
\end{proof}

\noindent Note that for any fixed $N \in \N$ and $t > 0$ with $t N\in\N$, the random walk $\Sc$ cannot leave the box 
\begin{equation*}
B_{tN} := \left\{z \in \Z^d: \ \| z \|_1 \leq tN \right\} 
\end{equation*}
before time $tN$, so $Z_t^N$ only depends on $(\omega(z,k))_{z \in B_{tN}, 0 \leq k \leq tN}$. Let $\Xi_t^N$ be the collection of arrays $\xi = (\xi(z,k))_{z \in B_{tN}, 0 \leq k \leq tN}$ of real numbers indexed by $z \in B_{tN}$ and $k \in \{1, \ldots, tN\}$. For $\xi \in \Xi_t^N$ define
$$
Z_t^N(\xi) :=\; e^{-\frac{\beta^2}{2} t} \E \biggl[ \exp \biggl(\beta \sum_{i=0}^{tN - 1} \xi(S_i, i) \biggr) \biggr]
$$
and define $X_t^N$ as 
$$
\Biggl\{ \xi \in \Xi^N_t: Z_t^N(\xi) \geq \frac{1}{2}; \; \E \biggl[ L_{tN} \exp \biggl( \beta \sum_{i=0}^{tN-1} \sum_{j=1}^2 \xi(S^j_i, i) \biggr) \biggr] \leq C N e^{\beta^2 t} Z_t^N(\xi)^2 \Biggr\},
$$
where $C$ is the constant from Lemma~\ref{lm:c_2}.
\noindent Let $\omega_{tN}$ be the random vector \\ $(\omega(z,i))_{z \in B_{tN}, i \in \{0,\ldots,tN\}}$. Then, by Lemma~\ref{lm:c_2},  
\begin{equation*}
Q^N(\omega_{tN} \in X_t^N) > \frac{1}{C}.
\end{equation*}
Finally, for $m \in \N$, $g, h \in \R^m$, and $A \subset \R^m$ measurable, define 
$$
d(g,h) := \| g - h \|   
\quad \text{and} \quad 
d(g,A) := \inf_{h \in A} \| g - h \|. 
$$
The following lemma is a consequence of the Gaussian concentration inequality (see, e.g.,~\cite{Talagrand03} and~\cite[Theorem 1.3.4]{Talagrand_11}). 


\bigskip
\begin{lemma}     \label{lm:Gaussian_ci}
Let $C$ be the constant from Lemma~\ref{lm:c_2}. For any $v > 0$, 
\begin{equation*}
Q^N\biggl( d(\omega_{tN}, X_t^N) > v + \sqrt{\frac{2}{N} \ln (2C)} \biggr) \leq 2 e^{-\frac{N}{2} v^2}.
\end{equation*}
\end{lemma}

\begin{proof}	
As the components of $\sqrt{N} \omega_{tN}$ are i.i.d. standard normal, the Gaussian concentration inequality (\cite[Theorem~1.3.4]{Talagrand_11}) yields, for any $v>0$,
\begin{equation}   \label{eq:gci}
Q^N \biggl( \biggl \lvert d(\omega_{tN},X_t^N) - \Big\langle d(\omega_{tN}, X_t^N) \Big\rangle_{\hspace{-0.5mm}N} \biggr\lvert > v \biggr) \leq 2e^{-\frac{N}{2} v^2}. 
\end{equation}
Suppose that $v < \big\langle d(\omega_{tN}, X_t^N) \big\rangle_N$. Then 
\begin{equation*}
\frac{1}{C} < Q^N(\omega_{tN} \in X_t^N) \leq Q^N\biggl( \biggl \lvert d(\omega_{tN}, X_t^N) - \Big\langle d(\omega_{tN}, X_t^N) \Big\rangle_{\hspace{-0.5mm}N} \biggr\lvert > v \biggr) \leq 2e^{-\frac{N}{2} v^2}. 
\end{equation*}
As the inequality $1/C \leq 2 e^{-\frac{N}{2} v^2}$ holds for any $v < \big\langle d(\omega_{tN}, X_t^N) \big\rangle_{\hspace{-0.5mm}N}$, we also have 
\begin{equation*}
\Big\langle d(\omega_{tN}, X_t^N )\Big \rangle_{\hspace{-0.5mm}N} \leq \sqrt{\frac{2}{N} \ln(2C)}. 
\end{equation*}
Together with~\eqref{eq:gci}, we obtain the desired result. 
\end{proof}

\bigskip
\begin{proof}[Proof of Theorem~\ref{thm:gci_discrete}]

We recall our standing assumption that 
$$
\beta^2 < \inf_N (N \ln(1/q)).
$$
Fix $\xi' \in X_t^N$ and $\xi \in \Xi_t^N$. Then we can write 
\begin{equation}   \label{eq:partition_fun}  
Z_t^N(\xi) = e^{-\frac{\beta^2}{2} t} \E \biggl[ \exp \hspace{-0.7mm} \biggl(\beta \sum_{i=0}^{tN - 1} \big(\xi(S_i, i) - \xi'(S_i, i) \big) \biggr) \hspace{-0.5mm} \exp \hspace{-0.5mm} \biggl(\beta \sum_{i=0}^{tN - 1} \xi'(S_i, i) \biggr) \biggr]. \hspace{-1mm}
\end{equation}
If we define the Gibbs measure $\nu$ on the space of path realizations for $\Sc$ up to step $(tN - 1)$ by
\begin{equation*}
\nu(B) = \frac{e^{-\frac{\beta^2}{2} t}}{Z_t^N(\xi')} \E \biggl[ \mathbbm{1}_B(S_0, \ldots, S_{tN - 1}) \exp \biggl(\beta \sum_{i=0}^{tN - 1} \xi'(S_i, i) \biggr) \biggr], 
\end{equation*}
then the right-hand side of \eqref{eq:partition_fun} becomes
\begin{align*}
& Z_t^N(\xi') \int \exp \biggl(\beta \sum_{i=0}^{tN-1} \left( \xi(S_i, i) - \xi'(S_i, i) \right) \biggr) \ d \nu \\
	& \geq  \frac{1}{2} \exp \biggl( - \beta \biggl \lvert \int \sum_{i=0}^{tN-1} \left( \xi(S_i, i) - \xi'(S_i, i) \right) \ d \nu \biggr\lvert \biggr), 
\end{align*}
where we used that $\xi' \in X_t^N$. One has 
$$
\biggl \lvert \int \sum_{i=0}^{tN-1} \left( \xi(S_i, i) - \xi'(S_i, i) \right) \ d\nu \biggr\rvert = \biggl \lvert \sum_{i=0}^{tN - 1} \sum_{z \in B_{tN}} \left( \xi(z,i) - \xi'(z,i) \right) \int \mathbbm{1}_{\{S_i = z\}} \ d \nu \biggr\rvert. 
$$
By the Cauchy--Schwarz inequality, the expression on the right is less than  
$$
d(\xi, \xi') \left( \int \int L_{tN} \ d\nu \ d\nu \right)^{\frac{1}{2}}.
$$
Again using the fact that $\xi' \in X_t^N$, we infer from the definition of $\nu$ the inequality $\int \int L_{tN} \ d \nu \ d \nu \leq CN$ and thus 
\begin{equation*}
Z_t^N(\xi) \geq \tfrac{1}{2} \exp \biggl(-\beta d(\xi, \xi') \left( \int \int L_{tN} \ d\nu \ d\nu \right)^{\frac{1}{2}} \biggr) \geq \tfrac{1}{2} \exp\left( -\beta d(\xi, \xi') \sqrt{C N} \right).
\end{equation*}
As the inequality above holds for any $\xi' \in X_t^N$, 
\begin{equation*}
Z_t^N(\xi) \geq \tfrac{1}{2} \exp \left( -\beta \sqrt{C} \ d(\xi, X_t^N) \sqrt{N} \right).
\end{equation*}
Fix $u > \ln(2)$ and suppose
\begin{equation*}
\frac{u - \ln(2)}{\beta \sqrt{C N}} \geq d(\xi, X_t^N). 
\end{equation*}
Then
\begin{equation*}
Z_t^N(\xi) \geq \tfrac{1}{2} \exp \left( - (u - \ln(2)) \right) = e^{-u}.
\end{equation*}
Accordingly, 
$$
Q^N(Z_t^N \geq e^{-u}) \geq Q^N \left( d(\omega_{tN}, X_t^N) \leq \frac{u - \ln(2)}{\beta \sqrt{C N}} \right). 
$$
By Lemma~\ref{lm:Gaussian_ci} with $v = \frac{u - \ln(2)}{\beta \sqrt{CN}} - \sqrt{2 \ln(2C)/N}$ and $u$ so large that 
$\frac{u - \ln(2)}{\beta \sqrt{C}} > \sqrt{2 \ln(C)}$ and thus $v > 0$, the expression on the right is greater than 
\begin{equation*}
1 - 2 \exp \biggl( - \frac{1}{2} \left( \frac{u - \ln(2)}{\beta \sqrt{C}} - \sqrt{2 \ln(C)} \right)^2 \biggr), 
\end{equation*}
so there is $c > 0$ such that 
\begin{equation*}
Q^N(Z_t^N \geq e^{-u}) \geq 1 - e^{- u^2 / c}.  
\end{equation*}
By choosing $c$ even larger if necessary, one obtains the desired inequality 
$$
Q^N(Z_t^N < e^{-u}) \leq c e^{-u^2/c} 
$$
for every $u > 0$. 
\end{proof}

\subsection{The continuous-time case}\label{continuous_Talagrand}

In this section we present the proof of Theorem~\ref{thm:continuous_Talagrand}. 
We will show that there is a constant $c > 0$ such that 
$$ 
Q(Z_{0,0}^t < e^{-u}) < c e^{-u^2 / c}, \quad \forall t, u > 0. 
$$
To simplify notation, we write $Z^t$ instead of $Z_{0,0}^t$. We begin with a result on continuity in time for the partition function.  


\begin{lemma}\label{lm:continuity_of_Zt}
Fix $t > 0$. For any $\eps>0$, there exists $s_\eps >0$ such that 
\begin{equation*}
Q(\lvert Z^{t+s}- Z^t \rvert > \eps) < \eps \ \text{ for all } s \in (0, s_\eps).
\end{equation*}
\end{lemma}

\begin{proof}
Recall that $\AC_s^t$ is the action defined by~\eqref{eq:def_action}. For $s > 0$, we have
\begin{eqnarray*}
	\langle ( Z^{t+s} - Z^t)^2 \rangle  
	&=& e^{-{\beta^2} t} 
\biggl\langle \left(\E_{0,0} \left[ \exp (\beta \AC_0^t) \left( e^{-\frac{\beta^2}{2} s} \exp(\beta \AC_t^{t+s})-1 \right) \right] \right)^2 \biggr\rangle\\
	&\leq & e^{-{\beta^2} t} \ \E_{0,0} \biggl[ \left\langle \exp(2\beta \AC_0^t)\right\rangle \biggl\langle  \left( e^{-\frac{\beta^2}{2} s}  \exp(\beta \AC_t^{t+s})-1 \right)^2 \biggr\rangle \biggr]\\
	&=& e^{\beta^2 t}( e^{\beta^2 s} -1 ).
\end{eqnarray*}
Let $s_\eps>0$ be so small that $e^{\beta^2 t}( e^{\beta^2 s} -1 ) < \eps^3$ for all $s \in (0,s_\eps)$. Then, by Markov's inequality, 
\begin{equation*}
\tag*{\qedhere}
Q(\lvert Z^{t+s} - Z^t \rvert > \eps) \leq \frac{ \langle ( Z^{t+s} - Z^t)^2 \rangle}{\eps^2} < \eps. 
\end{equation*}
\end{proof}

Let $N\in \N$ and $t>0$ such that $tN \in \N$.  We will now represent the partition function $Z_t^N$ from Subsection~\ref{ssec:Talagrand} in a way that mimicks the definition of $Z^t$. Let $\Sc^N = (S^N_i)_{i \in \N_0}$ denote the random walk $\Sc$ from Subsection~\ref{ssec:Talagrand}. The change in notation reflects that we now want to vary $N$. For $s \geq 0$, we define the continuous-time random walk $\eta^N$ by
$$ 
\eta^N_s := S^N_i \; \text{ if $s \in [\tfrac{i}{N}, \tfrac{i+1}{N})$}.
$$ 
 A sample path of $\eta^N$ over the time interval $[0,t)$ is characterized by the number of actual jumps $n^N_t$ (i.e., jumps from a site $x$ to a site $y \neq x$) that occur in $(0,t)$, the embedded discrete-time path $\gamma^N = (\gamma^N_0, \gamma^N_1, \ldots, \gamma^N_{n^N_t})$ on $\Z^d$ such that $\|\gamma^N_j - \gamma^N_{j-1}\|_1 = 1$ for $1 \leq j \leq n^N_t$, and the jump times $0 <  s^N_1 < \ldots < s^N_{n^N_t} < t$, which are of the form $\tfrac{i}{N}$. To keep the definition of the associated action compact, we denote $s_0^N := s$ and $s_{n_t^N +1}^N := t$. To such a sample path of $\eta^N$, we assign the action 
$$ 
\Asc^N_t := \sum_{j=0}^{n^N_t} \left(W^{\gamma^N_j}_{s^N_{j+1}} - W^{\gamma^N_j}_{s^N_j} \right). 
$$ 
Next, we define the probability measure  
\begin{equation*}
g_N \left( \left\{ \frac{k}{N} \right\} \right) := \frac{1}{N+1} \left(\frac{N}{N+1} \right)^{k-1}, \quad k \in \N. 
\end{equation*}
The jump times of $\eta^N$ can be represented as 
\begin{equation*}
s_k^N = \sum_{j=1}^k \tau^N_j, \quad k \in \N_0,  
\end{equation*}
where $(\tau^N_j)_{j \in \N}$ is an i.i.d. sequence of random variables distributed according to $g_N$. As $N \to \infty$, $g_N$ converges weakly to the exponential distribution with intensity $1$. 
 As a result, for any $k \in \N$, $(\tau^N_1, \ldots, \tau^N_k)$ converges weakly to $(\tau_1, \ldots, \tau_k)$ as $N \to \infty$, where $(\tau_j)_{j \in \N}$ is a sequence of independent exponentially distributed random variables with intensity $1$.  
Also note that the partition function $Z_t^N$ from Subsection~\ref{ssec:Talagrand} has the same distribution under $Q^N$ as  
\begin{equation}    \label{eq:Z_N_defi}
e^{-\frac{\beta^2}{2} t} \E_{\tau^N} \E_{\gamma} e^{\beta \Asc^N_t} 
\end{equation}
under $Q$. Here, $\E_{\tau^N}$ denotes expectation with respect to $(\tau^N_j)_{j \in \N}$ and $\E_{\gamma}$ averages with respect to the sample paths of a discrete-time simple symmetric random walk on $\Z^d$.


\begin{lemma}\label{lm:D_small}
Fix $t \in (0,\infty) \cap \Q$. For any $\eps > 0$, there exist $M_\eps, N_{\eps} \in \N$ such that for every $N \in \N$ with $N \geq N_{\eps}$ and $tN \in \N$, one has  
\begin{equation*}
Q \left( e^{-\frac{\beta^2}{2} t} \E_{\tau^N} \left[ \E_{\gamma}\left[ \exp(\beta \Asc_t^N)\right] \id_{n^N_t > M_\eps} \right] > \eps \right) < \eps. 
\end{equation*}
\end{lemma}


\begin{proof}
For any $M, N \in \N$ with $tN \in \N$, we have 
\begin{equation}   \label{eq:exp_P} 
\left\langle e^{-\frac{\beta^2}{2} t} \E_{\tau^N} \left[\E_{\gamma}\left[ \exp(\beta \Asc_t^N)\right] \id_{n^N_t > M} \right]\right\rangle = \Psf_{\tau^N}(n^N_t > M). 
\end{equation} 
Recall from Section~\ref{sec:setting} that $n_{t}$ is the number of jumps that occur within the finite time interval $(0,t)$ for a sample path of the continuous-time simple symmetric random walk $\eta$. By weak convergence,
$$
\lim_{N \to \infty} \Psf_{\tau^N} \big(n^N_t > M \big) = \Psf\big( n_{t} > M \big).  
$$
Choose $M_\eps \in \N$ so large that 
$$ 
\Psf\Big(n_{t} > M_\eps\Big) < \frac{\eps^2}{2}. 
$$ 
Then there exists $N_{\eps} \in \N$ such that for every $N \in \N$ with $N \geq N_{\eps}$ and $tN \in \N$,   
\begin{equation*}
\Psf_{\tau^N}(n^N_t > M_{\eps}) < \eps^2, 
\end{equation*}
and we can conclude using~\eqref{eq:exp_P} and Markov's inequality. 
\end{proof}

\begin{proof}[Proof of Theorem~\ref{thm:continuous_Talagrand}]  
We recall our standing assumption 
$$
\beta^2 < \inf_N (N \ln(1/q)).
$$
By Theorem~\ref{thm:gci_discrete}, there exists $c > 0$ such that for all $N \in \N$ and $t > 0$ with $tN \in \N$, 
\begin{equation}    \label{eq:continuous_Talagrand} 
Q^N(Z^N_t < e^{-u}) \leq c e^{-u^2/c}, \quad \forall u > 0. 
\end{equation} 
Let us fix $t > 0$ and $\eps > 0$. By Lemma~\ref{lm:continuity_of_Zt}, there is $s_{\eps} > 0$ such that 
$$
Q(\lvert Z^{t+s} - Z^t \rvert > \eps) < \eps, \quad \forall s \in (0, s_{\eps}). 
$$
Let $s\in (0,s_{\eps})$ such that $t+s \in \Q$. For $u > 0$, one has 
\begin{align}
Q\big( Z^t < e^{-u} \big) 
\notag	&= Q\Big(Z^t < e^{-u}; \lvert Z^{t+s} - Z^t \rvert > \eps) + Q(Z^t < e^{-u}; \; \lvert Z^{t+s} - Z^t \rvert \leq \eps\Big) \\ 
	&\leq Q\Big( \lvert Z^{t+s} - Z^t \rvert > \eps \Big) + Q\Big( Z^{t+s} < e^{-u} + \eps \Big) \notag \\
    &< \eps + Q\Big( Z^{t+s} < e^{-w_{\eps}} \Big), \label{eq:QZt_first_bound}
\end{align}
where $w_{\eps} := -\ln(\eps+e^{-u})$. Since $t + s \in \Q$, Lemma~\ref{lm:D_small} implies that there exist $M_{\eps}, N_{\eps} \in \N$ such that 
\begin{equation}    \label{eq:D_small}
Q \biggl(e^{-\frac{\beta^2}{2} (t+s)} \E_{\tau^N} \left[ \E_{\gamma} \left[ \exp(\beta \AC_{t+s}^N) \right] \id_{n^N_{t+s} > M_{\eps}} \right] > \eps \biggr) < \eps 
\end{equation} 
for every $N \in \N$ with $N \geq N_{\eps}$ and $(t+s) N \in \N$. 
Consider the event 
\begin{equation*}
A := \biggl\{ \omega \, : \, e^{ - \frac{\beta^2}{2} (t+s)} \E_\tau \left[  \E_{\gamma} \left[  \exp(\beta \Asc_0^{t+s})  \right] \id_{ n_{t+s} \leq M_\eps} \right] < e^{- w_\eps} \biggr\},
\end{equation*}
where $\E_{\tau}$ denotes expectation with respect to $(\tau_j)_{j \in \N}$, an i.i.d. sequence of exponentially distributed random variables with intensity $1$. 

As $\big \{ \omega \, : \,  Z^{t+s}(\omega) < e^{-w_{\eps}} \big\} \subset A$, we have 
\begin{equation} \label{eq:bounded_by_QA}
Q\big(Z^{t+s} < e^{-w_{\eps}}\big) \leq Q(A). 
\end{equation}
For a fixed realization $\omega$ of the disorder $\big (W^z \big)_{z \in \Z^d}$ and for $k \in \N_0$, we define the map 
\begin{align*}
&\varphi^{\omega}_k(t_1, \ldots, t_{k+1}) 
:= \\
&\id_T e^{-\frac{\beta^2}{2} (t+s)} 
 \E_{\gamma} \biggl[\exp \biggl(\beta \sum_{i=0}^{k-1} \left( W^{\gamma_i}_{\sum_{j=1}^{i+1} t_j} - W^{\gamma_i}_{\sum_{j=1}^i t_j} \right)+\beta \left( W^{\gamma_k}_{t+s} - W^{\gamma_k}_{\sum_{j=1}^k t_j} \right) \biggr) \biggr],  
\end{align*}
where $T :=(0,\infty)^{k+1} \cap \{\sum_{j=1}^k t_j < t+s \leq \sum_{j=1}^{k+1} t_j\}$. Path continuity of Brownian motion implies that for every $k \in \N_0$, the function $\varphi_k^{\omega}$ is bounded for $Q$-almost every $\omega$. 
Moreover, $Q$-almost surely, the set of discontinuities of $\varphi_k^{\omega}$ has measure zero with respect to the law of $(\tau_1,\ldots,\tau_{k+1})$. 
Thus, by the Portmanteau Theorem (see, e.g. \cite[Theorem 13.16]{klenkeprobability}), $Q$-almost surely,   
\begin{equation*}
\lim_{N \to \infty} \E_{\tau^N} \left[ \varphi_k^{\omega}(\tau_1^N, \ldots, \tau^N_{k+1}) \right] = \E_{\tau} \left[ \varphi^{\omega}_k(\tau_1, \ldots, \tau_{k+1}) \right], \quad  \forall k \in \N_0. 
\end{equation*}
In particular, 
$$ 
Q \biggl(\bigcap_{k=0}^{M_\eps} \left\{\omega: \lim_{N \to \infty} \E_{\tau^N} [ \varphi^{\omega}_k(\tau_1^N, \ldots, \tau_{k+1}^N)] = \E_{\tau}[\varphi^{\omega}_k(\tau_1, \ldots, \tau_{k+1})] \right\} \biggr) = 1. 
$$
As 
\begin{align*}
& \bigcap_{k=0}^{M_\eps} \left\{\omega: \lim_{N \to \infty} \E_{\tau^N} \left[ \varphi^{\omega}_k(\tau_1^N, \ldots, \tau^N_{k+1}) \right] = \E_{\tau} \left[ \varphi^{\omega}_k(\tau_1, \ldots, \tau_{k+1}) \right] \right\}\\
\subset& \bigcap_{k=0}^{M_\eps} \bigcup_{N=1}^\infty\bigcap_{j=N}^\infty  
\biggl\{\omega: \left \lvert \E_{\tau^j}[ \varphi^{\omega}_k(\tau_1^j, \ldots, \tau^j_{k+1})] - \E_{\tau}[ \varphi^{\omega}_k(\tau_1, \ldots, \tau_{k+1})] \right\rvert < \frac{\eps}{M_\eps+1} \biggr\},
\end{align*}
there is an integer $\tilde N_\eps \geq N_{\eps}$ such that for all $N\geq \tilde N_\eps$,
\begin{equation*}
Q\left(B_N \right) > 1-\eps,
\end{equation*}
where
 \begin{equation*}
B_N :=
 \bigcap_{j=N}^\infty \bigcap_{k=0}^{M_\eps} \biggl\{ \omega: 
\left \lvert \E_{\tau^j}[\varphi_k^\omega (\tau_1^j,\ldots, \tau_{k+1}^j)]  - \E_{\tau}[\varphi_k^\omega (\tau_1,\ldots, \tau_{k+1}) ] \right\rvert < \frac{\eps}{M_\eps+1} \biggr\}.
\end{equation*}

Consider $N> \tilde N_\eps$ such that $(t+s)N\in\N$. 
Since $Q(B_N^c) < \eps$, we have
\begin{equation}\label{eq:bound_on_QA}
Q(A) = Q(A\cap B_N) + Q(A \cap B_N^c) \leq Q(A \cap B_N) + \eps. 
\end{equation}
For $\omega \in B_N$, 
\begin{align*}
&e^{-\frac{\beta^2}{2} (t+s)} \biggl | \E_\tau \biggl[ \id_{\{n_{t+s}\leq M_\eps)\}} \E_{\gamma} \left[e^{\beta \AC_0^{t+s}} \right] \biggr] 
-
\E_{\tau^N} \biggl[ \id_{n^N_{t+s}\leq M_\eps} \E_{\gamma} \left[e^{\beta \AC_{t+s}^N} \right] \biggr]  \biggr |\\
	\leq& \sum_{k=0}^{M_\eps} \Big\lvert \E_\tau[ \varphi_k^{\omega}(\tau_1, \ldots, \tau_{k+1})] - \E_{\tau^N}[\varphi_k^{\omega}(\tau_1^N, \ldots, \tau^N_{k+1}) ] \Big\rvert < (M_\eps + 1)\frac{\eps}{M_\eps +1}	=\eps. 
\end{align*}
Accordingly, 
\begin{equation*}
A\cap B_N \subset  C := \biggl\{ e^{-\frac{\beta^2}{2} (t+s)}  \E_{\tau^N}\biggl[ \id_{n^N_{t+s}\leq M_\eps} \E_{\gamma} \left[e^{\beta \AC_{t+s}^N} \right] \biggr] < e^{-v_{\eps}} \biggr\} , 
\end{equation*}
where $v_{\eps} := -\ln(\eps+e^{-w_{\eps}})$.  Let
\begin{equation*}
D := \biggl\{ e^{-\frac{\beta^2}{2} (t+s) } \E_{\tau^N} \biggl[ \id_{n_{t+s}^N > M_\eps} \E_{\gamma}\left[e^{\beta \Asc_{t+s}^N}\right] \biggr] > \eps \biggr\}.
\end{equation*}
Since $N > N_{\eps}$, one obtains with~\eqref{eq:D_small} and~\eqref{eq:Z_N_defi}  
\begin{equation} \label{eq:Q_AcapB}
Q(A \cap B_N) \leq Q(C) = Q(C \cap D) + Q(C \cap D^c) < \eps + Q^N(Z_{t+s}^N < e^{-y_{\delta}}), 
\end{equation}
where $y_{\eps} := -\ln(\eps + e^{-v_{\eps}})$. 
On account of~\eqref{eq:continuous_Talagrand}, one has 
$$ 
Q^N(Z_{t+s}^N < e^{-y_{\eps}}) \leq c e^{- y_{\eps}^2 / c}. 
$$ 
Combining this estimate with the estimates~\eqref{eq:QZt_first_bound},~\eqref{eq:bounded_by_QA},~\eqref{eq:bound_on_QA}, and~\eqref{eq:Q_AcapB}, we have 
\begin{equation*}
Q\big( Z^t < e^{-u} \big) \leq 3 \eps + c e^{- y_{\eps}^2 / c}.
\end{equation*}
Since, in addition, $\lim_{\eps \searrow 0} (3 \eps + c e^{-y_{\eps}^2/c}) = c e^{-u^2/c}$, 
we obtain the desired estimate. 
\end{proof}

\begin{appendix} 

\section{Proofs of Estimates for Transition Probabilities}
\label{ssec:transition_prob_proofs}

\subsection{Proof of Lemma~\ref{lm:lclt_con_time}}

Denote the coordinate components of the continuous-time simple symmetric random walk $\eta$ on $\Z^d$ starting at $0$ by $\eta^{(1)}, \ldots, \eta^{(d)}$ and set $\hat \eta^{(i)}_t \coleq \eta^{(i)}_{dt}$ for $1 \leq i \leq d$.
Then Proposition~1.2.2 in~\cite{Lawler_Limic} implies that $\hat \eta^{(1)}, \ldots, \hat \eta^{(d)}$ are independent continuous-time simple symmetric random walks in dimension $1$ starting at $0$.
According to Theorem~2.5.6 in~\cite{Lawler_Limic}, 
$$ 
\Pp(\hat \eta^{(i)}_t = z) = \frac{1}{\sqrt{2 \pi t}} e^{-z^2/(2t)} \exp \left(O\left( \frac{1}{\sqrt{t}} + \frac{\lvert z \rvert^3}{t^2} \right) \right)
$$ 
for $1 \leq i \leq d$, and for all $t > 0$ and $z \in \Z$ such that $\lvert z \rvert \leq \tfrac{t}{2}$. 
Now, if $z = (z_1, \ldots, z_d) \in \Z^d$ such that $\|z\| \leq \tfrac{t}{2 \sqrt{d}}$, the above equality implies~\eqref{eq:lclt_con_time} and the estimate~\eqref{ineq:lclt_con_time} because
$$
	p_t^{z} = \prod_{i=1}^d \Pp \big(\hat \eta^{(i)}_{\frac{t}{d}} = z_i \big).
$$


    
\subsection{Proof of Lemma~\ref{lm:q_iota}}  

Since $\iota(y,l) > \|y\|_1$ for $\|y\| \leq t^{\sublin}$, $l \in J(t-2t^{\xi_1})$, and $t$ sufficiently large, and since $\iota(y,l) \equiv \|y\|_1$, we have $q^y_{\iota(y,l)} > 0$. For $m \in J(2 t^{\xi_1})$, Lemma~\ref{lm:ratio_estimate} implies the estimate 
\begin{align*}
q^y_{m+l} =& \frac{q^y_{m+l}}{q^y_{\iota(y,l)}} q^y_{\iota(y,l)} \\
\leq& \left(1+O(l^{-\frac{2}{5}}) \right) \exp \left(c \left(\frac{\|y\|}{\iota(y,l)} m + \ln(\iota(y,l)) \frac{m}{\iota(y,l)} \right) \right) q^y_{\iota(y,l)} \lesssim q^y_{\iota(y,l)}.
\end{align*}
For the last estimate, we used the assumption that $\xi_1 < 1 - \sublin$. 


\subsection{Proof of Lemma~\ref{lm:uniform_D_conv}} 
On account of~\eqref{eq:lclt_con_time}, for $y \in \Z^d$ there is a constant $C>0$ such that 
$$
\sum_{n=0}^{\infty} e^{-t} \frac{t^n}{n!} q_n^{y-x} = p_t^{y-x} \geq \exp \left(-C t^{2 \sublin -1} \right), \quad \|x\| \leq t^{\sublin}. 
$$
On the other hand, for $\rho > 0$, 
$$
\exp \left(C t^{2 \sublin -1} \right) \sum_{n \notin J(t)} e^{-t} \frac{t^n}{n!} q_n^{y-x} \lesssim e^{(\rho -1)t} \sum_{n \notin J(t)} \frac{t^n}{n!} \sum_{1 \leq r \leq n+1} r \alpha^r A(t,n,r), 
$$
and 
by~(B1) of Proposition~\ref{prop:key},
the expression on the right tends to $0$ as $t \to \infty$ provided that $\rho$ is small enough. 

\subsection{Proof of Lemma~\ref{lm:p_ratio}} 

Choose the parameter $\Jtconst \in (\tfrac{1}{2},1)$ in the definition of $J(t)$ so close to $1$ that $4 (1-\Jtconst) < \beta^2$, and let $\tilde \sublin \in (\sublin,1)$. We have  
$$ 
\limsup_{t \to \infty} \sup_{\|y\| \leq t^{\sublin}} \Dc_{t-2t^{\xi}}(y,0) \leq \lim_{t \to \infty} \sup_{\|y\| \leq (t-2t^{\xi})^{\tilde \sublin}} \Dc_{t-2t^{\xi}}(y,0),  
$$ 
and the righthand side tends to $0$ as $t \to \infty$ by virtue of Lemma~\ref{lm:uniform_D_conv}. Thus, 
\begin{align*}
\frac{p^y_{t-2t^{\xi}}}{p^y_t} =& \frac{\sum_{n \in J(t-2 t^{\xi})} e^{2t^{\xi} - t} \frac{(t-2t^{\xi})^n}{n!} q^y_n (1+\Dc_{t-2t^{\xi}}(y,0))}{\sum_{n=0}^{\infty} e^{-t} \frac{t^n}{n!} q^y_n} \\
\leq& (1+\Dc_{t-2t^{\xi}}(y,0)) e^{2t^{\xi}} \frac{\sum_{n \in J(t-2t^{\xi})} \frac{(t-2t^{\xi})^n}{n!} q^y_n}{\sum_{n \in J(t-2t^{\xi})} \frac{t^n}{n!} q^y_n} \\
\lesssim& e^{2t^{\xi}} \frac{\sum_{n \in J(t-2t^{\xi})} \frac{(t-2t^{\xi})^n}{n!} q^y_n}{\sum_{n \in J(t-2t^{\xi})} \frac{t^n}{n!} q^y_n}. 
\end{align*}
For $n \in J(t-2t^{\xi})$, 
$$ 
\left(\frac{t-2t^{\xi}}{t} \right)^n = (1-2t^{\xi -1})^n \leq (1-2t^{\xi-1})^{\Jtconst (t-2t^{\xi})} \leq (1-2t^{\xi-1})^{(2 \Jtconst -1)t} \leq e^{-2(2 \Jtconst-1) t^{\xi}}. 
$$ 
Accordingly, 
$$ 
\frac{p^y_{t-2t^{\xi}}}{p^y_t} \lesssim  e^{4 (1-\Jtconst) t^{\xi}} \lesssim e^{\beta^2 t^{\xi}}. 
$$

\section{Proofs of Some Preliminary Estimates}
\label{ssec:key_lemma_estimates_proofs}


\subsection{Proof of Lemma~\ref{lm:expected_value}}
Let us first consider the special case $l=0$. If $l=0$, then $r=1$, and 
$$
A(t,0,1) = \E \big[ e^{\beta^2 t} \big| n_t = 0 \big] = e^{\beta^2 t} = \beta^{-2} t^{-1} e^{\beta^2 t} \beta^2 t \lesssim e^{\beta^2 t} \beta^2 t. 
$$
Now, assume that $l > 0$. For $t > 0$, define the $l$-dimensional simplex 
\begin{flalign}    \label{eq:def_simplex} 
	\Delta(t,l) 
		&\coleq \left\{(t_1, \ldots, t_l) \in \R^l_+: t_1 + \ldots + t_l < t \right\}. 
\intertext{and for $1 \leq r \leq l+1$ the integral}
	\Ic(t,l,r) 
		&\coleq \int_{\Delta(t,l)} \prod_{j=1}^{r-2} \left( e^{\beta^2 t_j} - 1 \right) e^{\beta^2 (t_{r-1} + t_r)} \ d t_1 \cdots d t_l,  
\label{190921092510}
\end{flalign}
where $t_0 \coleq 0$ and $t_{l+1} \coleq t - (t_1 + \ldots + t_l)$. 
Then we have the following identity:
\begin{equation}\label{eq:A_I_identity}
A(t,l,r) = \frac{l!}{t^l} \Ic(t,l,r). 
\end{equation}
We claim that if $\beta < 1$ and $t \geq \max \{ \beta^{-2}; 2 \}$, then
\begin{equation}    \label{eq:si_2}
	\Ic(t,l,r) 
		\leq e^{\beta^2 t} \beta^{2 (r-3)} 
			\frac{t^{l+r-2}}{(l+r-2)!},
			\qquad \forall l \in \N, \ 1 \leq r \leq l + 1. 
\end{equation}
This estimate together with~\eqref{eq:A_I_identity} implies~\eqref{eq:A_bound}.

To complete the proof of Lemma~\ref{lm:expected_value}, it remains to prove the estimate~\eqref{eq:si_2}. In several places, we will use the identity 
\begin{equation}     \label{eq:integral}
\int_0^t (t - s)^k \ s^n \ ds = \frac{n! k!}{(n+k+1)!} t^{n +k + 1}
\end{equation}
for $t > 0$ and $n, k \in \N_0$.  
Consider three separate cases: $r = l +1$, $r = l$, and $r < l$.
In the case $r = l$, we use the easily verifiable identity
\begin{equation}\label{190920230435}
\Ic(t,l+1,l+1) = \int_0^t \left( e^{\beta^2 s} -1 \right) \Ic(t-s,l,l) \ ds,
\qquad l \geq 2,
\end{equation}
and argue by induction on $l$ with bases cases $l = 1,2$:
\begin{flalign*}
	\Ic (t, 1,1) &= e^{\beta^2 t} \beta^{-2} - \beta^{-2} \leq e^{\beta^2 t} \beta^{-4},
\\	\Ic (t,2,2) &= \beta^{-2} (e^{\beta^2 t} t - \beta^{-2} e^{\beta^2 t} + \beta^{-2}) \leq \beta^{-2} e^{\beta^2 t} t \leq \beta^{-2} e^{\beta^2 t} \frac{t^2}{2}.
\end{flalign*}
By the induction hypothesis, the right-hand side of~\eqref{190920230435} is dominated by  
\begin{align*}
	& e^{\beta^2 t} \beta^{2(l-3)} \frac{1}{(2(l-1))!} \int_0^t (t-s)^{2(l-1)} \left(1 - e^{-\beta^2 s} \right) \ ds
	\\ &\leq  e^{\beta^2 t} \beta^{2(l-2)} \frac{1}{(2(l-1))!} \int_0^t (t-s)^{2(l-1)} s \ ds 
	= e^{\beta^2 t} \beta^{2(l-2)} \frac{t^{2l}}{(2l)!},
\end{align*}
which implies~\eqref{eq:si_2} for $r = l$.

The case $r = l+1$ is similar: In the base case $l=1$, 
$$
\Ic(t,1,2) = e^{\beta^2 t} t \leq e^{\beta^2 t} \beta^{-2} t. 
$$
And in the induction step, for $l \geq 1$, 
\begin{align*}
\Ic(t,l+1, l+2) =& \int_0^t \left(e^{\beta^2 s} - 1 \right) \Ic(t-s,l,l+1) \ ds \\
\leq& e^{\beta^2 t} \beta^{2 (l-2)}  \frac{1}{(2l-1)!}  \int_0^t (t-s)^{2l-1} \left(1 - e^{-\beta^2 s} \right)  \ ds \\
\leq& e^{\beta^2 t} \beta^{2 (l-1)} \frac{t^{2l+1}}{(2l+1)!}.  
\end{align*}

In the case $r < l$, we write
\begin{equation*}            
\Ic(t,l,r) = \int_{\Delta(t,l-r)} \Ic \biggl(t - \sum_{j=r+1}^l t_j, r,r \biggr) \  d t_{r+1} \ldots d t_l. 
\end{equation*}
The right-hand side is dominated by 
\begin{equation}      \label{eq:upper_bound_int}
\frac{\beta^{2(r-3)}}{(2(r-1))!} \int_{\Delta(t,l-r)} \biggl(t - \sum_{j = r+1}^l t_j \biggr)^{2(r-1)} \ e^{\beta^2 (t - \sum_{j=r+1}^l t_j)} \ d t_{r+1} \ldots d t_l.  
\end{equation}
If $r = l - 1$, the expression above becomes  
\begin{equation*}
\frac{\beta^{2 (l-4)}}{(2 (l-2))!} \int_0^t  (t - s)^{2(l - 2)} \ e^{\beta^2 (t-s)} \ ds \leq e^{\beta^2 t} \beta^{2(l-4)} \frac{t^{2l-3}}{(2l - 3)!}.
\end{equation*}
If $r < l-1$, we use the change of variables $s = \sum_{j=r+1}^l t_j$, $s_1 = t_{r+2}, s_2 = t_{r+3}, \ldots, s_{l-r-1} = t_l$ to convert the expression in~\eqref{eq:upper_bound_int} into  
\begin{align}      
& \frac{\beta^{2(r-3)}}{(2(r-1))!} \int_0^t \int_{\Delta(s,l-r-1)} (t - s)^{2(r-1)} \ e^{\beta^2 (t - s)} \ d s_1 \ldots \ d s_{l-r-1} \ ds  \notag \\
=& \frac{\beta^{2(r-3)}}{(2(r-1))! (l - r - 1)!} \int_0^t (t - s)^{2(r-1)} \ e^{\beta^2 (t-s)} \ s^{l - r - 1} \ ds. \label{eq:si_4}
\end{align} 
With the identity in~\eqref{eq:integral}, we see that the expression in the second line of~\eqref{eq:si_4} is dominated by
\begin{equation*}
e^{\beta^2 t} \beta^{2(r-3)} \frac{t^{l+r-2}}{(l+r-2)!}.
\end{equation*}
%
%


\subsection{Proof of Lemma~\ref{lm:precise_I_bounds}} 
Recall that $\Jtconst \in (\tfrac{1}{2}, 1)$, $\delta \in (0,1)$, $\nuzero \in (0, \tfrac{\delta}{2})$, $\kappa \in (0, \delta)$, and $\hat \kappa \in (0, \Jtconst)$. To prove~\eqref{190921111055}, let $t$ be so large that $\max\{2 \nuzero t + 1; \kappa t\} < \lfloor \delta t \rfloor$.
For $0 \leq l \leq \nuzero t$ and $1 \leq r \leq l+1$, one has $l+r \leq 2l+1 \leq 2 \nuzero t + 1 < \lfloor \delta t \rfloor < t$.
Thus, 
$$
	\frac{t^{l+r}}{(l+r)!} \leq \frac{t^{\lfloor \delta t \rfloor}}{\lfloor \delta t \rfloor!} \leq \frac{1}{\sqrt{2 \pi \kappa t}} \left(\frac{e}{\kappa} \right)^{\delta t}, 
$$
using Sterling's formula and $\kappa t < \lfloor \delta t \rfloor$. 

To prove~\eqref{190921111101}, note that for $t$ so large that $\hat \kappa t < \lfloor \nu t \rfloor$ and for $\nuzero t < l \leq \nu t$, 
$$
\frac{t^l}{l!} \leq \frac{t^{\lfloor \nu t \rfloor}}{\lfloor \nu t \rfloor !} \leq \frac{1}{\sqrt{2 \pi \hat \kappa t}} \left(\frac{e}{\hat \kappa} \right)^{\nu t}. 
$$

\subsection{Proof of Lemma~\ref{lm:A_estimate}} 
For $t > 0$, $l \in \N$, and $1 \leq r \leq l$, define 
$$
\rho^t_{l,r}(t_1, \ldots, t_r) := t^{-r} \prod_{j=0}^{r-1} (l-j) \biggl(1 - \frac{1}{t} \sum_{j=1}^r t_j \biggr)^{l-r}, \quad (t_1, \ldots, t_r) \in \R^r_+, 
$$
and notice that 
\begin{equation}     \label{190921092451}
A(t,l,r) = \int_{\Delta(t,r)} \prod_{j=1}^{r-2} \left(e^{\beta^2 t_j} - 1 \right) e^{\beta^2 (t_{r-1} + t_r)} \rho^t_{l,r}(t_1, \ldots, t_r) \ d t_1 \ldots \ d t_r, 
\end{equation}
where $\Delta(t,r)$ was defined in~\eqref{eq:def_simplex}. For $\Jtconst t < l < (2-\Jtconst) t$ and $1 \leq r < \nuone l$, the function $\rho^t_{l,r}$ can be bounded as follows:
\begin{equation*}
\rho^t_{l,r}(t_1, \ldots, t_r) 
	\leq \left(\frac{l}{t} \right)^r \prod_{j=1}^{r} e^{-\frac{l- r}{t} t_j}
	\leq (2-\Jtconst)^r \prod_{j=1}^r e^{-\Jtconst (1- \nuone) t_j}.
\end{equation*}
As a result, the integrand in~\eqref{190921092451} can be bounded by the following expression:
$$
	(2-\Jtconst)^r
	\prod_{j=1}^{r-2}
		\left(e^{(\beta^2 - \Jtconst (1-\nuone)) t_j} - e^{-\Jtconst (1-\nuone)t_j} \right)
		e^{(\beta^2 - \Jtconst (1-\nuone)) t_{r-1}}
		e^{(\beta^2 - \Jtconst (1-\nuone)) t_r}.
$$
Then, if we integrate in~\eqref{190921092451} over the larger domain $\R^r_+$, we can rewrite the integral as a product of integrals over $\R^r_+$.
Therefore, we can bound $A (t, l, r)$ by the following expression: 
$$
	(2-\Jtconst)^r
			\left(
		\int_0^{\infty} \left(e^{(\beta^2 - \Jtconst (1-\nuone)) s} - e^{-\Jtconst (1-\nuone) s} \right) \, ds \right)^{r-2}
		\left( \int_0^{\infty} e^{(\beta^2 - \Jtconst (1-\nuone)) s} \, ds \right)^2,
$$
which is less than a constant times $(2-\Jtconst)^r \left(\frac{\beta^2}{(1-\nuone) (2-\Jtconst) ((2-\Jtconst) (1-\nuone) - \beta^2)}\right)^r = \psi(\beta)^r$.

%
%

\subsection{Proof of Lemma~\ref{lm:tail_exponential}}
Using the tail estimate
$$
\sum_{n=k}^{\infty} \frac{t^n}{n!} \leq \frac{t^k}{k!} \sum_{n=k}^{\infty} \left(\frac{t}{k} \right)^{n-k} = \frac{t^k}{k!} \frac{1}{1 - \frac{t}{k}}, \quad k > t 
$$
and Stirling's formula, we have 
\begin{equation*}
e^{(\lambda - 1) t} \sum_{n = f(t)}^{\infty} \frac{t^n}{n!} \lesssim t^{-\frac{1}{2}} e^{(\lambda - 1) t} \left(\frac{e}{\rho_1} \right)^{\rho_2 t}.
\end{equation*}  
From this we obtain the desired convergence for $\lambda < 1 - \rho_2 (1-\ln(\rho_1))$.

\bigskip 

\bigskip

\section{Proof of Lemma~\ref{lm:building_blocks}}
\label{ssec:building_blocks_proofs}

Let $\sublin \in (\tfrac{3}{4},1)$, $\sigma \in (0, \tfrac{1}{2} (1-\sublin))$, $\mu \in (-\infty,-1)$, $\Jtconst \in (\tfrac{1}{2},1)$, and $\nuone \in (\Jtconst^{-1} - 1,1)$.
Let $\chi_1(t)$ be the smallest even integer $\geq t (1-t^{-\sigma})$, and let $\chi_2(t)$ be the largest odd integer $\leq t (1+t^{-\sigma})$. Recall that 
$$
J(t) = \left\{ n \in \N: \left \lvert \frac{n}{t} - 1 \right \rvert < 1 - \Jtconst \right\}, 
$$
and set
$$
K(t) := \left\{l \in \N: \chi_1(t) \leq l \leq \chi_2(t) \right\}. 
$$
We will derive Lemma~\ref{lm:building_blocks} from the following proposition.

\begin{proposition}     \label{prop:key}
For $\beta > 0$ sufficiently small, the following statements hold. 
\begin{enumerate}
\item[\rm(C1)] There is $\tilde \rho > 0$ such that for every $\rho \in (0,\tilde \rho]$, $c \geq 0$, 
$$
\lim_{t \to \infty} e^{(\rho -1) t} \sum_{l \notin J(t)} e^{c t^{\mu} l} \frac{t^l}{l!} \sum_{1 \leq r \leq l+1} r \alpha^r A(t,l,r) = 0. 
$$
\item[\rm(C2)] There is $\tilde \rho > 0$ such that for every $\rho \in (0, \tilde \rho]$,
$$
\lim_{t \to \infty} e^{(\rho - 1) t} \sum_{l \in J(t)} \frac{t^l}{l!} \sum_{\nuone l \leq r \leq l+1} r \alpha^r A(t,l,r) = 0. 
$$
\item[\rm(C3)] We have  
$$
\lim_{t \to \infty}  e^{t^{\sublin}} e^{-t} \sum_{l \in J(t) \setminus K(t)}  \frac{t^l}{l!} \sum_{1 \leq r \leq l+1} r \alpha^r A(t,l,r) = 0.
$$
\end{enumerate}
\end{proposition}  

%
%

\begin{proof}[Proof of Proposition~\ref{prop:key}]
 Let us first show~(C1). Fix $\delta \in (0,1)$, $\kappa \in \left(e^{1-\frac{1}{\delta}}, \delta \right)$, $\hat \kappa \in \left(e^{1 - \frac{1}{\Jtconst}}, \Jtconst \right)$, and $\nuzero \in (0,\tfrac{\delta}{2})$. Let $\beta > 0$ be so small that 
$$
e^{1-\beta^2} > \left(\frac{e}{\kappa} \right)^{\delta} \vee \left(\frac{e}{\hat \kappa} \right)^{\Jtconst}, \quad \beta^2 < \nuzero \alpha^{-1}. 
$$
Then, let $\tilde \rho > 0$ be so small that 
$$
e^{1 - \beta^2 - \tilde \rho} > \left(\frac{e}{\kappa} \right)^{\delta} \vee \left(\frac{e}{\hat \kappa} \right)^{\Jtconst}. 
$$
For fixed $\rho \in (0, \tilde \rho]$ and $c \geq 0$, we decompose 
$$
e^{(\rho-1) t} \sum_{l \notin J(t)} e^{c t^{\mu} l} \frac{t^l}{l!} \sum_{1 \leq r \leq l+1} r \alpha^r A(t,l,r) 
$$
into 
\begin{equation}   \label{eq:three_sums}      
\sum_{0 \leq l \leq \nuzero t} Y_l(t) + \sum_{\nuzero t < l \leq \Jtconst t} Y_l(t) + \sum_{l \geq (2-\Jtconst) t} Y_l(t), 
\end{equation}
where
$$
Y_l(t) = e^{(\rho-1)t} \frac{t^l}{l!} e^{c t^{\mu} l} \sum_{1 \leq r \leq l+1} r \alpha^r A(t,l,r). 
$$
For $0 \leq l \leq \Jtconst t$, 
$$
e^{c t^{\mu} l} \leq e^{c \Jtconst t^{\mu + 1}}. 
$$
As $\mu + 1 < 0$, $\exp(c \Jtconst t^{\mu +1})$ stays bounded, so we disregard the factors $\exp(c t^{\mu} l)$ in the first two sums in~\eqref{eq:three_sums}. 
Lemma~\ref{lm:expected_value} implies for $t \geq \beta^{-2} \vee 2$ 
\begin{equation}     \label{eq:Y_estimate}
Y_l(t) \lesssim \beta^{-2} e^{-(1-\beta^2- \rho)t} (l+1)^2 e^{c t^{\mu} l} \sum_{1 \leq r \leq l+1} r (\alpha \beta^2)^r \frac{t^{l+r}}{(l+r)!}. 
\end{equation}
Then 
\begin{equation*}
Y_l(t) \lesssim e^{\lfloor \delta t \rfloor} (t+1)^5 \frac{t^{\lfloor \delta t \rfloor}}{\sqrt{2 \pi \lfloor \delta t \rfloor} \lfloor \delta t \rfloor^{\lfloor \delta t \rfloor}} \leq \frac{(t+1)^3}{\sqrt{2 \pi \kappa t}} \left(\frac{e}{\kappa} \right)^{\delta t}.
\end{equation*} 
Since  
\begin{equation*}
e^{1 - \beta^2 - \rho} > \left( \frac{e}{\kappa} \right)^{\delta},
\end{equation*}
we have 
$$ 
\lim_{t \to \infty}  \sum_{0 \leq l \leq \nuzero t} Y_l(t) = 0. 
$$ 
Next, observe that for $l > \nuzero t$ and $1 \leq r \leq l+1$, 
\begin{equation}    \label{eq:t_l_less_nu_0}  
\frac{t^{r}}{(l+1) \ldots (l+r)} \leq \left(\frac{t}{l} \right)^r \leq \nuzero^{-r}.  
\end{equation} 
From this,~\eqref{eq:Y_estimate}, and $\beta^2 < \nuzero \alpha^{-1}$, we infer the estimate 
$$ 
\sum_{\nuzero t < l \leq \Jtconst t} Y_l(t) \lesssim \beta^{-2} e^{-(1 - \beta^2 - \rho) t} \sum_{\nuzero t < l \leq \Jtconst t} (l+1)^2 \frac{t^l}{l!}. 
$$ 
As 
$$
e^{1 - \beta^2 - \rho} > \left(\frac{e}{\hat \kappa} \right)^{\Jtconst}, 
$$
we have similarly, again with the help of Stirling's formula, 
$$ 
\lim_{t \to \infty} \sum_{\nuzero t < l \leq \Jtconst t} Y_l(t) = 0. 
$$ 
From~\eqref{eq:t_l_less_nu_0} and $\beta^2 < \nuzero \alpha^{-1}$, we may also infer  
\begin{align*}
\sum_{l \geq (2-\Jtconst) t} Y_l(t) \lesssim& \ \beta^{-2} e^{-(1 - \beta^2 - \rho) t} \sum_{l \geq (2-\Jtconst) t} (l+1)^2 \frac{(t e^{c t^{\mu}})^l}{l!} \\
\lesssim& \ \beta^{-2} t^2 e^{-(1 - \beta^2 - \rho) t} \sum_{l \geq (2-\Jtconst) t-2} \frac{(t e^{c t^{\mu}})^l}{l!}.  
\end{align*}
Let $\rho_2 = 2-\Jtconst$ and $\rho_1 \in \left(e^{1 - \frac{1}{\rho_2}}, \rho_2 \right)$. Then, 
$$
\rho_1 t \exp(c t^{\mu}) < \lfloor (2-\Jtconst) t - 2 \rfloor < \rho_2 t \exp(c t^{\mu})
$$
for $t$ sufficiently large, and Lemma~\ref{lm:tail_exponential} gives 
$$
\lim_{t \to \infty} \sum_{l \geq (2-\Jtconst) t} Y_l(t) = 0. 
$$

Now, we show (C2). Notice that $f(t) = \lfloor \Jtconst (1+\nu_1) t \rfloor$, $\rho_2 = \Jtconst (1+\nu_1)$, and $\rho_1 \in \left(e^{1-\frac{1}{\rho_2}}, \rho_2 \right)$ satisfy the conditions in Lemma~\ref{lm:tail_exponential}. Therefore, we can choose $\lambda > 0$ so small that 
\begin{equation}     \label{eq:tail_convergence}
\lim_{t \to \infty} e^{(\lambda -1) t} \sum_{n = f(t)}^{\infty} \frac{t^n}{n!} = 0. 
\end{equation}
Let $\beta > 0$ be so small that $\beta^2 < \lambda \wedge \alpha^{-1}$, and let $\tilde \rho \in (0,(\lambda \wedge \alpha^{-1}) - \beta^2)$. Fix $\rho \in (0,\tilde \rho]$.  
Lemma~\ref{lm:expected_value} implies for $\beta^2 < 1$ and $t \geq \beta^{-2} \vee 2$
\begin{align}      
&e^{(\rho - 1) t} \sum_{l \in J(t)} \frac{t^l}{l!} \sum_{\nu_1 l \leq r \leq l+1} r \alpha^r A(t,l,r) \notag \\
\lesssim& \beta^{-2} e^{-(1-\beta^2 - \rho)t} \sum_{l \in J(t)} (l+1)^2 \sum_{\nu_1 l \leq r \leq l+1} r (\alpha \beta^2)^r \frac{t^{l+r}}{(l+r)!}.   \label{eq:bb_3_1} 
\end{align}
For $l > \Jtconst t$ and $\nu_1 l \leq r \leq l+1$, we have 
$$
l+r \geq (1+\nu_1) l > (1+\nu_1) \Jtconst t > t, 
$$
where we used that $\nu_1 > \Jtconst^{-1} - 1$. Hence,   
\begin{equation}    \label{eq:t_l_r_estim}
\frac{t^{l+r}}{(l+r)!} \leq \frac{t^{\lceil (1+\nu_1) l \rceil}}{\lceil (1+\nu_1) l \rceil !}.
\end{equation}
Since $\beta^2 < \alpha^{-1}$, the expression in the second line of~\eqref{eq:bb_3_1} is thus less than a constant times 
$$
\beta^{-2} e^{-(1 - \beta^2 - \rho) t}  \sum_{l \in J(t)} (l+1)^2 \frac{t^{\lceil (1+\nu_1) l \rceil}}{\lceil (1+\nu_1) l \rceil!} \lesssim \beta^{-2} e^{-(1- \beta^2 - \rho) t} t^{2}  \sum_{l > \Jtconst (1+\nu_1) t} \frac{t^{l}}{l!}.
$$
Since $\beta^2 + \rho < \lambda$, the right side tends to $0$ on account of~\eqref{eq:tail_convergence}. 

Finally, we show~(C3). In light of~(C2), it is enough to show 
$$
\lim_{t \to \infty} e^{t^{\sublin}} e^{-t} \sum_{l \in J(t) \setminus K(t)} \frac{t^l}{l!} \sum_{1 \leq r < \nu_1 l} r \alpha^r A(t,l,r) = 0. 
$$
By Lemma~\ref{lm:A_estimate}, we have for 
$\beta$ so small that $\alpha (2-\Jtconst) \psi < 1$ the estimate  
$$
\sum_{l \in J(t) \setminus K(t)} \frac{t^l}{l!} \sum_{1 \leq r < \nu_1 l} r \alpha^r A(t,l,r) \lesssim \sum_{l \in J(t) \setminus K(t)} \frac{t^l}{l!}. 
$$
Fix $\tilde \sigma \in (0,\sigma)$. 
By Stirling's formula, 
\begin{align}   
& e^{t^{\sublin}} e^{-t} \sum_{0 \leq l < \chi_1(t)} \frac{t^l}{l!} \notag \\
\lesssim& e^{t^{\sublin}} e^{-t} \sum_{0 \leq l \leq (1-t^{- \tilde \sigma})t} \frac{t^l}{l!} \lesssim t e^{t^{\sublin} - t} \frac{e^{\lceil (1-t^{-\tilde \sigma})t \rceil} t^{\lceil (1-t^{-\tilde \sigma})t \rceil}}{\lceil (1-t^{-\tilde \sigma})t \rceil^{\lceil (1-t^{-\tilde \sigma})t \rceil}}. \label{eq:stanarg_7} 
\end{align}
If we set  
\begin{equation*}
r(t) = \frac{1}{2} \left(e^{1-\frac{1}{1 -t^{-\tilde \sigma}}} +  \left(1 - t^{-\tilde \sigma} \right) \right), 
\end{equation*}
we have 
\begin{equation*}
e^{1-\frac{1}{1-t^{-\tilde \sigma}}} < r(t) < 1 - t^{-\tilde \sigma}.
\end{equation*}
By L'Hospital's rule, 
\begin{align*}
\lim_{t \to \infty} t \left(1 - t^{-\tilde \sigma} - r(t) \right) =& \frac{\tilde \sigma}{2} \lim_{t \to \infty} t^{1-\tilde \sigma} \left( \frac{e^{1-\frac{1}{1-t^{-\tilde \sigma}}}}{(1-t^{-\tilde \sigma})^2} - 1 \right) \\
=& \frac{\tilde \sigma^2}{2(\tilde \sigma-1)} \lim_{t \to \infty} t^{1-2 \tilde \sigma} \frac{e^{1-\frac{1}{1-t^{-\tilde \sigma}}}}{(1-t^{-\tilde \sigma})^4} \left(2 t^{-\tilde \sigma} -1 \right) = \infty,
\end{align*}
where we used that $\tilde \sigma < \tfrac{1}{2}$. Then, for $t$ sufficiently large, 
\begin{equation*}
\lceil (1-t^{-\tilde \sigma})t \rceil \geq (1-t^{-\tilde \sigma})t - \frac{t}{2} \left(1 - t^{-\tilde \sigma} - r(t) \right) = \frac{t}{2} \left(1-t^{-\tilde \sigma} + r(t) \right) > t r(t). 
\end{equation*}
Thus, the right side of~\eqref{eq:stanarg_7} is less than a constant times  
\begin{equation*}
t e^{t^{\sublin} - t} \left(\frac{e}{r(t)} \right)^{(1-t^{- \tilde \sigma})t}. 
\end{equation*}
It remains to check that the term above converges to $0$ as $t \to \infty$, or equivalently 
\begin{align*}
& \lim_{t \to \infty} \ln \biggl(t  e^{t^{\sublin} - t} \left(\frac{e}{r(t)} \right)^{(1-t^{-\tilde \sigma})t} \biggr) \\
=& \lim_{t \to \infty} \left(\ln(t) - t \left(t^{-\tilde \sigma} - t^{\sublin -1} + (1-t^{-\tilde \sigma}) \ln(r(t)) \right) \right) = -\infty.
\end{align*}
To simplify notation, we set 
\begin{equation*} 
\iota(t) = \frac{1}{1-t^{-\tilde \sigma}}, \quad e(t) = e^{1-\iota(t)}.
\end{equation*}
For fixed $\varphi \in (\sublin,1-2\tilde \sigma)$, L'Hospital's rule implies 
\begin{align*}
& \lim_{t \to \infty} t^{-\varphi} t \left(t^{-\tilde \sigma} - t^{\sublin - 1} + (1-t^{-\tilde \sigma}) \ln(r(t)) \right) \\
=& \frac{\tilde \sigma}{\varphi-1} \lim_{t \to \infty} t^{1-\tilde \sigma-\varphi} \biggl( \frac{1-t^{-\tilde \sigma}+e(t) \iota(t)^2 (1-t^{-\tilde \sigma})}{1-t^{-\tilde \sigma}+e(t)} - 1 +  \frac{1-\sublin}{\tilde \sigma} t^{\tilde \sigma + \sublin-1} + \ln(r(t))  \biggr)  \\
=& \frac{\tilde \sigma^2}{(\varphi-1)(\tilde \sigma+\varphi-1)} \lim_{t \to \infty} t^{1-2\tilde \sigma-\varphi} \biggl( \frac{2 + 2e(t) \iota(t)^2}{1-t^{-\tilde \sigma} +e(t)} + \frac{1-t^{-\tilde \sigma}}{(1-t^{-\tilde \sigma} + e(t))^2} \\
& \cdot \left( \left(1-t^{-\tilde \sigma}+e(t) \right) \left( e(t) \iota(t)^4 -2 e(t) \iota(t)^3 \right) - \left(1 + e(t) \iota(t)^2 \right)^2  \right) \\
& + \frac{1-\sublin}{\tilde \sigma^2} (\tilde \sigma+\sublin-1) t^{2\tilde \sigma + \sublin -1} \biggr) = \frac{\tilde \sigma^2}{2 (\varphi-1)(\tilde \sigma+\varphi-1)} \lim_{t \to \infty} t^{1-2\tilde \sigma-\varphi} = \infty.
\end{align*}
If $t$ is so large that 
\begin{equation*}
t^{-\varphi} t \left(t^{-\tilde \sigma}- t^{\sublin-1} + (1-t^{-\tilde \sigma}) \ln(r(t)) \right) \geq 1, 
\end{equation*}
we have 
\begin{equation*}
\ln(t) - t \left(t^{-\tilde \sigma} - t^{\sublin-1} + (1-t^{-\tilde \sigma}) \ln(r(t)) \right) \leq  \left(\ln(t)-t^{\varphi} \right), 
\end{equation*}
which tends to $-\infty$ as $t \to \infty$. Stirling's formula and the tail estimate 
$$ 
\sum_{n=k}^{\infty} \frac{t^n}{n!} \leq \frac{t^k}{k!} \frac{1}{1-\frac{t}{k}}, \quad k > t,
$$ 
yield in addition 
\begin{align*}
e^{t^{\sublin}} e^{-t} \sum_{l > \chi_2(t)} \frac{t^l}{l!} \lesssim e^{t^{\sublin}} e^{-t} \sum_{l \geq (1+t^{-\tilde \sigma}) t} \frac{t^l}{l!} \lesssim& e^{t^{\sublin}-t} \frac{e^{\lceil (1+t^{-\tilde \sigma})t \rceil} t^{\lceil (1+t^{-\tilde \sigma})t \rceil}}{((1+t^{-\tilde \sigma})t)^{\lceil (1+t^{-\tilde \sigma})t \rceil}} \cdot \frac{1}{1 - \frac{1}{1+t^{-\tilde \sigma}}} \\
\lesssim& t^{\tilde \sigma} e^{t^{\sublin}-t} \left(\frac{e}{1+t^{-\tilde \sigma}} \right)^{(1+t^{-\tilde \sigma})t}. 
\end{align*} 
To complete the proof of~(C3), let us now show that 
$$ 
\lim_{t \to \infty} \ln \biggl( t^{\tilde \sigma} e^{t^{\sublin}-t} \left(\frac{e}{1+t^{-\tilde \sigma}} \right)^{(1+t^{-\tilde \sigma})t} \biggr) = -\infty. 
$$
For $\varphi \in (\sublin,1-2 \tilde \sigma)$, we have by virtue of L'Hospital's rule 
\begin{align*}
& \lim_{t \to \infty} t^{-\varphi} t \left((1+t^{-\tilde \sigma}) \ln(1+t^{-\tilde \sigma}) - t^{-\tilde \sigma} - t^{\sublin-1} \right) \\
=& \frac{\tilde \sigma}{1-\varphi} \lim_{t \to \infty} \left(t^{1-\tilde \sigma-\varphi} \ln(1+t^{-\tilde \sigma}) + \frac{\sublin-1}{\tilde \sigma} t^{\sublin-\varphi} \right) \\
=& \frac{\tilde \sigma}{1-\varphi} \lim_{t \to \infty} t^{1-\tilde \sigma-\varphi} \ln(1+t^{-\tilde \sigma}) \\
=& \frac{\tilde \sigma^2}{(\varphi-1) (\tilde \sigma+\varphi-1)} \lim_{t \to \infty} \frac{t^{1-2 \tilde \sigma-\varphi}}{1+t^{-\tilde \sigma}} = \infty. 
\end{align*}
In particular, we can choose $t$ so large that 
$$ 
t^{-\varphi} t \left((1+t^{-\tilde \sigma}) \ln(1+t^{-\tilde \sigma}) - t^{-\tilde \sigma} - t^{\sublin -1} \right) \geq \tilde \sigma, 
$$
in which case 
$$
\tilde \sigma \ln(t) - t ( (1+t^{-\tilde \sigma}) \ln(1+t^{-\tilde \sigma}) - t^{-\tilde \sigma} - t^{\sublin-1}) 
\leq \tilde \sigma (\ln(t) - t^{\varphi}) \to -\infty
$$
as $t \to \infty$. 
\end{proof}

\begin{proof}[Proof of Lemma~\ref{lm:building_blocks}]

(A0) is an immediate consequence of (A1) since 
$$
\sum_{0 \leq r < \nuone l-1} (r+1) \alpha^r A(t-2t^{\xi_1},l,r+1) < 1.
$$
We now prove (A1). We split the proof into two parts. First, we show that 
\begin{align}    
\limsup_{t \to \infty} e^{-\beta^2 t^{\xi_1}} & \sup_{\|y\| \leq t^{\sublin}} \frac{1}{p_t^y} \sum_{l \in J(t-2 t^{\xi_1})} q^y_{\iota(y,l)} \Pp^{\bullet}(l) \notag \\
& \sum_{0 \leq r < \nuone l-1} (r+1) \alpha^r A(t-2t^{\xi_1},l,r+1) < \infty.   \label{eq:bb_1_less_than_nu}
\end{align}
For $l \in J(t-2 t^{\xi_1})$ and $0 \leq r < \nuone l-1$, we have by Lemma~\ref{lm:A_estimate} 
$$
A(t-2t^{\xi_1},l,r+1) \lesssim ((2-\Jtconst) \psi(\beta))^r. 
$$
Moreover, 
$$
\frac{\Pp^{\bullet}(l)}{\Pp^{\bullet}(l+1)} = \frac{(t-2t^{\xi_1})^l (l+1)!}{(t-2t^{\xi_1})^{l+1} l!} = \frac{l+1}{t-2t^{\xi_1}}, 
$$
which is bounded in $t$. Hence, 
\begin{align*}
& \sum_{l \in J(t-2 t^{\xi_1})} q^y_{\iota(y,l)} \Pp^{\bullet}(l) \sum_{0 \leq r < \nuone l-1} (r+1) \alpha_d^r A(t-2t^{\xi_1},l,r+1) \\
\lesssim& \sum_{\substack{l \in J(t-2 t^{\xi_1}), \\ l \equiv \|y\|_1}} q^y_l \Pp^{\bullet}(l) + 
\sum_{\substack{l \in J(t-2 t^{\xi_1}), \\ l \equiv \|y\|_1 +1}} q^y_{l +1} \Pp^{\bullet}(l) \lesssim p^y_{t-2t^{\xi_1}}. 
\end{align*}
By Lemma~\ref{lm:p_ratio}, 
$$
\frac{p^y_{t-2t^{\xi_1}}}{p_t^y} \lesssim e^{\beta^2 t^{\xi_1}}
$$
for $t$ sufficiently large. 
This implies~\eqref{eq:bb_1_less_than_nu}.  The statement in~(A1) will then follow from  
\begin{equation} \label{eq:bb_1_greater_than_nu}
\lim_{t \to \infty}  \sup_{\|y\| \leq t^{\sublin}} \frac{1}{p_t^y} \sum_{l \in J(t-2 t^{\xi_1})} q^y_{\iota(y,l)} \Pp^{\bullet}(l) 
\sum_{\nuone l -1 \leq r \leq l} (r+1) \alpha^r A(t-2t^{\xi_1},l,r+1) = 0.  
\end{equation}
Let $\tilde \sublin \in (\sublin,1)$. We have 
\begin{align*}
& \sup_{\|y\| \leq t^{\sublin}} \frac{1}{p_t^y} \sum_{l \in J(t- 2 t^{\xi_1})} q^y_{\iota(y,l)} \Pp^{\bullet}(l) \sum_{\nuone l - 1 \leq r \leq l} (r+1) \alpha^r A(t-2t^{\xi_1}, l, r+1) \\
\lesssim& e^{(t- 2 t^{\xi_1})^{\tilde \sublin}} e^{-(t-2 t^{\xi_1})}  \sum_{l \in J(t-2 t^{\xi_1})} \frac{(t-2 t^{\xi_1})^l}{l!} \sum_{\nuone l \leq r \leq l+1} r \alpha^r A(t-2t^{\xi_1},l,r), 
\end{align*}
and the expression above tends to $0$ as $t \to \infty$ by~(B2) of Proposition~\ref{prop:key}.

To prove~(A2), it is enough to replace $t-2t^{\xi_1}$ with $t^{\xi_1}$ in~\eqref{eq:bb_1_less_than_nu} and to drop the terms $e^{-\beta^2 t^{\xi_1}}$ and $q^y_{\iota(y,l)}/ p_t^y$. With respect to~(A3), replace $t-2t^{\xi_1}$ with $t^{\xi_1}$ in the proof of~\eqref{eq:bb_1_greater_than_nu}, and notice that~(C2) of Proposition~\ref{prop:key} also lets us deal with the additional factors $t^{\theta}$ and $e^{\beta^2 t^{\xi_1}}$ provided that $\beta^2 \leq \tilde \rho$.  

With regard to~(A4)  , we have for $\theta, c > 0$, $\beta^2 < \tilde \rho$, and $\mu \coleq (\sublin - 1)/ \xi_1 < -1$ 
\begin{align*}
& t^{\theta} e^{\beta^2 t^{\xi_1}} \sum_{l \notin J(t^{\xi_1})} e^{c t^{\sublin - 1} l} \Pp^-(l) \sum_{0 \leq r \leq l} (r+1) \alpha^r A(t^{\xi_1}, l, r+1) \\
\lesssim& e^{(\tilde \rho -1) t^{\xi_1}} \sum_{l \notin J(t^{\xi_1})} e^{c (t^{\xi_1})^{\mu} l}  \frac{(t^{\xi_1})^l}{l!} \sum_{1 \leq r \leq l+1} r \alpha^r A(t^{\xi_1},l,r), 
\end{align*}
and the expression above tends to $0$ as $t \to \infty$ by~(C1) of Proposition~\ref{prop:key}. 

Instead of proving (A5), we will show the following stronger statement: 

\begin{equation}   \label{eq:stronger_B4} 
\lim_{t \to \infty} t^{\theta} e^{t^{\sublin}} \sum_{l \notin K(t-2 t^{\xi_1})} \Pp^{\bullet}(l) \sum_{0 \leq r \leq l} (r+1) \alpha^r A(t-2t^{\xi_1}, l, r+1) = 0, \quad \theta > 0. 
\end{equation}

Notice that for $\theta > 0$ and $\tilde \sublin \in (\sublin, 1)$, 
\begin{align*}
& t^{\theta} e^{t^{\sublin}} \sum_{l \notin K(t-2 t^{\xi_1})} \Pp^{\bullet}(l) \sum_{0 \leq r \leq l} (r+1) \alpha^r A(t-2 t^{\xi_1}, l, r+1) \\
\lesssim& e^{(\tilde \rho -1) (t-2 t^{\xi_1})} \sum_{l \notin J(t-2 t^{\xi_1})} \frac{(t- 2 t^{\xi_1})^l}{l!} \sum_{1 \leq r \leq l+1} r \alpha^r A(t- 2 t^{\xi_1}, l,r) \\
&+ e^{(t-2t^{\xi_1})^{\tilde \sublin}}  e^{-(t-2t^{\xi_1})} \sum_{l \in J(t- 2t^{\xi_1}) \setminus K(t- 2 t^{\xi_1})} \frac{(t-2 t^{\xi_1})^l}{l!} \sum_{1 \leq r \leq l+1} r \alpha^r A(t-2 t^{\xi_1}, l,r).
\end{align*}
and the term above tends to $0$ as $t \to \infty$ by~(C1) and~(C3) of Proposition~\ref{prop:key}. 
\end{proof}

\begin{proof}[Proof of Claim~\ref{cl:sum_notin_K}]
This is an immediate consequence of~\eqref{eq:stronger_B4}.
\end{proof}

\section{A Calculus Estimate}
\label{appendix_calc}



In the proof of Lemma~\ref{lm:Z_increment}, we use the following auxiliary result that corresponds to Inequality (3.29) in~\cite{Kifer}.  

\begin{lemma}     \label{lm:lclt_lemma}
There is a constant $c > 0$, depending only on the dimension $d$, such that for any $r \in \N$,     
	\begin{equation}    \label{eq:lclt_lemma}
		\sum_{0 < i_1 < \ldots < i_r < n} i_1^{-\frac{d}{2}} (i_2 - i_1)^{-\frac{d}{2}} \ldots (i_r - i_{r-1})^{-\frac{d}{2}} (n - i_r)^{-\frac{d}{2}} \leq c^r  n^{-\frac{d}{2}}, \quad n \geq r + 1.   
	\end{equation}  
	\end{lemma}

\begin{proof} We prove the statement by induction.
As $d \geq 3$, Jensen's inequality yields   
	\begin{equation*}
		\biggl( \frac{1}{i} + \frac{1}{n-i} \biggr)^{\frac{d}{2}} \leq 2^{\frac{d}{2}-1} \biggl( \frac{1}{i^{\frac{d}{2}}} + \frac{1}{(n-i)^{\frac{d}{2}}} \biggr), \quad1 \leq i < n.  
	\end{equation*}
For $r = 1$, the left side of~\eqref{eq:lclt_lemma} becomes 
	\begin{align*}
		\sum_{i=1}^{n-1} \frac{1}{i^{\frac{d}{2}}} \cdot \frac{1}{(n - i)^{\frac{d}{2}}} 
		&= \frac{1}{n^{\frac{d}{2}}} \sum_{i=1}^{n-1} \biggl( \frac{1}{i} + \frac{1}{n-i} \biggr)^{\frac{d}{2}} \\
		&\leq 2^{\frac{d}{2}-1} \frac{1}{n^{\frac{d}{2}}} \sum_{i=1}^{n-1} \biggl( \frac{1}{i^{\frac{d}{2}}} + \frac{1}{(n-i)^{\frac{d}{2}}} \biggr) \leq c \frac{1}{n^{\frac{d}{2}}},   
	\end{align*}
where $c \coleq 2^{d/2} \sum_{i=1}^{\infty} i^{-d/2}$. In the induction step, assume that~\eqref{eq:lclt_lemma} holds for some $r \in \N$. Then, for $n \geq r + 2$, 
\begin{align*}
		& \sum_{0 < i_1 < \ldots < i_{r+1} < n} i_1^{-\frac{d}{2}} (i_2 - i_1)^{-\frac{d}{2}} \ldots (i_{r+1} - i_r)^{-\frac{d}{2}} (n - i_{r+1})^{-\frac{d}{2}} \\
		=& \sum_{i_{r+1} = r+1}^{n-1} (n - i_{r+1})^{-\frac{d}{2}} \sum_{\substack{i_1, \ldots, i_r : \\ 0 < i_1 < \ldots i_r < i_{r+1}}} i_1^{-\frac{d}{2}} (i_2 - i_1)^{-\frac{d}{2}} \ldots (i_{r+1} - i_r)^{-\frac{d}{2}} \\
		\leq& \ c^r \sum_{i_{r+1} = 1}^{n-1} (n - i_{r+1})^{-\frac{d}{2}} \cdot i_{r+1}^{-\frac{d}{2}}\leq c^{r+1} n^{-\frac{d}{2}}.     
	\end{align*}
\end{proof}


%

\end{appendix}


\bibliographystyle{alpha}

\bibliography{references_v2}


%
%
%

\end{document}